\documentclass[11pt,a4paper,reqno]{amsart}
\usepackage[T1]{fontenc}
\usepackage{a4wide}
\usepackage{comment}
\usepackage{yhmath,mathrsfs,amsthm,amsmath,amssymb,amsfonts,enumerate,lipsum,appendix,mathtools,bm}
\usepackage{bbold}
\usepackage{pdfpages}
\usepackage[mathscr]{euscript}
\usepackage{graphicx}
\usepackage{pgf,tikz}
\usepackage{doi}
\usepackage{tkz-euclide}
\usepackage{caption}
\usepackage{subcaption}
\usetikzlibrary{shapes.geometric}
\usetikzlibrary{arrows,hobby}
\usepackage[shortlabels]{enumitem}
\usepackage[english]{babel}
\usetikzlibrary{shapes.geometric}
\usetikzlibrary{arrows,hobby}
\newtheorem{theorem}{Theorem}[section]
\newtheorem{lemma}[theorem]{Lemma}
\newtheorem{proposition}[theorem]{Proposition}
\newtheorem{corollary}[theorem]{Corollary}
\newtheorem{definition}[theorem]{Definition}
\theoremstyle{definition}

\newtheorem{remark}[theorem]{Remark}
\numberwithin{equation}{section}
\allowdisplaybreaks
\definecolor{bluee}{rgb}{0.09, 0.25, 0.99}
\definecolor{greeen}{rgb}{0.01, 0.75, 0.24}
\usepackage[sort,numbers]{natbib}

\usepackage{hyperref}
\hypersetup{
    colorlinks=true,
    linkcolor=bluee,
    filecolor=black,
    urlcolor=greeen,
    citecolor=greeen,
}
\usepackage{orcidlink} 
\usepackage{cleveref}  
\usepackage[a4paper,margin=0.90in]{geometry}
\usepackage{bigints}
\usepackage{url}
\usepackage{esint}
\usepackage[hyperpageref]{backref}

\DeclarePairedDelimiter\abs{\lvert}{\rvert}
\DeclarePairedDelimiter\norm{\lVert}{\rVert}
\makeatletter
\let\oldnorm\norm
\def\norm{\@ifstar{\oldnorm}{\oldnorm*}}
\makeatother 

\usepackage{fancyhdr}
\newcommand{\al} {\alpha}

\newcommand{\la} {\lambda}

\newcommand{\no} {\nonumber}
\newcommand{\noi} {\noindent}

\newcommand{\ra} {\rightarrow}

\newcommand{\bee} {\begin{equation}}
\newcommand{\eee} {\end{equation}}
\newcommand{\bea} {\begin{eqnarray}}
\newcommand{\eea} {\end{eqnarray}}
\newcommand{\Bea} {\begin{eqnarray*}}
\newcommand{\Eea} {\end{eqnarray*}}

\providecommand{\fint}{\rlap{\kern.15em$-$}\int}
\newcommand{\R}{\mathbb{R}}
\newcommand{\Sph}{\mathbb{S}}
\newcommand{\D}{\mathcal{D}}
\newcommand{\eps}{\varepsilon}
\DeclareMathOperator{\sgn}{sign}
\newcommand{\pLap}{(-\Delta_p)^s}
\newcommand{\Jp}{J_p}

\def\d{\,{ \rm d}}
\def\dx{\,{\rm d}x}
\def\dy{\,{\rm d}y}
\def\dz{\,{\rm d}z}
\def\dxy{\,{ \rm d}x{ \rm d}y}

\def\dv{\,{\rm d}v}

\def\dr{\,{\rm d}r}

\def\Dsp{{{\mathcal D}^{s,p}(\R^N)}}

\def\C{{\mathcal C}}

\def\cc{{\mathcal{C}_c^\infty}}

\makeatletter
\newcommand{\customlabel}[2]{%
  \phantomsection
  \def\@currentlabel{#2}%
  \label{#1}%
}
\makeatother

\usepackage{xpatch}
\makeatletter   
\xpatchcmd{\@tocline}
{\hfil\hbox to\@pnumwidth{\@tocpagenum{#7}}\par}
{\ifnum#1<0\hfill\else\dotfill\fi\hbox to\@pnumwidth{\@tocpagenum{#7}}\par}
{}{}
\def\l@subsection{\@tocline{2}{0pt}{2pc}{6pc}{}}
\def\l@subsubsection{\@tocline{3}{0pt}{8pc}{8pc}{}}
\makeatother
\title[The homogeneity function and the optimal decay estimates]
{The homogeneity function of the fractional $p$-Laplacian and optimal asymptotic exponents for critical Hardy-Sobolev equations}

\author[N. Biswas and D. Mukherjee]{Nirjan Biswas$^1$ \orcidlink{0000-0002-3528-8388} \and Debangana Mukherjee$^2$ \orcidlink{0000-0002-1516-840X}}
\address{\rm  $^1$ Department of Mathematics, Indian Institute of Science Education and Research Pune \\
Dr. Homi Bhabha Road, Maharashtra 411008, India}
\address{\rm $^3$ Krea University, School of Interwoven Arts and Sciences (SIAS), Sricity, India}
\email{nirjaniitm@gmail.com, debangana.mukherjee@krea.edu.in}
\thanks{$^2$Corresponding author}


\date{}

\begin{document}
\begin{abstract}
Let $N \ge 2$, $s\in (0,1)$, and $p \in (1, \infty)$ with $N>sp$. For the fractional $p$-Laplacian, the power function $|x|^{-\gamma}$ with $0<\gamma<\frac{N}{p-1}$ satisfies
$$\pLap\big(|x|^{-\gamma}\big)=\Phi(\gamma)|x|^{-\gamma(p-1)-sp} \, \text{ in }\R^{N}\setminus\{0\}.$$
Starting from the power function formula of Del Pezzo and Quaas (2026), we derive a paired integral representation of $\Phi$, and show that $\Phi$ is strictly increasing and then strictly decreasing, with a unique maximum at $\frac{N-sp}{p}$. This maximum $\bar\mu:=\Phi(\frac{N-sp}{p})$ equals one half of the sharp fractional Hardy constant. Consequently, for every $0<\mu<\bar\mu$, the equation $\Phi(\gamma)=\mu$ has exactly two roots $0<\gamma_{-}(\mu)<\frac{N-sp}{p}<\gamma_{+}(\mu)<\frac{N-sp}{p-1}$. These roots determine the homogeneous solutions $|x|^{-\gamma_{-}(\mu)}$ and $|x|^{-\gamma_{+}(\mu)}$ of the Hardy equation
\begin{equation*}
(-\Delta_p)^s u - \mu \frac{u^{p-1}}{|x|^{sp}}=0
\quad
\text{in }\mathbb R^N\setminus\{0\}, \; u> 0 \text{ in }  \mathbb{R}^N \setminus \{0\}.  
\end{equation*}
For $0<\mu<\bar\mu$, we then study the following critical Hardy-Sobolev equation 
\begin{align}\tag{$\mathcal{P}$}\label{P}
\pLap u-\mu \frac{u^{p-1}}{|x|^{sp}} = u^{\frac{Np}{N-sp}-1} \text{ in } \mathbb{R}^N, \; u > 0 \text{ in }  \mathbb{R}^N.  
\end{align}
Let $p\ge2$ and $U$ be a positive, radially symmetric, nonincreasing finite-energy weak solution to \eqref{P}. We prove that there exists $C>0$ for which
\begin{align*}
U(x)\le C|x|^{-\gamma_-(\mu)}\;\text{as }|x|\to0,\quad \text{ and } \quad U(x)\le C|x|^{-\gamma_+(\mu)} \; \text{as }|x|\to\infty. 
\end{align*}
Further, we show that $\gamma_-(\mu)$ and $\gamma_+(\mu)$ are the optimal exponents describing the singularity at the origin and the decay at infinity respectively. In particular, $U \in L^{q}(\R^{N})$ for $q \in (\tfrac{N}{\gamma_{+}(\mu)},\tfrac{N}{\gamma_{-}(\mu)})$, and this range cannot be enlarged, except possibly at endpoints. Moreover, as $\mu\downarrow0$, we have $\gamma_{-}(\mu)\to0$ and $\gamma_{+}(\mu)\to \frac{N-sp}{p-1}$. The latter coincides with the optimal decay exponent for extremals of the fractional Sobolev inequality due to Brasco, Mosconi, and Squassina (2016). Our approach for the optimal asymptotic estimates is new and based on the two ground-state substitutions, weighted Caccioppoli estimates, finite Moser iterations, and a comparison principle.
\end{abstract}
\maketitle
\tableofcontents
\section{Introduction}\label{sec:intro}
\subsection{Background and motivation}

Hardy potentials introduce a distinguished singular scale into elliptic equations. The behaviour of positive solutions near the singular point, and their decay at infinity, are determined by the homogeneous solutions of the Hardy equation. For the fractional $p$-Laplacian these are determined by a single scalar function.

Throughout the paper we assume
$$
N\ge 2,\quad s\in(0,1),\quad p>1,\; \text{ and } \; sp<N.
$$
The fractional $p$-Laplace operator is defined as 
\begin{equation}\label{eq:op}
(-\Delta_p)^s u(x) :=  \operatorname{P.V.} \int_{\mathbb R^N} \frac{|u(x)-u(y)|^{p-2}(u(x)-u(y))}{|x-y|^{N+sp}}\dy.
\end{equation}
Here $J_p(t):=|t|^{p-2}t$ for $t \in \R$, and and P.V. denotes the principal value. Throughout the paper, we use the normalization constant $C_{N,p,s}=1$.
The action of this operator on power functions $|x|^{-\gamma}$ is given by
\begin{equation}\label{eq:homog}
(-\Delta_p)^s\bigl(|x|^{-\gamma}\bigr) = \Phi(\gamma)\,|x|^{-\gamma(p-1)-sp} \text{ in }\mathbb R^N\setminus\{0\}.
\end{equation}
The coefficient $\Phi$ is one of the main objectives in this paper. In \cite{DPQ25}, Del Pezzo and Quaas computed this $\Phi$ to determine its zeros and sign. We study the behaviour of $\Phi$ between its zeros.

We use
\begin{equation}\label{eq:gammas}
\gamma_c:=\frac{N-sp}{p-1}, \quad \gamma^*:=\frac{N-sp}{p}, \quad p_s^*:=\frac{Np}{N-sp},
\end{equation}
so that $0<\gamma^*<\gamma_c$ and $\gamma^*=N/p_s^*$. The coefficient $\Phi$ has the following paired representation:
\begin{equation}\label{eq:intro-paired}
\Phi(\gamma) = \int_0^1 (1-r^\gamma)^{p-1} \bigl(1-r^{\beta(\gamma)}\bigr) r^{sp-1}\Psi(r)\dr,
\quad
\beta(\gamma):=N-sp-(p-1)\gamma,
\end{equation}
with $\Psi$ as given in \eqref{eq:Psi}. It gives the sign of $\Phi'$, hence the shape of
$\Phi$, and identifies its maximum
\begin{equation}\label{eq:intro-mubar}
\bar\mu := \Phi(\gamma^*),
\end{equation}
with one half of the sharp constant $C_H$ in the fractional Hardy inequality \cite{FS08} (see \cite{Herbst} for $p=2$, $s=\frac12$, and \cite{FT,LLM07} for local counterparts),
\begin{align}\label{intro:fractionalHardy}
\iint_{\R^{N}\times\R^{N}}\frac{|u(x)-u(y)|^{p}}{|x-y|^{N+sp}}\dxy \ge C_H(N,p,s) \int_{\R^{N}}\frac{|u(x)|^{p}}{|x|^{sp}}\dx \quad \forall \, u \in \cc(\R^{N}\setminus\{0\}).
\end{align}
We consider the following Euler Lagrange equation for \eqref{intro:fractionalHardy}:
\begin{equation}\label{eq:intro-hardy}
(-\Delta_p)^s u = \mu \frac{u^{p-1}}{|x|^{sp}}
\quad
\text{in }\mathbb R^N\setminus\{0\}, \; u > 0. \; 
\end{equation}
Solutions of (\ref{eq:intro-hardy}) are understood in the weak sense on $\mathbb R^N\setminus \{0\}$, as defined in Section~\ref{sec:rep}.
We have observed that for $0<\mu<\bar\mu$ the equation $\Phi(\gamma)=\mu$ has exactly two roots
\begin{equation}\label{eq:intro-gammapm}
0<\gamma_-(\mu)
<
\gamma^*
<
\gamma_+(\mu)
<
\gamma_c .
\end{equation}
Thus \eqref{eq:intro-hardy} admits the two positive homogeneous solutions $|x|^{-\gamma_-(\mu)}$ and
$|x|^{-\gamma_+(\mu)}$ in $\mathbb R^N\setminus \{0\}$. Neither of these solutions belongs to $D^{s,p}(\mathbb R^N)$ but both are weak solutions in the punctured space. The two characteristic exponents merge at the Hardy threshold as $\mu\uparrow\bar\mu$.

We also study the homogeneity function of the linearized equation. Its two exponents are $\gamma$ and
$N-sp-(p-1)\gamma$. Unlike the nonlinear one it is reflection-symmetric for every $p>1$. The contrast
appears in the roots. For $p=2$,
\begin{equation}\label{eq:intro-reflection}
\gamma_-(\mu)+\gamma_+(\mu)=\gamma_c ,
\end{equation}
whereas for $p\ne2$ this fails, since
\begin{equation}\label{eq:intro-reflection-failure}
\lim_{\mu\uparrow\bar\mu}
\left(
\gamma_-(\mu)+\gamma_+(\mu)-\gamma_c
\right)
=
\frac{(N-sp)(p-2)}{p(p-1)}.
\end{equation}

As an application, we study the behaviour of positive radially symmetric nonincreasing weak
solutions to the following critical problem driven by the fractional $p$-Laplace Hardy Sobolev operator: 
\begin{equation}\label{eq:P}
(-\Delta_p)^s u-\mu \frac{u^{p-1}}{|x|^{sp}} = u^{p_s^*-1} \text{ in }\mathbb R^N, \; u>0, \; u\in \mathcal{D}^{s,p}(\mathbb R^N).
\end{equation}
For $0 \le \mu <\bar{\mu}=\frac{C_H}{2}$, we consider the following quantity: 
\begin{align}\label{best-constant}
    S_{\mu} \coloneqq  \inf_{u \in  \Dsp \setminus \{0\}} \frac{\displaystyle   \iint_{\R^{2N}} \frac{|u(x)-u(y)|^p}{|x-y|^{N+sp}}\dxy - 2\mu  \int_{\R^N} \frac{\abs{u(x)}^p}{\abs{x}^{sp}}\dx }{\displaystyle \left(\int_{\R^N} \abs{u(x)}^{p^*_s} \dx\right)^{\tfrac{p}{p_s^*}}}.
\end{align}
In \cite[Theorem 1.1]{Shen24}, Shen proved that $S_{\mu}>0$ is attained by a non-negative extremal which is positive, radially symmetric and nonincreasing with respect to the origin. These extremals (up to a multiplication of a normalized constant) weakly satisfy \eqref{eq:P}. In this  paper, we establish the asymptotic decay for these extremals.
Since
$p_s^*-1=(p-1)+(p_s^*-p)$, equation \eqref{eq:P} reads
\begin{equation}\label{eq:intro-rewrite-critical}
(-\Delta_p)^s u = \left(\mu+|x|^{sp}u^{p_s^*-p}\right) |x|^{-sp}u^{p-1} \text{ in }\mathbb R^N.
\end{equation}
When $|x|^{sp}u^{p_s^*-p}$ is small, \eqref{eq:intro-rewrite-critical} is a perturbation of \eqref{eq:intro-hardy}. 
The Hardy term $\mu|x|^{-sp}u^{p-1}$ has the same scaling as $\pLap u$. So it is not a small perturbation. It determines the behaviour of $u$ near the origin and at infinity. By \eqref{eq:intro-rewrite-critical}, the term $|x|^{sp}u^{p^*_s-p}$ is small near the origin and near infinity. 
So the two homogeneous profiles should give the correct powers at the two ends. The lower root governs the behaviour at the origin, while the upper root governs the decay at infinity. 
The power solutions of \eqref{eq:intro-hardy} are $|x|^{-\gamma_-(\mu)}$ and $|x|^{-\gamma_+(\mu)}$. We show that, near the origin and near infinity, the extrema for $S_{\mu}$ behave like these solutions of the Hardy equation \eqref{eq:intro-hardy}.

For this purpose, we prove a comparison principle for
\begin{equation}\label{eq:intro-comparison-equation}
(-\Delta_p)^s w=\nu h(x)w^{p-1} \text{ in }\mathbb R^N, \; h\ge0.
\end{equation}
The comparison principle removes the loss $\eta$ once the preliminary bounds
\begin{equation}\label{eq:intro-M0}
u(x)\le C_\eta |x|^{-\gamma_-(\mu)-\eta} \quad \text{near }0,
\end{equation}
and
\begin{equation}\label{eq:intro-Minf}
u(x) \le C_\eta |x|^{-\gamma_+(\mu)+\eta} \quad \text{near infinity}
\end{equation}
are known for every $\eta>0$. It then yields
\begin{equation}\label{eq:intro-sharp0}
u(x) \le C|x|^{-\gamma_-(\mu)}, \quad \text{as }|x|\to0,
\end{equation}
and
\begin{equation}\label{eq:intro-sharpinf}
u(x) \le C|x|^{-\gamma_+(\mu)}, \quad \text{as }|x|\to\infty.
\end{equation}
As a consequence, any positive radially symmetric nonincreasing weak solution to \eqref{eq:P} satisfies \eqref{eq:intro-M0} and \eqref{eq:intro-Minf}, and hence satisfies \eqref{eq:intro-sharp0} and \eqref{eq:intro-sharpinf} respectively.

\subsection{Previous works}
The sharp fractional Hardy inequality and its nonlinear ground-state representation were established by Frank and Seiringer \cite{FS08}. Their result identifies the sharp variational constant. Here the homogeneity coefficient is studied as a function of $\gamma$, and its unique maximum is shown to recover that constant.

For the local $p$-Laplacian with Hardy potential, characteristic exponents and sharp estimates have been studied in several works, in particular by Xiang \cite{Xia15}. Positive finite-energy solutions of the corresponding local critical equation were shown to be radial by Oliva, Sciunzi and Vaira \cite{OSV20}. The local radial problem has an additional ordinary differential equation structure which is unavailable for the fractional $p$-Laplacian.

{The homogeneous fractional Sobolev space $\Dsp$ is the closure of $\C_{c}^{\infty}(\R^N)$ with respect to the Gagliardo seminorm
\begin{equation}\label{eq:semin}
  [u]_{s,p} := \left(\iint_{\R^{2N}} \frac{|u(x)-u(y)|^p}{|x-y|^{N+sp}}\dxy\right)^{\frac{1}{p}}.
\end{equation}}
Brasco, Mosconi and Squassina \cite{BMS16} studied the fractional Sobolev equation without Hardy potential. They obtained optimal decay estimates for the fractional Sobolev extremals, which are related to $(-\Delta_p)^s u = \vert{}u\vert{}^{p^*_s-2}u$ in $\mathbb{R}^N$. 
More precisely, they have shown that any minimizer $U$ for 
\begin{align*}
    S:= \inf_{u \in \Dsp \setminus \{0\}} \frac{[u]_{s,p}^{p}}{\left( \displaystyle \int_{\R^N} |u(x)|^{p^*_s} \dx \right)^{\frac{p}{p^*_s}}}, 
\end{align*}
is of constant sign, radially nonincreasing, and satisfies 
\begin{align}\label{sharp-brasco}
    \lim_{|x| \ra \infty} U(x) |x|^{\frac{N-sp}{p-1}} = U_{\infty}, \text{ where } U_{\infty} \neq 0, 
\end{align}
and the exponent $\frac{N-sp}{p-1}$ is sharp.  From now onward, we denote $\gamma_c:= \frac{N-sp}{p-1}$.
Later, Marano and Mosconi in \cite{MM19} determined sharp asymptotic profiles for the fractional Hardy-Sobolev optimizers, which are related to the following weighted problem $$(-\Delta_p)^s u =  \frac{\vert{}u\vert{}^{p^*_s(\alpha)-2}u}{|x|^{\alpha}} \text{ in } \mathbb{R}^N, \text{ with } \;p^*_s(\al) := \frac{p(N-\al)}{N-sp} \text{ and } \alpha \in [0,sp).$$ More precisely, they have shown that any minimizer $U$ for 
\begin{align*}
    S_{\alpha}:= \inf_{u \in \Dsp \setminus \{0\}} \frac{[u]_{s,p}^{p}}{\left( \displaystyle \int_{\R^N} \frac{|u(x)|^{p^*_s(\alpha)}}{|x|^{\alpha}}\dx \right)^{\frac{p}{p^*_s(\alpha)}}}, 
\end{align*}
is of constant sign, radially nonincreasing, and fulfills
\begin{align*}
    \frac{1}{C|x|^{\gamma_c}} \le U(x) \le \frac{C}{|x|^{\gamma_c}}, \quad |x| \ge 1, 
\end{align*}
for some positive constant $C=C(N,p,s, \alpha, u)$.

For $p=2$, precise asymptotic information for fractional equations with Hardy-type potentials can also be obtained by linear methods. These include extension techniques, spectral decompositions and Almgren-type monotonicity formulas, see the work of Fall and Felli \cite{FF14}, Mainini, Mukherjee, and Ognibene \cite{MMO26}. Such tools are not directly available for the nonlinear fractional $p$-Laplacian. Our approach instead relies on the homogeneity function, Picone comparison, ground-state substitutions and weighted iteration. For the local regularity theory used throughout, we refer the works of \cite{DKP14,DKP16,BLS18,Coz17,KKL19,GL24}.

Earlier works on fractional problems with Hardy terms mainly study existence, multiplicity, maximum principles, symmetry, or Hardy-Sobolev minimization problems, see for instance \cite{CMS18,MM19,MMO26,NSD26}. Wang and Liu \cite{WL24} proved maximum principles and monotonicity properties for equations involving the fractional $p$-Laplacian with a Hardy potential. Their work does not identify the powers that govern the behaviour at the singularity and at infinity. Even without the Hardy potential, Feng and Du \cite{FD27} recently established the optimal decay at infinity of general nonnegative finite-energy solutions of the critical fractional $p$-Laplace equation. The Hardy potential $|x|^{-sp}$ adds a new difficulty. It produces two distinct characteristic exponents, one for the origin and one for infinity. To the best of our knowledge, this is the first work to identify both roots $\gamma_-(\mu)$ and $\gamma_+(\mu)$ of the Hardy equation as the optimal power exponents of positive radially symmetric nonincreasing finite-energy solutions of \eqref{eq:P} for $p\ge2$.

Finally, our comparison principle for \eqref{eq:intro-comparison-equation} is based on the fractional Picone inequality of Brasco and Franzina \cite{BF14}. The same $p$-homogeneous zeroth-order term appears for both functions compared, and cancels. No smallness of $\nu$ and no upper bound on $h$ are needed, provided the test functions are admissible. Comparison principles of other types are due to Xie \cite{Xie20}, Jarohs \cite{Jar18} and Biswas, Roy, and Sen \cite{BRS26}. 

\subsection{Main results}\label{subsec:mainresults}
Now we state the main results of this paper. Our first result describes the shape of $\Phi$ and the two Hardy exponents.
\begin{theorem}\label{thm:A}
The coefficient $\Phi$ of \eqref{eq:homog} admits the paired representation
\eqref{eq:intro-paired}, and
\[
  \operatorname{sign}\Phi'(\gamma)=\operatorname{sign}\big(\beta(\gamma)-\gamma\big),
  \quad \beta(\gamma)=N-sp-(p-1)\gamma .
\]
Hence $\Phi$ is strictly increasing on $(0,\gamma^{*})$, strictly decreasing on
$\big(\gamma^{*},\tfrac{N}{p-1}\big)$, and $\gamma^{*}$ is its unique maximizer. Its maximum satisfies
$2\bar\mu=C_H$, the sharp constant in the fractional Hardy inequality. For every
$0<\mu<\bar\mu$ the equation $\Phi(\gamma)=\mu$ has exactly the two roots
\eqref{eq:intro-gammapm}, so that \eqref{eq:intro-hardy} has the two characteristic radial power
profiles $|x|^{-\gamma_{-}(\mu)}$ and $|x|^{-\gamma_{+}(\mu)}$. Moreover $\gamma_{-}$ is increasing and
$\gamma_{+}$ decreasing in $\mu$, with $\gamma_{-}\to0$, $\gamma_{+}\to\gamma_{c}$ as $\mu\downarrow0$
and $\gamma_{\pm}\to\gamma^{*}$ as $\mu\uparrow\bar\mu$.
\end{theorem}

Our second result shows that the Hardy term improves the integrability of radial solutions.
\begin{theorem}\label{thm:integ}
Let $0<\mu<\bar\mu$ and let $U\in\Dsp$ be a positive radially symmetric nonincreasing weak
solution of \eqref{eq:P}. We put
\begin{equation}\label{eq:hbeta}
  h(\beta):=\frac{1+p\beta}{(1+\beta)^{p}},\text{ where } \beta>-\frac1p .
\end{equation}
Then $h$ is strictly increasing on $(-\frac1p,0)$, strictly decreasing on $(0,\infty)$,
$h(0)=1$ and $h((-\frac1p)_+)=h(+\infty)=0$. Consequently
$\{h>\frac{\mu}{\bar\mu}\}=(\beta_{-},\beta_{+})$ with $-\frac1p<\beta_{-}<0<\beta_{+}<\infty$, and
\begin{equation}\label{eq:Lq}
  U\in L^{q}(\R^{N})\quad\text{for every } q\in(q_{-},q_{+}),
  \quad q_{\pm}:=(1+\beta_{\pm})p^{*}_{s} .
\end{equation}
In particular, $q_{-}<p^{*}_{s}<q_{+}$.
\end{theorem}

Next we compare the Hardy energy with two weighted energies. For $u\in\Dsp$, $\mu>0$, and $\gamma>0$, we consider the following two integrals:
\begin{align}
&\mathcal D_\mu(u):=[u]_{s,p}^p-2\mu\int_{\mathbb R^N}\frac{|u(x)|^p}{|x|^{sp}}\dx, \label{eq:intro-defect} \text{ and }\\
&\mathcal E_\gamma(u):=\iint_{\mathbb R^{2N}}\frac{|u(x)-u(y)|^p}{|x-y|^{N+sp}} |x|^{-\frac{p\gamma}{2}}|y|^{-\frac{p\gamma}{2}}\dxy.\label{eq:intro-Egamma}
\end{align}
Our third result shows that one value of $\mu$ controls both weighted energies.
\begin{theorem}\label{thm:C}
Let $p\ge2$ and $0<\mu<\bar\mu$. There is $c_{p}>0$ such that, for every $u\in\Dsp$,
\begin{equation}\label{eq:intro-two-groundstates}
\mathcal D_\mu(u) \ge c_p\max\left\{\mathcal E_{\gamma_-}\bigl(|x|^{\gamma_-}u\bigr),\ \mathcal E_{\gamma_+}\bigl(|x|^{\gamma_+}u\bigr)\right\}.
\end{equation}
\end{theorem}

We first specify the notion of optimal exponent used below.

\begin{definition}\label{def:optimal}
Let $U>0$ on $\R^{N}\setminus\{0\}$, and let $\gamma>0$.
\begin{enumerate}
\item[{\rm (a)}] We say that $\gamma$ is the \emph{optimal exponent of $U$ at the origin} if there are $C,r>0$
with $U(x)\le C|x|^{-\gamma}$ for $0<|x|<r$, and if for every $\alpha<\gamma$ there are
$c_{\alpha},r_{\alpha}>0$ with $U(x)\ge c_{\alpha}|x|^{-\alpha}$ for $0<|x|<r_{\alpha}$.
\item[{\rm (b)}] We say that $\gamma$ is the \emph{optimal exponent of $U$ at infinity} if there are $C,R>0$
with $U(x)\le C|x|^{-\gamma}$ for $|x|>R$, and if for every $\beta>\gamma$ there are
$c_{\beta},R_{\beta}>0$ with $U(x)\ge c_{\beta}|x|^{-\beta}$ for $|x|>R_{\beta}$.
\end{enumerate}
\end{definition}

The optimal exponent is unique when it exists, and $\log U(x)/\log(1/|x|)\to\gamma$ as $x\to0$ in (a) and as $|x|\to\infty$ in (b). The constants $c_{\alpha}$, $c_{\beta}$ may tend to $0$, and we only need to check the lower bounds for $\alpha$ and $\beta$ close to $\gamma$.

Our final result gives the optimal exponents of radial solutions at both ends.
\begin{theorem}\label{thm:D}
Let $p\ge2$ and $0<\mu<\bar\mu$, and let $U\in\Dsp$ be a positive radially symmetric
nonincreasing weak solution of \eqref{eq:P}. Then $\gamma_{-}(\mu)$ is the optimal exponent of $U$ at
the origin, and $\gamma_{+}(\mu)$ is the optimal exponent of $U$ at infinity. Consequently,
\begin{align*}
    &U\in L^{q}(\R^{N}) \text{ for } \frac{N}{\gamma_{+}(\mu)}<q<\frac{N}{\gamma_{-}(\mu)}, \text{ and }\\
    &U\notin L^{q}(\R^{N}) \text{ for } 0<q<\frac{N}{\gamma_{+}(\mu)}\ \text{ and for }
  q>\frac{N}{\gamma_{-}(\mu)}.
\end{align*}
\end{theorem}

\begin{remark}
Near the origin $U$ is unbounded, since $U(x)\ge c_{\alpha}|x|^{-\alpha}$ for $\alpha<\gamma_{-}(\mu)$. At infinity $U$ decays more slowly than the Sobolev extremals in \eqref{sharp-brasco}, since $\gamma_{+}(\mu)<\gamma_c$. As $\mu\downarrow0$ we have $\gamma_{-}(\mu)\to0$ and $\gamma_{+}(\mu)\to\gamma_c$ (Corollary~\ref{cor:roots}), which is the exponent in \eqref{sharp-brasco}.
\end{remark}

\subsection{Strategies for the optimal asymptotic estimates}\label{subsec:difficulties}
The Sobolev exponent $\gamma^*=\frac{N}{p^*_s}$ is the natural scale of the energy space. A power $|x|^{-\gamma}$ has finite energy near the origin only if $\gamma<\gamma^*$. It has finite energy near infinity only if $\gamma>\gamma^*$. Since $0<\gamma_-<\gamma^*<\gamma_+<\gamma_c$,
we expect the profile $|x|^{-\gamma_-}$ at the origin and the profile $|x|^{-\gamma_+}$ at infinity. Both powers solve the same Hardy equation, so at each end only the energy excludes the other root. When $\mu=0$ we have $\gamma_-=0$ and $\gamma_+=\gamma_c$. Then solutions are bounded near the origin, and only the decay at infinity has to be proved, as in \cite{BMS16}.
The method of Brasco, Mosconi and Squassina \cite{BMS16} for the fractional Soboloev equation cannot be applied directly to the present problem. Their decay argument uses the fundamental solution $\Gamma(x)=|x|^{-(N-sp)/(p-1)}$ which satisfies $(-\Delta_p)^s\Gamma=0$ away from the origin. In the presence of a positive Hardy potential, however, this function no longer provides the supersolution needed for the upper decay estimate. The Hardy term has the same scaling as the fractional $p$-Laplacian and therefore changes the characteristic exponent at infinity. Moreover, when $\mu>0$, the solution may be singular at the origin, where a second characteristic exponent is needed. To obtain the sharp bounds, we use the two positive homogeneous solutions of the Hardy equation, establish weighted energy estimates for the corresponding ground-state substitutions, and prove higher integrability through finite Moser iterations. A comparison argument then gives the exact characteristic exponents. We explain this in three points.
\begin{enumerate}
\item[(a)] \emph{Two ground states.} At each end we work with the weight $\omega_\gamma=|x|^{-\gamma}$ adapted to that end, taking $\gamma=\gamma_-$ at the origin and $\gamma=\gamma_+$ at infinity. The weighted energies $\mathcal E_{\gamma_-}$ and $\mathcal E_{\gamma_+}$ in \eqref{eq:intro-Egamma} differ. So are the weighted Sobolev inequalities. For $\gamma_-<\gamma^*$ this inequality is available from \cite{AB17}. For $\gamma_+>\gamma^*$ we obtain it via the inversion $x\mapsto x/|x|^2$. This inversion sends $\gamma$ to $2\gamma^*-\gamma$. This image is positive for all $\gamma<\gamma_c$ exactly when $p\ge2$.
\item[(b)] \emph{No symmetry between the ends.} For $p=2$ we have $\gamma_-+\gamma_+=\gamma_c$. Then the inversion exchanges the two roots, and the result at one end gives the result at the other. For $p\neq2$ this relation breaks down by \eqref{eq:intro-reflection-failure}. Also, the nonlinear operator admits no Kelvin transform. Each end therefore requires its own argument.
\item[(c)] \emph{Different cutoffs.} Near the origin, the weight and the function $v=|x|^{\gamma_-}u$ are singular at $0$. So we cut off the test functions near $0$. At infinity we cut them off at a large radius. The errors produced by these two cutoffs are controlled by different ways.
\end{enumerate}
Finally, the comparison principle removes the small loss $\eta$ in \eqref{eq:intro-M0} and \eqref{eq:intro-Minf}. But it cannot prove these bounds. Its proof uses them as known results, since they make the Picone test functions admissible (see Remark~\ref{rem:barrier}). So the main work of the paper is therefore the proof of \eqref{eq:intro-M0} and \eqref{eq:intro-Minf}.
\vspace{0.2 cm}

\noi \textbf{Roadmap for the proof of \eqref{eq:intro-M0} and \eqref{eq:intro-Minf}:}
Let $p\ge2$ and let $U$ be a positive radially symmetric nonincreasing solution of \eqref{eq:P}. We take $\gamma=\gamma_-$ at the origin and $\gamma=\gamma_+$ at infinity. The proof of \eqref{eq:intro-M0} and \eqref{eq:intro-Minf} consists of six parts.
\begin{enumerate}
\item[(1)] \emph{Ground-state substitution.} We write $U=\omega_\gamma v$ with $\omega_\gamma=|x|^{-\gamma}$. We recall that $\pLap\omega_\gamma=\mu|x|^{-sp}\omega_\gamma^{p-1}$. We test \eqref{eq:P} with $\varphi$, and we test the equation for $\omega_\gamma$ with $v^{p-1}\varphi$. When we subtract, the Hardy term cancels exactly, see Lemma~\ref{lem:route}. Only the term $U^{p^*_s-p}$ remains. This term is small near $0$ and near infinity.
\item[(2)] \emph{Coercivity.} We test with $\varphi=\omega_\gamma\eta^pG_M(v)$. Here $G_M$ is a truncated power of order $r-p+1$, and $\eta$ is a cutoff. We bound the resulting integrand from below by $c\,(ab)^{\frac{p}{2}}|H_M(X)-H_M(Y)|^p$, where $H_M(t)\approx t^{\frac{r}{p}}$, with the notation of Section~\ref{sec:rep}. For $r=p$ this is the ground-state inequality of Frank and Seiringer \cite{FS08}. We extend it to all powers $r\ge p$ and to their truncations, see Section~\ref{sec:route} and Appendix~\ref{app:coerc}.
\item[(3)] \emph{Caccioppoli inequality.} The cutoff $\eta$ produces error terms. They are of first order in the increments of $v$. We control them by an additional quadratic coercivity estimate. Then Proposition~\ref{prop:cacc} gives
\[
  c_r\,\mathcal E_\gamma\big(\eta H_M(v)\big) \le  \int_{\R^N}U^{p^*_s-p}\big(\omega_\gamma\eta H_M(v)\big)^p\dx+\text{cutoff terms} .
\]
\item[(4)] \emph{Weighted Sobolev inequality.} By Proposition~\ref{prop:WS} we have $\mathcal E_\gamma(w)\ge S_\gamma\big(\int|x|^{-\gamma p^*_s}|w|^{p^*_s}\dx\big)^{p/p^*_s}$. By H\"older inequality, the first term on the right in (3) is at most
\[
  \left(\int_{\operatorname{supp}\eta}U^{p^*_s}\dx\right)^{\frac{sp}{N}}\left(\int_{\R^N}|x|^{-\gamma p^*_s}|\eta H_M(v)|^{p^*_s}\dx\right)^{\frac{p}{p^*_s}}.
\]
We choose $\operatorname{supp}\eta$ to be a small ball at the origin, or the complement of a large ball at infinity. Then the first factor is small. So we can move this term to the left-hand side, see Corollary~\ref{cor:onestep}.
\item[(5)] \emph{Finite higher integrability.} We let $M\to\infty$ and obtain the gain
\[
  v\in L^r\big(\Sigma',|x|^{-\gamma p^*_s}\dx\big) \Longrightarrow v\in L^{r\chi}\big(\Sigma'',|x|^{-\gamma p^*_s}\dx\big),\qquad \chi=\frac{p^*_s}{p}>1 .
\]
Here $\Sigma''\subset\Sigma'$ is a smaller region. We start from $r_0=p^*_s$, which only means $U\in L^{p^*_s}$. Then we apply the gain finitely many times, with $r_k=p^*_s\chi^k$, see Theorem~\ref{thm:cond}.
\item[(6)] \emph{Pointwise bounds.} Since $U$ is radially nonincreasing, the bound $\int_{B_{\rho_0}}|x|^{-\gamma p^*_s}v^r\dx<\infty$ gives
\[
  U(\rho)\le C\rho^{-\gamma- p^*_s\frac{\gamma^*-\gamma}{r}},\qquad 0<\rho<\rho_0 ,
\]
see Lemma~\ref{lem:radexp}. We have a similar bound at infinity. The loss $p^*_s|\gamma^*-\gamma| r^{-1}$ tends to $0$ as $r\to\infty$. In particular, for large $r$ we obtain \eqref{eq:intro-M0} and \eqref{eq:intro-Minf}. 
\end{enumerate}
\vspace{0.2 cm}

\noi \textbf{The requirement of the Caccioppoli inequality and the iteration:}
From Theorem~\ref{thm:C} we get $\mathcal E_{\gamma_\pm}(|x|^{\gamma_\pm}U)<\infty$ for every solution $U$. We might hope that this energy bound already gives the exponents $\gamma_\pm$ directly, but it does not. Combined with the weighted Sobolev inequality, it produces only $\int_{\R^N} U^{p^*_s}\dx<\infty$, which we already know. This is the case $r=p^*_s$ in part (6) above, and it gives $U(x)\le C|x|^{-\gamma^*}$. The exponent $\gamma^*$ is identical at both ends, and is independent of $\mu$. The loss in part (6) is of order $1/r$, so the exponents $\gamma_\pm$ become visible only for large $r$. This is why we need integrability of $v$ of high order, which we obtain by iteration.

The iteration has to be carried out locally, for several reasons. First,  after the substitution, the equation still contains the term $U^{p^*_s-p}$, which belong to $L^{N/sp}(\R^N)$. This is exactly the scaling of the Hardy potential $|x|^{-sp}$. The norm of $U^{p^*_s-p}$ on $\R^N$ need not be small, so the Sobolev inequality on the whole space cannot absorb it.
The norm is small, however, on sufficiently small balls around the origin and outside sufficiently large balls. We therefore treat one end at a time.
Second, the test functions are admissible only away from the origin, since the weighted Sobolev inequality holds for functions with compact support in $\R^N\setminus\{0\}$. Third, the localization produces error terms, and the Caccioppoli inequality is the tool that controls them.

We stop the iteration after finitely many steps. The constant $c_r$ in part (3) becomes small as $r\to\infty$, so we choose the radius in part (4) after $r$, and the radius changes at every step. An infinite iteration would require control of all these radii and constants, which we avoid. For a given loss $\eta>0$, a finite $r$ is enough in part (6), and the comparison principle then removes the remaining loss.

In summary, for a positive radially symmetric nonincreasing solution $U$ of \eqref{eq:P}, Theorem~\ref{thm:integ} gives
\begin{equation}\label{eq:intro-first-Lq}
U\in L^q(\mathbb R^N) \quad \text{for every } q\in(q_-,q_+),
\end{equation}
and, when $p\ge2$, Theorem~\ref{thm:D} gives
\begin{equation}\label{eq:intro-final-bounds}
U(x)\le C|x|^{-\gamma_-(\mu)}\quad\text{as }|x|\to0,\quad U(x)\le C|x|^{-\gamma_+(\mu)} \quad \text{as }|x|\to\infty,
\end{equation}
and
\begin{equation}\label{eq:intro-final-Lq}
U\in L^q(\mathbb R^N) \quad \text{for every } \frac{N}{\gamma_+(\mu)}<q<\frac{N}{\gamma_-(\mu)}.
\end{equation}

\subsection{Organization of the paper}
The structure of the paper is as follows.

\begin{itemize}
    \item[Section~\ref{sec:rep}] studies the homogeneity function $\Phi$. We derive the paired representation \eqref{eq:intro-paired}, prove its monotonicity and identify its maximum with one half of the sharp fractional Hardy constant. We then obtain the two characteristic roots $\gamma_-(\mu)$ and $\gamma_+(\mu)$.
    \item[Section~\ref{sec:Phi2}] introduces the homogeneity function associated with the linearized equation. We determine its two homogeneous exponents and study the reflection property. We also compare this symmetry with the nonlinear relation between $\gamma_-(\mu)$ and $\gamma_+(\mu)$.
    \item[Section~\ref{sec:rigidity}] develops the tools needed for the Hardy-Sobolev application. We construct explicit homogeneous subsolutions and supersolutions and prove the Picone-based comparison principle.
    \item[Section~\ref{sec:app}] applies these results to the critical Hardy-Sobolev equation. We first obtain the characteristic upper bounds under the preliminary assumptions \eqref{eq:intro-M0} and \eqref{eq:intro-Minf}. For radially symmetric nonincreasing solutions, we then prove the Hardy-dependent integrability improvement \eqref{eq:intro-first-Lq}.
    \item[Section~\ref{sec:gsrep}] develops the ground-state representation associated with an arbitrary homogeneous Hardy solution. For $p\ge2$, applying this result to $\gamma_-(\mu)$ and $\gamma_+(\mu)$ gives the two-ground-state remainder inequality \eqref{eq:intro-two-groundstates}.
    \item[Section~\ref{sec:verify}] performs the ground-state substitution and proves the coercivity estimate it requires. Section~\ref{sec:scheme} turns that estimate into a weighted Caccioppoli inequality, runs a finite Moser iteration, and proves \eqref{eq:intro-M0} and \eqref{eq:intro-Minf} for positive radially symmetric nonincreasing solutions when $p\ge2$. Together with the comparison argument of Section~\ref{sec:app}, this gives the characteristic bounds \eqref{eq:intro-final-bounds}, and the range \eqref{eq:intro-final-Lq}.
    \item[Section~\ref{sec:proofs}] collects the proofs of the results stated in Section~\ref{subsec:mainresults}.
    \item[Appendices] Appendix~\ref{app:pv} checks the principal value in \eqref{eq:Phidef}. Appendix~\ref{app:power} contains further properties of power solutions (cylinder formulation, touching lemma, compactness and a blow-up dichotomy), which are not needed for the main results. Appendix~\ref{app:coerc} contains the proofs of the pointwise coercivity inequalities of Section~\ref{sec:verify}.
\end{itemize}

\section{The homogeneity function and the Hardy exponents}\label{sec:rep}
\subsection*{Notations and definitions}
We use the following notations and definitions throughout this paper. 
\begin{itemize}
    \item For $sp<N$, we denote $$\gamma_c:=\frac{N-sp}{p-1},
\quad
\gamma^*:=\frac{N-sp}{p},
\quad
p_s^*:=\frac{Np}{N-sp}, \quad P_{s}:=p^{*}_{s}-p=\frac{sp}{\gamma^{*}}.$$
    \item The tail space is defined as 
    \begin{align*}
        L^{p-1}_{sp}(\R^{N}) := \left\{ f\in L^{p-1}_{\mathrm{loc}}(\R^N) : \int_{\R^N}(1+|y|)^{-N-sp}|f(y)|^{p-1}\dy < \infty \right\}.
    \end{align*}
    \item  For an open set $D$, $\mathcal W^{s,p}_{0}(D)$ is the closure of $\cc(D)$ in $W^{s,p}(\R^N)$.
    \item The Gagliardo seminorm is defined as 
    \[ [u]_{s,p} := \left(\iint_{\R^{2N}} \frac{|u(x)-u(y)|^p}{|x-y|^{N+sp}}\dxy\right)^{\frac{1}{p}}.\]
    \item For $sp<N$, $\Dsp$ has the following characterization (see \cite{BGV21}): 
          \[\Dsp:=\left\{u \in L^{p^{*}_{s}}(\R^{N}):[u]_{s,p}<\infty\right\}.\]
    \item We denote $J_p(t):=|t|^{p-2}t$ for $t \in \R$.
    \item A function $u\in W^{s,p}_{\mathrm{loc}}(\Omega)\cap L^{p-1}_{sp}(\R^{N})$ is a weak solution of
$\pLap u=f$ in an open set $\Omega$ if
\begin{equation}\label{eq:weakdef}
  \frac12\iint_{\R^{2N}}\frac{\Jp\big(u(x)-u(y)\big)\big(\varphi(x)-\varphi(y)\big)}{|x-y|^{N+sp}}\dxy
   = \int_{\Omega}f\varphi\dx
  \quad\text{for all }\varphi\in \cc(\Omega),
\end{equation}
subsolutions and supersolutions being defined by the corresponding inequalities for $\varphi\ge0$. In particular, the weak formulation \eqref{eq:weakdef} also works for $u \in \Dsp$ and $\phi \in \mathcal W^{s,p}_{0}(\Omega)$.
\item Taking
$x=e_{1}$ and $y=r\omega$ in \eqref{eq:op} gives the one-dimensional representation
\begin{equation}\label{eq:Phidef}
  \Phi(\gamma) = \mathrm{P.V.}\int_{0}^{\infty}
  \Jp\big(1-r^{-\gamma}\big) r^{N-1} \Psi(r)\dr,
\end{equation}
where
\begin{equation}\label{eq:Psi}
  \Psi(r) := \int_{\Sph^{N-1}}\frac{\d\omega}{|e_{1}-r\omega|^{N+sp}}
   = |\Sph^{N-2}|\int_{0}^{\pi}\frac{\sin^{N-2}\theta\d\theta}
  {\big(1-2r\cos\theta+r^{2}\big)^{\frac{N+sp}{2}}} .
\end{equation}
This is the kernel appearing in the ground-state representation of \cite{FS08}. It is positive and real
analytic on $(0,\infty)\setminus\{1\}$, and satisfies the inversion identity
\begin{equation}\label{eq:Psiinv}
  \Psi\left( \frac{1}{\rho} \right) = \rho^{N+sp} \Psi(\rho),\quad \rho>0.
\end{equation}
 \item $c$ and $C$ denote positive constants whose values may change from line to line, and $f\asymp g$ means $cg\le f\le Cg$ for some $c,C>0$.
 \item  For a family $f_M$ of functions on $[0,\infty)$ depending on a truncation level $M>0$, we write $f_M(t)\approx g(t)$ if $f_M(t)=g(t)$ for $0\le t\le M$ and $f_M(t)\uparrow g(t)$ as $M\uparrow\infty$.
 \item For $a,b\in\R$ and $c,d>0$, we denote 
\begin{equation*}
  \mathfrak{B}(a,b,c,d):=|a-b|^{p}-\Jp(c-d)\left(\frac{|a|^{p}}{c^{p-1}}-\frac{|b|^{p}}{d^{p-1}}\right).
\end{equation*}
\item For $a,b\in\R$ and $c,d>0$, we denote
\begin{align*}
\mathfrak{K}(a,b,c,d):=\Jp(ac-bd)\big[aG(c)-bG(d)\big] -\Jp(a-b)\big[ac^{p-1}G(c)-bd^{p-1}G(d)\big],
\end{align*}
where $G: \R^+ \mapsto \R$ is nondecreasing.
\end{itemize}

\subsection{Representation formula}
We first write $\Phi$ as an absolutely convergent integral over $(0,1)$. This formula follows from the work of Del Pezzo and Quaas \cite[Theorem~1.1]{DPQ25}.
\begin{proposition}\label{prop:rep}
We set
\begin{equation}\label{eq:beta}
  \beta(\gamma):=\;N-sp-\gamma(p-1)=(p-1)\big(\gamma_{c}-\gamma\big).
\end{equation}
Then for every $\gamma\in\big(0,\tfrac{N}{p-1}\big)$ the principal value in \eqref{eq:Phidef} exists
and
\begin{equation}\label{eq:rep}
  \Phi(\gamma)=\int_{0}^{1}
  \big(1-r^{\gamma}\big)^{p-1} \big(1-r^{\beta(\gamma)}\big) r^{sp-1} \Psi(r)\dr,
\end{equation}
the integral on the right being absolutely convergent.
\end{proposition}

\begin{proof}
We substitute $r=e^{v}$ in \eqref{eq:Phidef}, so that $r^{N-1}dr=e^{Nv}dv$ and
\[
  \Phi(\gamma)=\mathrm{P.V.}\int_{-\infty}^{\infty}
  \Jp\big(1-e^{-\gamma v}\big) e^{Nv} \Psi(e^{v})\dv ,
\]
the principal value being the symmetric one at $v=0$. We fix $v>0$ and add the part of $v$ and
$-v$. Since $1-e^{-\gamma v}>0$ we have $\Jp(1-e^{-\gamma v})=(1-e^{-\gamma v})^{p-1}$, while
$1-e^{\gamma v}=-e^{\gamma v}(1-e^{-\gamma v})$ gives
\begin{equation}\label{eq:Jpid}
  \Jp\big(1-e^{\gamma v}\big) = - e^{\gamma(p-1)v} \big(1-e^{-\gamma v}\big)^{p-1}.
\end{equation}
By \eqref{eq:Psiinv}, $\Psi(e^{-v})=e^{(N+sp)v}\Psi(e^{v})$. The two parts
therefore combine into
\begin{gather*}
    \big(1-e^{-\gamma v}\big)^{p-1} \Psi(e^{v})
  \Big[ e^{Nv}-e^{\gamma(p-1)v+spv} \Big]
  =\big(1-e^{-\gamma v}\big)^{p-1}\big(1-e^{-\beta(\gamma)v}\big) e^{Nv} \Psi(e^{v}),
\end{gather*} 
with $\beta(\gamma)$ as in \eqref{eq:beta}. The two singular parts at $v=0$ have thus been
combined explicitly. Therefore,
\begin{equation*}
\Phi(\gamma)=\int_{0}^{\infty}\big(1-e^{-\gamma v}\big)^{p-1}
  \big(1-e^{-\beta(\gamma) v}\big) e^{Nv} \Psi(e^{v})\dv ,
\end{equation*}
and the substitution $r=e^{-v}$, with \eqref{eq:Psiinv} once more, gives \eqref{eq:rep}.

To prove absolute convergence, note that as $r\uparrow1$ we have  $(1-r^{\gamma})^{p-1}\sim\gamma^{p-1}(1-r)^{p-1}$,
$1-r^{\beta}\sim\beta(1-r)$ and $\Psi(r)\sim c(1-r)^{-1-sp}$, so the integrand is
$O\big((1-r)^{p-1-sp}\big)$, integrable because $p>sp$. As $r\downarrow0$ we have 
$\Psi(r)\to|\Sph^{N-1}|$ and $(1-r^{\gamma})^{p-1}\to1$. If $\beta(\gamma)\ge0$ the integrand is
$O(r^{sp-1})$, and if $\beta(\gamma)<0$ it is $O\big(r^{\beta(\gamma)+sp-1}\big)$, integrable
when $\beta(\gamma)+sp>0$, that is when $\gamma<\frac{N}{p-1}$.
\end{proof}

Lemma~\ref{lem:pv} in Appendix~\ref{app:pv} shows that the principal value in \eqref{eq:Phidef} coincides with the usual one, obtained by removing small balls around $e_1$.

\subsection{Monotonicity and the sharp Hardy constant}\label{sec:props}
We now locate the maximum of $\Phi$ and identify its value. Part (i) below is due to Del Pezzo and Quaas \cite[eq.~(2)]{DPQ25}.

\begin{proposition}\label{thm:props}
Let $\gamma_{c},\gamma^{*}$ be as in \eqref{eq:gammas}. On $I:=\big(0,\tfrac{N}{p-1}\big)$ the following holds.
\begin{enumerate}
\item[{\rm (i)}] $\Phi>0$ on $(0,\gamma_{c})$,
      $\Phi(0^{+})=\Phi(\gamma_{c})=0$, and
      $\Phi<0$ on $\big(\gamma_{c},\tfrac{N}{p-1}\big)$.
\item[{\rm (ii)}] $\Phi\in \C^{1}(I)$ and
\begin{equation}\label{eq:dPhi}
  \Phi'(\gamma)=(p-1)\int_{0}^{1}(-\log r) \big(1-r^{\gamma}\big)^{p-2}
  \big(r^{\gamma}-r^{\beta(\gamma)}\big) r^{sp-1} \Psi(r)\dr .
\end{equation}
\item[{\rm (iii)}] $\Phi$ is strictly increasing on $(0,\gamma^*)$
and strictly decreasing on $(\gamma^*,\frac{N}{p-1})$. In particular, $\gamma^*$ is the unique maximizer of $\Phi$.
\item[{\rm (iv)}] Let $\bar\mu:=\Phi(\gamma^{*})$. Then we have 
      $2\bar\mu=C_H$ where $C_H$ is the sharp constant in the fractional Hardy inequality
\begin{equation}\label{eq:FS}
  \iint_{\R^{N}\times\R^{N}}\frac{|u(x)-u(y)|^{p}}{|x-y|^{N+sp}}\dxy
  \ge C_H\int_{\R^{N}}\frac{|u(x)|^{p}}{|x|^{sp}}\dx ,
  \quad \forall  \, u\in \cc(\R^{N}\setminus\{0\}).
\end{equation}
\end{enumerate}
\end{proposition}

\begin{proof}
(i) On $(0,1)$ every factor in \eqref{eq:rep} other than $1-r^{\beta(\gamma)}$ is strictly positive
when $\gamma>0$. Since $0<r<1$, the sign of $1-r^{\beta}$ is the sign of $\beta$, and by \eqref{eq:beta}
$\beta(\gamma)>0$ for $\gamma<\gamma_{c}$. At $\gamma=\gamma_{c}$ the factor
$1-r^{\beta(\gamma_{c})}$ vanishes identically, and at $\gamma=0$ the factor $(1-r^{\gamma})^{p-1}$
vanishes identically.

(ii) Differentiating the integrand of \eqref{eq:rep} and using $\beta'(\gamma)=-(p-1)$,
\begin{equation*}
  \partial_{\gamma}\Big[(1-r^{\gamma})^{p-1}(1-r^{\beta})\Big]
  =(p-1)(-\log r)(1-r^{\gamma})^{p-2}\Big[r^{\gamma}\big(1-r^{\beta}\big)-\big(1-r^{\gamma}\big)r^{\beta}\Big],   
\end{equation*}
and the bracket equals $r^{\gamma}-r^{\gamma+\beta}-r^{\beta}+r^{\gamma+\beta}=r^{\gamma}-r^{\beta}$,
the cross terms cancel. Differentiation under the integral sign is allowed because, for $\gamma$ in a
compact subinterval of $I$, the resulting integrand is bounded near $r=1$ by $C(1-r)^{p-1-sp}$ and
near $r=0$ by $C r^{\min\{0,\beta\}+sp-1}\log(1/r)$, both integrable.

(iii) In \eqref{eq:dPhi} the factors $-\log r$, $(1-r^{\gamma})^{p-2}$, $r^{sp-1}$ and $\Psi(r)$
are strictly positive on $(0,1)$. Further, 
\begin{equation*}
   \sgn\big(r^{\gamma}-r^{\beta}\big)=\sgn\big(\beta-\gamma\big)   \text{ for every }r\in(0,1), 
\end{equation*}
independently of $r$. Thus $\sgn\Phi'=\sgn(\beta(\gamma)-\gamma)$, and $\beta(\gamma)=\gamma$
if and only if $N-sp=\gamma p$, that is $\gamma=\gamma^{*}$. As $\gamma\mapsto\beta(\gamma)-\gamma$ is
strictly decreasing, the monotonicity follows.

(iv) At $\gamma=\gamma^{*}$ we have $\beta(\gamma^{*})=\gamma^{*}$, so the two factors in \eqref{eq:rep}
coincide and
\[
  \Phi(\gamma^{*})=\int_{0}^{1}\big(1-r^{\frac{N-sp}{p}}\big)^{p} r^{sp-1} \Psi(r)\dr,
\]
which is one half of
$C_H=2\int_{0}^{1}r^{sp-1}\big|1-r^{\frac{N-sp}{p}}\big|^{p}\Psi(r)\dr$, the
sharp constant in \eqref{eq:FS} by \cite[Theorem ~1.1]{FS08}.
\end{proof}

\begin{remark}\label{rem:sym}
(a) For $p=2$, \eqref{eq:rep} gives $\Phi(\gamma)=\Phi(\gamma_c-\gamma)$ and $\gamma_=\gamma_c/2$. For $p\neq2$ this symmetry is lost. Since $\gamma^*/\gamma_c=(p-1)/p$, we have $\gamma^*<\gamma_c/2$ if $1<p<2$, and $\gamma^*>\gamma_c/2$ if $p>2$.

(b) The factor $2$ in $2\bar\mu=C_H$ comes from the factor $\frac12$ in \eqref{eq:weakdef}. With the normalization of \cite{DPQ25} we have $\bar\mu=C_H$.
\end{remark}

We next describe the two roots of $\Phi(\gamma)=\mu$.
\begin{corollary}\label{cor:roots}
Let $\bar\mu$ be as in Proposition~\ref{thm:props}-(iv). For every $\mu\in(0,\bar\mu)$ the
equation $\Phi(\gamma)=\mu$ has two solutions in $(0,\gamma_{c})$,
\[
  0<\gamma_{-}(\mu)<\gamma^{*}=\frac{N-sp}{p}<\gamma_{+}(\mu)<\gamma_{c}=\frac{N-sp}{p-1}.
\]
The maps $\mu\mapsto\gamma_{\pm}(\mu)$ are $\C^{1}$ on $(0,\bar\mu)$, with $\gamma_{-}$ strictly
increasing, $\gamma_{+}$ strictly decreasing, and
\begin{equation*}
      \gamma_{-}(0^{+})=0,\quad \gamma_{+}(0^{+})=\gamma_{c},\quad
  \lim_{\mu\uparrow\bar\mu}\gamma_{\pm}(\mu)=\gamma^{*}.
\end{equation*}
In particular, the Hardy equation
\begin{equation}\label{eq:pureHardy}
  \pLap w  =  \mu \frac{w^{p-1}}{|x|^{sp}}\text{ in }\R^{N}\setminus\{0\},\quad w>0,
\end{equation}
has the two characteristic radial power profiles $|x|^{-\gamma_{-}(\mu)}$ and $|x|^{-\gamma_{+}(\mu)}$.
\end{corollary}

\begin{proof}
Immediate from Proposition~\ref{thm:props}-((i)-(iii)) and the implicit function theorem, since
$\Phi'\neq0$ at $\gamma_{\pm}(\mu)$ when $\mu<\bar\mu$. The last statement is
\eqref{eq:homog}.
\end{proof}

\section{Linearization and symmetry}\label{sec:Phi2}
We linearize the Hardy equation at a power solution. The resulting homogeneity function has a symmetry that the nonlinear one lacks.
\subsection{The homogeneity function for the linearized equation}
For $\gamma,\gamma'>0$ with $N-\gamma(p-2)-\gamma'>0$ (we call such $\gamma,\gamma'$ admissible), define
$$
\Phi^{(2)}(\gamma,\gamma')
:=\mathrm{P.V.}\int_{\R^N}\big|1-|z|^{-\gamma}\big|^{p-2}\big(1-|z|^{-\gamma'}\big)\frac{dz}{|e_1-z|^{N+sp}}.
$$
For fixed $\gamma$ we put $\ell=\ell(\gamma):=N-sp-\gamma(p-2)$. The meaning of $\Phi^{(2)}$ as a linearized homogeneity coefficient is explained in Remark~\ref{rem:lin}.
\begin{proposition}\label{lem:Phi2}
Let $\gamma,\gamma'>0$ be admissible and $\ell=\ell(\gamma)$.
\begin{enumerate}
\item[{\rm (i)}] We have
\begin{equation}\label{eq:Phi2rep}
\Phi^{(2)}(\gamma,\gamma')=\int_0^1(1-r^\gamma)^{p-2}(1-r^{\gamma'})(1-r^{\ell-\gamma'})r^{sp-1}\Psi(r)\dr,
\end{equation}
the integral being absolutely convergent. Moreover, $\Phi^{(2)}(\gamma,\gamma)=\Phi(\gamma).$
\item[{\rm (ii)}] The function $\gamma'\mapsto\Phi^{(2)}(\gamma,\gamma')$ is positive on $(0,\ell)$, strictly increasing on $(0,\ell/2)$, strictly decreasing on $(\ell/2,\ell)$, and
\begin{equation}\label{eq:Phi2sym}
\Phi^{(2)}(\gamma,\gamma')=\Phi^{(2)}(\gamma,\ell-\gamma').
\end{equation}
\item[{\rm (iii)}] We have $\gamma<\ell(\gamma)/2$ if and only if $\gamma<\gamma^*$. Hence, if $\Phi(\gamma_-)=\mu$, then
\begin{equation}\label{eq:Phi2key}
\Phi^{(2)}(\gamma_-,\gamma')<\mu \text{ for every }\gamma'\in(0,\gamma_-).
\end{equation}
\end{enumerate}
\end{proposition}
\begin{proof}
{As in Proposition~\ref{prop:rep}, set $r=e^{v}$, pair $v$ with $-v$, use
$|1-e^{\gamma v}|^{p-2}=e^{\gamma(p-2)v}(1-e^{-\gamma v})^{p-2}$, $1-e^{\gamma'v}=-e^{\gamma'v}(1-e^{-\gamma'v})$ and \eqref{eq:Psiinv}, and then set $r=e^{-v}$. This gives \eqref{eq:Phi2rep}. Near $r=1$ the integrand is $O\big((1-r)^{p-1-sp}\big)$, which is integrable since $p>sp$, and near $r=0$ it is integrable because $\ell-\gamma'+sp>0$. For $\gamma'=\gamma$ we have $\ell-\gamma=\beta(\gamma)$, so $\Phi^{(2)}(\gamma,\gamma)=\Phi(\gamma)$ by \eqref{eq:rep}.

(ii) The only factor in \eqref{eq:Phi2rep} that can change sign is $(1-r^{\gamma'})(1-r^{\ell-\gamma'})$, which is positive for $0<\gamma'<\ell$. Replacing $\gamma'$ by $\ell-\gamma'$ exchanges these two factors, which gives \eqref{eq:Phi2sym}. As in Proposition~\ref{thm:props}-(ii), the cross terms cancel and
\[
\partial_{\gamma'}\Phi^{(2)}(\gamma,\gamma')
=\int_0^1 (1-r^\gamma)^{p-2}(-\log r)\big(r^{\gamma'}-r^{\ell-\gamma'}\big)r^{sp-1}\Psi(r)\dr .
\]
For $0<r<1$ the sign of $r^{\gamma'}-r^{\ell-\gamma'}$ is the sign of $\ell-2\gamma'$, which gives the monotonicity.

(iii) The inequality $2\gamma<N-sp-\gamma(p-2)$ is the same as $p\gamma<N-sp$. If $\Phi(\gamma_-)=\mu$, then $\gamma_-<\gamma^*$, so $\gamma_-<\ell(\gamma_-)/2$ and $\Phi^{(2)}(\gamma_-,\cdot)$ is strictly increasing on $(0,\gamma_-]$. Since $\Phi^{(2)}(\gamma_-,\gamma_-)=\mu$, we get \eqref{eq:Phi2key}.}
\end{proof}

\begin{remark}\label{rem:lin}
Let $w_{\gamma}=|x|^{-\gamma}$, $\phi=|x|^{-\gamma'}$ and $\delta f:=f(x)-f(y)$. The derivative of $t\mapsto J_p(\delta w_{\gamma}+t\,\delta\phi)$ at $t=0$ is $(p-1)|\delta w_{\gamma}|^{p-2}\delta\phi$. This gives the operator $L_{\gamma}$ of Corollary~\ref{cor:linexp}. For $1<p<2$ this kernel is singular on $\{|x|=|y|\}$, but after the pairing $r\leftrightarrow1/r$ the integrand in \eqref{eq:Phi2rep} is integrable. So we define the linearized coefficient by \eqref{eq:Phi2rep}. We do not claim that $\pLap$ is differentiable.
\end{remark}

The symmetry gives the exponents of the linearized Hardy equation explicitly.

\begin{corollary}\label{cor:linexp}
Let $0<\mu<\bar\mu$ and $\gamma\in\{\gamma_{-}(\mu),\gamma_{+}(\mu)\}$. Linearizing
\eqref{eq:pureHardy} at $|x|^{-\gamma}$ gives
\[
  L_{\gamma}\phi=\mu(p-1) |x|^{-sp-\gamma(p-2)}\phi,
  \quad
  L_{\gamma}\phi(x):=(p-1) \mathrm{P.V.}\!\int_{\R^{N}}
  \frac{\big||x|^{-\gamma}-|y|^{-\gamma}\big|^{p-2}\big(\phi(x)-\phi(y)\big)}{|x-y|^{N+sp}}\dy .
\]
Its power solutions $\phi=|x|^{-\gamma'}$ with $\gamma'\in(0,\ell(\gamma))$ are given by
\[
  \gamma'=\gamma\quad\text{and}\quad \gamma'=\beta(\gamma)=N-sp-\gamma(p-1),
\]
and these are distinct. For $p=2$ we have  $\beta(\gamma_{-})=\gamma_{+}$, so the exponents of the
linearized equation are those of \eqref{eq:pureHardy}.
For $p\neq2$ the two sets of exponents coincide only if $\beta(\gamma_-)=\gamma_+$.
\end{corollary}

\begin{proof}
{By Proposition~\ref{lem:Phi2}, $L_{\gamma}\big(|x|^{-\gamma'}\big)=(p-1)\Phi^{(2)}(\gamma,\gamma')|x|^{-\gamma(p-2)-\gamma'-sp}$, so $|x|^{-\gamma'}$ is a solution if and only if $\Phi^{(2)}(\gamma,\gamma')=\mu$. Since $\Phi^{(2)}(\gamma,\gamma)=\Phi(\gamma)=\mu$ and $\gamma\neq\ell/2$ (because $\gamma\neq\gamma^*$), Proposition~\ref{lem:Phi2}-(ii) shows that this equation has exactly two roots in $(0,\ell)$, symmetric about $\ell/2$. One is $\gamma'=\gamma$, so the other is $\ell-\gamma=\beta(\gamma)$, and they are distinct. For $p=2$, $\beta(\gamma_{-})=\gamma_{c}-\gamma_{-}=\gamma_{+}$ by Proposition~\ref{cor:defect}.}
\end{proof}

\subsection{Reflection of the Hardy exponents}
For $p=2$, $\Phi$ is symmetric under $\gamma\mapsto\gamma_c-\gamma$ (Remark~\ref{rem:sym}). We now see what this gives for the two roots.

\begin{proposition}\label{cor:defect}
    For $p=2$ we have  $\gamma_{-}(\mu)+\gamma_{+}(\mu)=\gamma_{c}$ for every $\mu\in(0,\bar\mu)$.
For $p\neq2$, this identity does not hold in general. More precisely, $\lim_{\mu\uparrow\bar\mu}
\left(\gamma_-(\mu)+\gamma_+(\mu)-\gamma_c\right)=\frac{(N-sp)(p-2)}{p(p-1)}$. Moreover,
$$\gamma_-(\mu)+\gamma_+(\mu)-\gamma_c\longrightarrow0 \, \text{as }\mu\downarrow0. $$
In particular, if $p\neq2$, the reflection identity fails for every
$\mu$ sufficiently close to $\bar\mu$.
\end{proposition}

\begin{proof}
For $p=2$, we have
$$\Phi(\gamma)=\Phi(\gamma_c-\gamma).$$
Hence $\gamma_c-\gamma_-(\mu)$ is a root of $\Phi(\gamma)=\mu$ lying above $\gamma^*$. By the uniqueness of
the upper root, $\gamma_c-\gamma_-(\mu)=\gamma_+(\mu)$, which proves the first statement. For general $p$, Corollary~\ref{cor:roots} gives $$\gamma_\pm(\mu)\longrightarrow\gamma^*\,\text{as }\mu\uparrow\bar\mu.$$ Therefore, $\gamma_-(\mu)+\gamma_+(\mu)-\gamma_c \longrightarrow2\gamma^*-\gamma_c$. Since
$$2\gamma^*-\gamma_c=(N-sp)\left(\frac{2}{p}-\frac{1}{p-1}\right)=\frac{(N-sp)(p-2)}{p(p-1)},$$
the limit is nonzero when $p\neq2$.  Finally, Corollary~\ref{cor:roots} gives
$$\gamma_-(\mu)\longrightarrow0,\quad\gamma_+(\mu)\longrightarrow\gamma_c. $$
as $\mu\downarrow0$. This proves the result.
\end{proof}
We do not determine whether $\gamma_-(\mu)+\gamma_+(\mu)-\gamma_c$ can vanish at an isolated value of $\mu$ when $p\neq2$.

\section{Homogeneous solutions, compactness and comparison}\label{sec:rigidity}
In this section we collect the tools used in Section~\ref{sec:app}. These are explicit barriers and a comparison principle. Further properties of power solutions (touching, scaling and compactness) are given in Appendix~\ref{app:power}.

\subsection{Subsolutions and supersolutions for the Hardy-Sobolev equation}\label{sec:barriers}
We first use the homogeneity function to construct explicit power subsolutions and supersolutions for \eqref{eq:P}.
We recall that
\begin{equation}\label{eq:Xi}
  P_s:=p_{s}^{*}-p=\frac{sp^{2}}{N-sp},
  \quad\text{so that}\quad P_s=\frac{sp}{\gamma^{*}}.
\end{equation}
We compute the Hardy-Sobolev operator on powers and obtain barriers.
\begin{proposition}\label{lem:barrier}
Let $\mu>0$, $\gamma\in\big(0,\tfrac{N}{p-1}\big)$, $A>0$ and $W(x):=A|x|^{-\gamma}$. Then in
$\R^{N}\setminus\{0\}$,
\begin{equation}\label{eq:barrier}
  \pLap W-\mu|x|^{-sp}W^{p-1}-W^{p_{s}^{*}-1}
   = |x|^{-\gamma(p-1)-sp}\Big[A^{p-1}\big(\Phi(\gamma)-\mu\big)
  -A^{p_{s}^{*}-1}|x|^{P_s(\gamma^{*}-\gamma)}\Big].
\end{equation}
Consequently, if $0<\mu<\bar\mu$, then the following hold.
\begin{enumerate}
\item[{\rm(i)}] for $\gamma\in(\gamma_{-},\gamma^{*})$, $W$ is a supersolution of \eqref{eq:P} on
$B_{r_{0}}\setminus\{0\}$, where
$r_{0}=\big[A^{p-p_{s}^{*}}(\Phi(\gamma)-\mu)\big]^{1/(P_s(\gamma^{*}-\gamma))}$,
\item[{\rm(ii)}] for $\gamma\in(\gamma^{*},\gamma_{+})$, $W$ is a supersolution of \eqref{eq:P} on
$\R^{N}\setminus \overline{B_{R_{0}}}$, where
$R_{0}=\big[A^{p_{s}^{*}-p}(\Phi(\gamma)-\mu)^{-1}\big]^{1/(P_s(\gamma-\gamma^{*}))}$,
\item[{\rm(iii)}] for $\gamma\in\big(0,\gamma_{-}\big)\cup\big(\gamma_{+},\tfrac{N}{p-1}\big)$, $W$ is a subsolution of \eqref{eq:P} on all of $\R^{N}\setminus\{0\}$, for every $A>0$.
\end{enumerate}
\end{proposition}

\begin{proof}
By \eqref{eq:homog}, $\pLap W=A^{p-1}\Phi(\gamma)|x|^{-\gamma(p-1)-sp}$, while 
$\mu|x|^{-sp}W^{p-1}=\mu A^{p-1}|x|^{-\gamma(p-1)-sp}$ and
$W^{p_{s}^{*}-1}=A^{p_{s}^{*}-1}|x|^{-\gamma(p_{s}^{*}-1)}$. The identity
\begin{equation}\label{eq:Xiid}
  \gamma(p_{s}^{*}-1)-\big[\gamma(p-1)+sp\big]
  =\gamma(p_{s}^{*}-p)-sp=P_s\Big(\gamma-\frac{N-sp}{p}\Big)=P_s(\gamma-\gamma^{*})
\end{equation}
turns $|x|^{-\gamma(p_{s}^{*}-1)}$ into $|x|^{-\gamma(p-1)-sp}|x|^{P_s(\gamma^{*}-\gamma)}$, which is
\eqref{eq:barrier}. For (i), $\gamma\in(\gamma_{-},\gamma_{+})$ gives $\Phi(\gamma)>\mu$ by
Proposition~\ref{thm:props}, and $\gamma<\gamma^{*}$ makes the exponent $P_s(\gamma^{*}-\gamma)$ positive,
so the bracket is positive for $|x|<r_{0}$. The proof of (ii) is analogous, since the exponent
$P_s(\gamma^{*}-\gamma)$ is now negative. For (iii), $\Phi(\gamma)<\mu$ makes the bracket negative for every $x$.
\end{proof}

\subsection{Comparison principle}
The comparison principle below is used to remove the loss $\eta$ in the exponents. It is based on the discrete Picone inequality of Brasco and Franzina \cite{BF14}. For $a,b>0$ and $\alpha,\beta\ge0$ it states that
\begin{equation}\label{eq:picone}
  \Jp(a-b)\left(\frac{\alpha^{p}}{a^{p-1}}-\frac{\beta^{p}}{b^{p-1}}\right)\le|\alpha-\beta|^{p}.
\end{equation}
We now state the comparison principle.
\begin{proposition}\label{thm:comp}
Let $D\subset\R^{N}$ be open with $|\R^{N}\setminus D|>0$, let $\nu\in\R$, and let $h\ge0$ be
measurable. Let
$u,v \in L^{p-1}_{sp}(\R^{N}), u,v>0$, $u$ be a weak subsolution and $v$ a weak
supersolution of $\pLap w=\nu h(x) w^{p-1}$ in $D$, and suppose $u\le v$ almost everywhere in
$\R^{N}\setminus D$. For $\delta\in(0,1]$ put
\[
  \sigma_{\delta}:=(u+\delta)^{p}-(v+\delta)^{p},\quad
  \varphi_{1,\delta}:=\frac{\sigma_{\delta}^{+}}{(u+\delta)^{p-1}},\quad
  \varphi_{2,\delta}:=\frac{\sigma_{\delta}^{+}}{(v+\delta)^{p-1}},
\]
and assume
\begin{enumerate}
\item[$(H_{1})$]\customlabel{H1}{$H_{1}$} $\varphi_{1,\delta},\varphi_{2,\delta}\in\mathcal{W}^{s,p}_{0}(D)$ for every
$\delta\in(0,1]$.
\item[$(H_{2})$]\customlabel{H2}{$H_{2}$} with $\sigma:=u^{p}-v^{p}$,
\[
  \int_{\{u>v\}}h\left[\sigma^{+}+(u+1)^{p-1}\right]\dx<\infty.
\]
\item[$(H_{3})$]\customlabel{H3}{$H_{3}$} $hu^{p-1},hv^{p-1}\in L^{1}_{\mathrm{loc}}(D)$, and for every $\delta\in(0,1]$ the functions $\varphi_{1,\delta},\varphi_{2,\delta}$ are bounded and supported in a compact subset of $D$.
\end{enumerate}
Then $u\le v$ almost everywhere in $\R^{N}$.
\end{proposition}

We need no condition on the size of $\nu$ or on $h$. The zeroth-order terms cancel because both functions satisfy the same equation. The functions $u,v$ may be unbounded, and the condition $|\R^{N}\setminus D|>0$ is only used at the end of the proof. The Hardy threshold $\bar\mu$ plays no role here.

\begin{proof}
We write $\Omega^{+}:=\{u>v\}$, which is contained in $D$ and coincides with
$\{u+\delta>v+\delta\}$ for every $\delta$. Both $\varphi_{i,\delta}$ are nonnegative and vanish
outside $\Omega^{+}$. By \eqref{H1}, they lie in $\mathcal{W}^{s,p}_{0}(D)$.
By \eqref{H3}, we can approximate them by mollifications which are uniformly bounded and supported in a fixed compact subset of $D$. The left-hand sides of \eqref{eq:weakdef} converge by H\"older's inequality, and the right-hand sides by dominated convergence, since $hu^{p-1},hv^{p-1}\in L^{1}_{\mathrm{loc}}(D)$. Hence $\varphi_{1,\delta}$ and $\varphi_{2,\delta}$ are admissible in \eqref{eq:weakdef}. Let $u,v:\R^{N}\to(0,\infty)$ with $\{u>v\}=\Omega^{+}$, put
$\sigma:=u^{p}-v^{p}$, $\psi_{1}:=\sigma^{+}u^{1-p}$, $\psi_{2}:=\sigma^{+}v^{1-p}$, and
\[
  A:=\Jp(\Delta_{xy} u) \Delta_{xy}\psi_{1},\quad B:=\Jp(\Delta_{xy} v) \Delta_{xy} \psi_{2},
  \quad \Delta_{xy}f:=f(x)-f(y).
\]
\textbf{Claim.} $A-B\ge0$ everywhere, and $A-B>0$ when one of
$x,y$ lies in $\Omega^{+}$. 

Both $A$ and $B$ are unchanged when $x$ and $y$ are
exchanged, so there are three cases. If both points lie in $\Omega^{+}$ then $\sigma^{+}=\sigma$, so $\psi_{1}=u-v^{p}u^{1-p}$ and
$\psi_{2}=u^{p}v^{1-p}-v$ at both points, and
\[
  A=|\Delta_{xy} u|^{p}-\Jp(\Delta_{xy} u)\Big(\frac{v(x)^{p}}{u(x)^{p-1}}-\frac{v(y)^{p}}{u(y)^{p-1}}\Big) \ge|\Delta_{xy} u|^{p}-|\Delta_{xy} v|^{p},
\]
by \eqref{eq:picone} with $a=u(x)$, $b=u(y)$, $\alpha=v(x)$, $\beta=v(y)$, while
\[
  B=\Jp(\Delta_{xy} v)\Big(\frac{u(x)^{p}}{v(x)^{p-1}}-\frac{u(y)^{p}}{v(y)^{p-1}}\Big)-|\Delta_{xy} v|^{p}
  \le|\Delta_{xy} u|^{p}-|\Delta_{xy} v|^{p},
\]
by \eqref{eq:picone} with $a=v(x)$, $b=v(y)$, $\alpha=u(x)$, $\beta=u(y)$. If neither point lies in
$\Omega^{+}$ then $\psi_{1}$ and $\psi_{2}$ vanish at both and $A=B=0$. If
$x\in\Omega^{+}$ and $y\notin\Omega^{+}$ then $\psi_{1}(y)=\psi_{2}(y)=0$, and using
$\Jp(t\tau)=t^{p-1}\Jp(\tau)$ for $t>0$,
\[
  A-B
  =\sigma(x)\left[\Jp\Big(1-\frac{u(y)}{u(x)}\Big)-\Jp\Big(1-\frac{v(y)}{v(x)}\Big)\right].
\]
All four values are positive, $u(x)>v(x)$ and $u(y)\le v(y)$, so $u(y)v(x)\le v(y)v(x)<v(y)u(x)$, that
is $\frac{u(y)}{u(x)}<\frac{v(y)}{v(x)}$. As $\Jp$ is strictly increasing and $\sigma(x)>0$, the
bracket is strictly positive. This completes the proof of the claim. 

We now apply the claim with $u=u+\delta$ and $v=v+\delta$, so that
$\psi_{i}=\varphi_{i,\delta}$. We write $A_{\delta},B_{\delta}$ for the resulting
$A,B$. Since $\pLap$ is unchanged by adding a constant, $u+\delta$ is a weak
subsolution of $\pLap w=\nu h u^{p-1}$ and $v+\delta$ a weak supersolution of $\pLap w=\nu h v^{p-1}$.
Testing the first with $\varphi_{1,\delta}$, the second with $\varphi_{2,\delta}$, and subtracting,
\begin{equation}\label{eq:ABdelta}
  \frac12\iint_{\R^{2N}}\big[A_{\delta}-B_{\delta}\big] K\dxy
  \le\nu\int_{\Omega^{+}} h \sigma_{\delta}^{+} g_{\delta}\dx,
  \quad
  g_{\delta}:=\Big(\frac{u}{u+\delta}\Big)^{p-1}-\Big(\frac{v}{v+\delta}\Big)^{p-1},
\end{equation}
with $K:=|x-y|^{-N-sp}$. Since $t\mapsto\big(t/(t+\delta)\big)^{p-1}$ is increasing and $u>v$ on
$\Omega^{+}$, we have  $0\le g_{\delta}\le1$, and $g_{\delta}\to0$ pointwise as $\delta\downarrow0$.
Since
$(u+\delta)^{p}-u^{p}\le p\delta(u+\delta)^{p-1}$ and $(v+\delta)^{p}\ge v^{p}$, for $\delta\le1$ we
have
\begin{gather*}
    \sigma_{\delta}^{+}\le\big(\sigma+p\delta(u+\delta)^{p-1}\big)^{+}
  \le\sigma^{+}+p \delta (u+1)^{p-1}, 
\end{gather*}
so, using $0\le g_{\delta}\le1$,
\begin{gather*}
  \left| \int_{\Omega^{+}}h\sigma_{\delta}^{+}g_{\delta} \right|
  \le\int_{\Omega^{+}}h\sigma^{+}g_{\delta}+p\delta\int_{\Omega^{+}}h (u+1)^{p-1} .
\end{gather*}
By \eqref{H2}, the first integrand is dominated by $h\sigma^{+}\in L^{1}(\Omega^{+})$ and $g_{\delta}\to0$
pointwise, so it tends to $0$. The second term is $O(\delta)$. Hence, the right-hand side of
\eqref{eq:ABdelta} tends to $0$. On the left, the integrand is nonnegative and converges pointwise to $A-B$. Therefore, Fatou's lemma gives
\begin{gather*}
   \frac12\iint_{\R^{2N}}\big[A-B\big] K\dxy\le0 .  
\end{gather*}
Suppose $|\Omega^{+}|>0$. Since $u\le v$ outside $D$ we have
$\R^{N}\setminus D\subset\R^{N}\setminus\Omega^{+}$, so $\Omega^{+}\times(\R^{N}\setminus\Omega^{+})$ has positive
measure in $\R^{2N}$. Since $A-B>0$ on $\Omega^{+}\times(\R^{N}\setminus\Omega^{+})$, and elsewhere $A-B\ge0$, the left-hand side above is strictly positive, a contradiction. Hence $|\Omega^{+}|=0$.
\end{proof}

The next lemma gives a simple condition which implies \eqref{H1}.
\begin{lemma}\label{rem:Hcheck}
Hypothesis \eqref{H1} of Proposition~\ref{thm:comp} holds if
\begin{enumerate}
\item[$(\widetilde{H_{1}})$]\customlabel{Hwide}{$\widetilde{H_{1}}$}
$\overline{\{u>v\}}\Subset V\Subset D$, $u,v\in W^{s,p}(V)$, and $\frac{u}{v}$ is bounded above on $\{u>v\}$.
\end{enumerate}
\end{lemma}
\begin{proof}
For fixed $\delta>0$ the maps
\[
  F_{1,\delta}(a,b)=\frac{\big((a+\delta)^{p}-(b+\delta)^{p}\big)^{+}}{(a+\delta)^{p-1}},\quad
  F_{2,\delta}(a,b)=\frac{\big((a+\delta)^{p}-(b+\delta)^{p}\big)^{+}}{(b+\delta)^{p-1}}
\]
are Lipschitz on the relevant range, the bound on $u/v$ being used for $F_{2,\delta}$. Hence
$\varphi_{i,\delta}=F_{i,\delta}(u,v)\in W^{s,p}(V)$ with support in $\overline{\{u>v\}}\Subset D$, so
their zero extensions lie in $\mathcal{W}^{s,p}_{0}(D)$.
\end{proof}

Removing the origin does not change the Dirichlet space. This is used in
Section~\ref{sec:app}.

\begin{lemma}\label{lem:cap}
If $sp<N$ and $r>0$ then $\mathcal{W}^{s,p}_{0}\big(B_{r}\setminus\{0\}\big)=\mathcal{W}^{s,p}_{0}(B_{r})$.
\end{lemma}

\begin{proof}
Only one inclusion needs proof. Let $\phi\in \cc(B_{r})$ and let $\eta_{\eps}$ be radial with
$\eta_{\eps}=0$ on $B_{\eps}$, $\eta_{\eps}=1$ outside $B_{2\eps}$ and $|\nabla\eta_{\eps}|\le2/\eps$. We put
$\psi:=\phi(1-\eta_{\eps})$, so that $\text{supp}(\psi) \subset \overline{B_{2\eps}}$, $\|\psi\|_{\infty}\le M:=\|\phi\|_{\infty}$
and $\|\nabla\psi\|_{\infty}\le \frac{C}{\eps}$. For $x,y\in B_{4\eps}$, we use
$|\psi(x)-\psi(y)|\le\min\{C\frac{|x-y|}{\eps},2M\}$ to get
\begin{align*}
\int_{B_{2\eps}}\int_{|y|<2\eps} \frac{|\psi(x)- \psi(y)|^{p}}{|x-y|^{N+sp}}\dy\dx & \le
\int_{B_{2\eps}} \int_{|z|<4\eps}  \frac{\min\{C \frac{|z|}{\eps},2M\}^{p}}{|z|^{N+sp}}\dz\dx \\
& \le C\eps^{N}\Big(\eps^{-p} \int_{0}^{\eps} \rho^{ p-1-sp}\d \rho+M^{p}\int_{\eps}^{4\eps}  \rho^{-1-sp}\d \rho\Big) \le C\eps^{N-sp}.
\end{align*}
For $x\in B_{2\eps}$ and $|y|>2\eps$ we have  $\psi(y)=0$ and
\begin{align*}
   \int_{B_{2\eps}}|\psi(x)|^{p}\int_{|y|>4\eps} \frac{1}{|x-y|^{N+sp}}\dy\dx\le CM^{p}\eps^{N-sp}. 
\end{align*}
Hence
$[\phi-\phi\eta_{\eps}]^{p}_{W^{s,p}}\le C\eps^{N-sp}\to0$ and
$\|\phi-\phi\eta_{\eps}\|^{p}_{L^{p}}\le M^{p}|B_{2\eps}|\to0$, because $sp<N$. Since
$\phi\eta_{\eps}$ is Lipschitz with compact support in $B_{r}\setminus\{0\}$, the claim $ \mathcal{W}^{s,p}_{0}(B_{r}) \subset \mathcal{W}^{s,p}_{0}\big(B_{r}\setminus\{0\}\big)$ follows using
mollification.
\end{proof}

\section{Application to the Hardy-Sobolev equation}\label{sec:app}
Let $0<\mu<\bar\mu$ and let $u$ be a positive solution of \eqref{eq:P}, continuous on
$\R^{N}\setminus\{0\}$. We state separately the assumptions needed near the origin and at infinity.

\begin{center}
  \begin{enumerate}
\item[$(M_{0})$]\customlabel{E0}{$M_{0}$} for every $\eta>0$ there is $C_{\eta}$ with
$u(x)\le C_{\eta}|x|^{-\gamma_{-}-\eta}$ for $|x|$ small,

\item[$(M_{\infty})$]\customlabel{E1}{$M_{\infty}$} for every $\eta>0$ there is $C_{\eta}$ with
$u(x)\le C_{\eta}|x|^{-\gamma_{+}+\eta}$ for $|x|$ large,

\item[$(M)$]\customlabel{M}{$M$} $u(x)=O\big(|x|^{-\gamma^{*}}\big)$ as $x\to0$ and as $|x|\to\infty$.
\end{enumerate}  
\end{center}

Since $\gamma_{-}<\gamma^{*}<\gamma_{+}$, \eqref{E0} gives \eqref{M} near the origin and \eqref{E1} gives \eqref{M} at infinity. So the bound near the origin uses \eqref{E0} and \eqref{M} at infinity, and the bound at infinity uses \eqref{E1} and \eqref{M} near the origin. For $p\ge2$ and positive radially symmetric nonincreasing solutions, these conditions are proved in Lemma~\ref{lem:radial} and Theorem~\ref{thm:cond}. Here we assume them.

The argument is based on one identity. Since $p_{s}^{*}-1=(p-1)+P_s$,
\begin{equation}\label{eq:rewrite}
  \pLap u=\big(\mu+|x|^{sp}u^{P_s}\big) |x|^{-sp}u^{p-1},
\end{equation}
so wherever $|x|^{sp}u^{P_s}\le\delta$ the function $u$ is a subsolution of the pure Hardy equation
\eqref{eq:pureHardy} with constant $\mu+\delta$, for which $|x|^{-\gamma_{-}(\mu+\delta)}$ is an explicit
solution.

\subsection{Characteristic upper bounds} This subsection proves the characteristic upper bound for a positive solution of \eqref{eq:P}.  
\begin{theorem}\label{thm:sharpupper}
We assume \eqref{M}. Under \eqref{E0} the first estimate below holds, and under
\eqref{E1} the second. There exist $r_{0}\in(0,1)$, $R_{0}>1$ and $C<\infty$ with
\begin{align*}
  u(x)\le C |x|^{-\gamma_{-}(\mu)}\ \ \text{for }0<|x|<r_{0},
  \; \text{ and } \;
  u(x)\le C |x|^{-\gamma_{+}(\mu)}\ \ \text{for }|x|>R_{0}.
\end{align*}
\end{theorem}

\begin{proof}
We first prove the estimate near the origin. The argument at infinity is similar after the monotonicity is reversed. By \eqref{eq:Xi}, $P_s=sp/\gamma^{*}$. Using \eqref{E0} with $\eta$ sufficiently small, we obtain 
\begin{align}\label{delta}
 \delta(r):=\sup_{0<|x|\le r}|x|^{sp}u(x)^{P_s}\le Cr^{\kappa},   
\end{align}
with
$\kappa=sp\big(1-\gamma_{-}/\gamma^{*}\big)-\eta P_s$.  Since $\gamma_{-}<\gamma^{*}$ due to
Corollary~\ref{cor:roots}, we can choose $\eta$ small enough such that $\kappa>0$.

We fix $\rho\in(0,1)$, $r_{k}=\rho^{k}r_{0}$, $\mu_{k}=\mu+\delta(r_{k-1})$ and
$\gamma_{k}=\gamma_{-}(\mu_{k})$, where we choose $r_0>0$ small so that $\mu_{1}<\bar\mu$. We observe that, near the origin, the sequence $\{ \gamma_{k}\}$ is nonincreasing. 
By Corollary~\ref{cor:roots}, $\gamma_{-}'(\nu)=1/\Phi'(\gamma_{-}(\nu))$, finite and positive by
Proposition~\ref{thm:props}-(iii). On the compact interval $[\mu,\mu_{1}]$ it is bounded by some $L$, so using mean value theorem and \eqref{delta}, 
\begin{equation}\label{eq:gammaloss}
  0\le\gamma_{k}-\gamma_{-}\le L(\mu_k-\mu) =  L \delta(r_{k-1})\le L C r_{0}^{\kappa} \rho^{(k-1)\kappa}.
\end{equation}
Since $\Phi(\gamma_{k})=\mu_{k}$, the function $A|x|^{-\gamma_{k}}$ is an explicit solution of
\eqref{eq:pureHardy} with constant $\mu_{k}$ for every $A>0$. We consider
$$\Gamma_{\infty}:=\sup_{k}\sup_{|x|\ge r_{0}}|x|^{\gamma_{k}}u(x),$$
which is positive and finite by \eqref{M} at $|x|$ near infinity, because $\gamma_{k}<\gamma^{*}$. We set
\[
  A_{1}:=2\Gamma_{\infty},\; \text{ and } \; A_{k+1}:=\big(1+2^{-k}\big) \frac{A_{k}}{r_{k}^{\gamma_{k}-\gamma_{k+1}}} .
\]
{\noi\textit{Claim.} For every $k\ge1$,}
\begin{equation}\label{eq:induct}
  u(x)\le A_{k}|x|^{-\gamma_{k}}, \text{ for }x\neq0 .
\end{equation}

We argue by induction on $k$. In each step we apply Proposition~\ref{thm:comp} on a ball. By Lemma~\ref{lem:cap} we may use the ball instead of the punctured ball. Hypotheses \eqref{H1} and \eqref{H2} are checked at the end.

Let $k=1$. By the definition of $\Gamma_{\infty}$, $u\le\Gamma_{\infty}|x|^{-\gamma_{1}}<A_{1}|x|^{-\gamma_{1}}$ on
$\R^{N}\setminus B_{r_{0}}$. By \eqref{eq:rewrite}, $u$ is a subsolution of \eqref{eq:pureHardy} with
constant $\mu_{1}$ on $B_{r_{0}}\setminus\{0\}$, and $A_{1}|x|^{-\gamma_{1}}$ is a homogeneous solution of the
same equation. Proposition~\ref{thm:comp} therefore gives \eqref{eq:induct} for $k=1$.

Now assume \eqref{eq:induct} for $k$. For $|x|\ge r_{k}$, using
$\gamma_{k+1}\le\gamma_{k}$ and $|x|\ge r_{k}$,
\[
  u(x)\le A_{k}|x|^{\gamma_{k+1}-\gamma_{k}}|x|^{-\gamma_{k+1}}
       \le A_{k}r_{k}^{-(\gamma_{k}-\gamma_{k+1})}|x|^{-\gamma_{k+1}}
       =\frac{A_{k+1}}{1+2^{-k}} |x|^{-\gamma_{k+1}}
       <A_{k+1}|x|^{-\gamma_{k+1}} ,
\]
so the ordering holds, strictly, on all of $\R^{N}\setminus B_{r_{k}}$. On $B_{r_{k}}\setminus\{0\}$, the
function $u$ is a subsolution of \eqref{eq:pureHardy} with coefficient $\mu_{k+1}$. The function
$A_{k+1}|x|^{-\gamma_{k+1}}$ is a homogeneous solution of the same equation. By Lemma~\ref{lem:cap}, Proposition~\ref{thm:comp} can be applied with $D=B_{r_{k}}$. Hence, \eqref{eq:induct} holds for $k+1$.

It remains to check \eqref{H1} and \eqref{H2}. We write $v:=A_{k+1}|x|^{-\gamma_{k+1}}$. Choosing
$\eta<\gamma_{k+1}-\gamma_{-}$ in \eqref{E0} gives $u\le C_{\eta}|x|^{-\gamma_{-}-\eta}$ near the origin, so
\[
  \frac{u(x)}{v(x)}\le\frac{C_{\eta}}{A_{k+1}} |x|^{\gamma_{k+1}-\gamma_{-}-\eta}
  \xrightarrow[\ x\to0\ ]{}0 ,
\]
so $\{u>v\}$ has positive distance from the origin and hence $u/v$ is bounded above on it. The induction step
also gives
\[
  \frac{u}{v}\le\frac{1}{1+2^{-k}} < 1\quad\text{on }\R^{N}\setminus B_{r_{k}}.
\]
By continuity, the set $\{u>v\}$ stays away from $\partial B_{r_{k}}$. The estimate near
the origin also keeps this set away from the origin. Hence, $\overline{\{u>v\}}\Subset B_{r_{k}}\setminus\{0\}$. This verifies \eqref{Hwide}, and Lemma~\ref{rem:Hcheck} gives \eqref{H1}. 
Since $u$ and $v$ are continuous on $\R^{N}\setminus\{0\}$, they are bounded on the compact set $\overline{\{u>v\}}\subset B_{r_k}\setminus\{0\}$, and $v>0$ there. Hence $\varphi_{1,\delta}\le pu$ and $\varphi_{2,\delta}\le p\big((u+1)/v\big)^{p-1}u$ are bounded with compact support in $B_{r_k}\setminus\{0\}$. Also $|x|^{-sp}u^{p-1},|x|^{-sp}v^{p-1}\in L^{1}_{\mathrm{loc}}(B_{r_k}\setminus\{0\})$. So \eqref{H3} holds. Next, we verify \eqref{H2}. We note that $\sigma^{+}\le u^{p}$. Therefore, using the Hardy inequality \eqref{eq:FS},
$$\int_{\R^N} h\sigma^{+} \dx \le\int_{\R^N}|x|^{-sp}u^{p} \dx < \infty, \text{ since } u\in\D^{s,p}(\R^{N}).$$ 
Moreover, the Hardy inequality gives
$\int_{\{u>v\}}|x|^{-sp}\sigma^+\dx<\infty$. Since $\{u>v\}\Subset B_{r_k}\setminus\{0\}$ and $u$ is continuous, we have 
$$\int_{\{u>v\}}|x|^{-sp}(u+1)^{p-1} < \infty.$$ Hence \eqref{H2} also holds. This proves the claim.

Writing $\log(1/r_{k})=k\log\left(1/\rho\right)+\log(1/r_{0})$ and using \eqref{eq:gammaloss}, we get
\begin{align*}
  &\log A_{K}\le C_{\Gamma}
  +LCr_{0}^{\kappa}\Big(\log\tfrac1\rho\sum_{k\ge1}k\rho^{(k-1)\kappa}
  +\log\tfrac1{r_{0}}\sum_{k\ge1}\rho^{(k-1)\kappa}\Big), \\
  &\text{ with } C_{\Gamma}:=\log(2\Gamma_{\infty})+\sum_{k\ge1}\log\big(1+2^{-k}\big),
\end{align*}
where we used $\gamma_{k}-\gamma_{k+1}\le\gamma_{k}-\gamma_{-}$. Here $C_{\Gamma}<\infty$ because
$\log(1+2^{-k})\le2^{-k}$, and the two remaining series converge because $\rho^{\kappa}<1$. Hence
$A_{\infty}:=\sup_{k}A_{k}<\infty$. For
$r_{k+1}\le|x|\le r_{k}$,
\[
  |x|^{\gamma_{-}}u(x)\le A_{\infty}
  \exp\Big[(\gamma_{k+1}-\gamma_{-})\Big((k+1)\log\tfrac1\rho+\log\tfrac1{r_{0}}\Big)\Big],
\]
and by \eqref{eq:gammaloss} the exponent is bounded in $k$, and tends to $0$.

At infinity, $\mu_k\downarrow\mu,\, \gamma_k=\gamma_+(\mu_k)\uparrow\gamma_+(\mu)$. We repeat the comparison on $\R^N\setminus\overline{B}_{R_k}$. We redefine $\mu_{k}$,
$\gamma_{k}$ and $A_{k}$ accordingly. We put $\delta_{\infty}(R):=\sup_{|x|\ge R}|x|^{sp}u^{P_s}$. Since $\gamma_{+}>\gamma^{*}$, hypothesis \eqref{E1} gives $\delta_{\infty}(R)\le CR^{-\kappa_{\infty}}$ with
$\kappa_{\infty}=sp\big(\gamma_{+}/\gamma^{*}-1\big)-\eta P_s>0$ for $\eta$ small, so by
\eqref{eq:rewrite} the function $u$ is a subsolution of \eqref{eq:pureHardy} with constant
$\mu+\delta_{\infty}(R)$ on $\{|x|>R\}$. We take $R_{k}:=\rho^{-k}R_{0}$, so that $R_{k}\uparrow\infty$, and
set $\mu_{k}:=\mu+\delta_{\infty}(R_{k-1})$ and $\gamma_{k}:=\gamma_{+}(\mu_{k})$. Then $\mu_{k}\downarrow\mu$,
and since $\gamma_{+}$ is \emph{decreasing} in $\mu$ by Corollary~\ref{cor:roots}, the exponents are
nondecreasing, $\gamma_{k}\uparrow\gamma_{+}(\mu)$, with
$0\le\gamma_{+}(\mu)-\gamma_{k}\le L_{\infty}\delta_{\infty}(R_{k-1})$ where
$L_{\infty}:=\max_{[\mu,\mu_{1}]}|\gamma_{+}'|<\infty$, again by Proposition~\ref{thm:props}-(iii). We now use the
comparison on the exterior domains $\{|x|>R_{k}\}$. We choose 
$A_{1}:=2\Gamma_{\infty}$ and define $A_{k+1}:=(1+2^{-k})A_{k}R_{k}^{\gamma_{k+1}-\gamma_{k}}$. Here
$\Gamma_{\infty}:=\sup_{k}\sup_{|x|\le R_{0}}|x|^{\gamma_{k}}u(x)$ is finite by the assumption \eqref{M} at the
origin since $\gamma_{k}>\gamma^{*}$. Using $\gamma_{k+1}>\gamma_{k}$ and $|x|\le R_{k}$, we have 
$|x|^{\gamma_{k+1}-\gamma_{k}}\le R_{k}^{\gamma_{k+1}-\gamma_{k}}$, which gives the required induction step. 
The set $\{u>v\}$ is bounded. As before, it stays away from $\partial B_{R_k}$, so \eqref{H1}, \eqref{H2} and \eqref{H3} hold. Indeed, by \eqref{E1}, if  $\eta<\gamma_{+}-\gamma_{k+1}$, then
$u/v\le(C_{\eta}/A_{k+1})|x|^{\gamma_{k+1}-\gamma_{+}+\eta}\to0$ as $|x|\to\infty$. Also,
$\log R_{k}=k\log\left( \frac{1}{\rho} \right)+\log R_{0}$. Therefore, the same estimate as before applies with $\Gamma_{0}$ in place of $\Gamma_{\infty}$. This gives the desired bound at infinity.
\end{proof}

\subsection{Radially symmetric solutions}\label{sec:radial}
In this subsection we consider positive radially symmetric nonincreasing solutions of \eqref{eq:P}. We prove \eqref{M} and an integrability improvement in the spirit of \cite{BMS16}. We also show that \eqref{E0} and \eqref{E1} are equivalent to one integrability statement. We recall $[u]_{s,p}$ from \eqref{eq:semin} and the Sobolev constant $S$ from Section~\ref{sec:intro}. We first record some simple consequences of radial monotonicity.

\begin{lemma}\label{lem:radial}
Let $U\ge0$ be radially symmetric and nonincreasing in $|x|$.
\begin{enumerate}
\item[{\rm(i)}] If $U\in L^{q}(\R^{N})$ with $q>0$, then
$U(x)\le|B_{1}|^{-1/q}\|U\|_{L^{q}(\R^{N})}|x|^{-N/q}$ for every $x\neq0$.
\item[{\rm(ii)}] If $U\in L^{p^{*}_{s}}(\R^{N})$, then $|x|^{\gamma^{*}}U(x)\to0$ as $|x|\to0$ and as
$|x|\to\infty$. In particular $U$ satisfies \eqref{M}, and for every $\lambda>0$ the set
\begin{equation}\label{eq:Omlam}
  \Omega_{\lambda}:=\big\{x\neq0:\ |x|^{sp}U(x)^{P_s}>\lambda\big\}
\end{equation}
is contained in an annulus $\{r_{1}\le|x|\le r_{2}\}$ with $0<r_{1}\le r_{2}<\infty$, and
$0<c_{\lambda}\le U\le C_{\lambda}<\infty$ on $\Omega_{\lambda}$.
\item[{\rm(iii)}] If $U\in L^{p^{*}_{s}}(\R^{N})$, then $U\in L^{p-1}_{sp}(\R^{N})$ and
$U\in L^{\infty}_{\mathrm{loc}}(\R^{N}\setminus\{0\})$.
\end{enumerate}
\end{lemma}

\begin{proof}
(i) Since $U\ge U(x)$ on $B_{|x|}$, we get $|B_{1}||x|^{N}U(x)^{q}\le\int_{B_{|x|}}U^{q}\le\|U\|^{q}_{L^{q}}$.

(ii) The same computation with $q=p^{*}_{s}$ and $N/p^{*}_{s}=\gamma^{*}$ gives
$$|x|^{\gamma^{*}}U(x)\le\left(|B_{1}|^{-1}\int_{B_{|x|}}U^{p^{*}_{s}}\right)^{\frac{1}{p^{*}_{s}}},$$ and the
right-hand side tends to $0$ as $x\to0$. For $|x|$ large we use
$U\ge U(x)$ on $B_{|x|}\setminus B_{|x|/2}$, so that
\[
  \big(1-2^{-N}\big)|B_{1}||x|^{N}U(x)^{p^{*}_{s}}
  \le\int_{\R^{N}\setminus B_{|x|/2}}U^{p^{*}_{s}}\xrightarrow[\ |x|\to\infty\ ]{}0 .
\]
Since $\gamma^{*}P_s=sp$ by \eqref{eq:Xi}, we have $|x|^{sp}U(x)^{P_s}=\big(|x|^{\gamma^{*}}U(x)\big)^{P_s}$,
so $\Omega_{\lambda}$ is bounded and stays away from the origin. Radial monotonicity bounds $U$
from above on $\Omega_{\lambda}$ by $C_{\lambda}:=U(r_{1}/2)$, which is finite by (i). For the lower bound we
use the definition of $\Omega_{\lambda}$ directly. There $U^{P_s}>\lambda|x|^{-sp}\ge\lambda r_{2}^{-sp}$, so
$U\ge c_{\lambda}:=\big(\lambda r_{2}^{-sp}\big)^{1/P_s}>0$. No positivity assumption on $U$ is used.

(iii) By (i) with $q=p^{*}_{s}$ we have $U(x)\le C|x|^{-\gamma^{*}}$, which is locally bounded away from
the origin. Moreover $\gamma^{*}(p-1)<N$ and $\gamma^{*}(p-1)+sp>0$, so
\begin{align*}
    \int_{\R^{N}}(1+|y|)^{-N-sp}|y|^{-\gamma^{*}(p-1)}\dy<\infty \Longrightarrow U\in L^{p-1}_{sp}(\R^{N}), 
\end{align*}
as required.
\end{proof}

\subsubsection{\textbf{An integrability improvement}} We now prove Theorem~\ref{thm:integ}. It is the analogue for \eqref{eq:P} of the integrability improvement of \cite{BMS16}. The Hardy term is estimated by means of \eqref{eq:FS}.

\begin{proof}[{\bf Proof of Theorem~\ref{thm:integ}}]
For $\beta>-\tfrac1p$ we have
$h'(\beta)=p(1-p)\beta(1+\beta)^{-p-1}$, so $h'$ has the sign of $\beta(1-p)$. Since $p>1$ this gives
the stated monotonicity, and the two limits are immediate. As $h(0)=1>\mu/\bar\mu$, the
superlevel set is an open interval containing $0$.

We fix $\beta\in(\beta_{-},\beta_{+})$ and set $c_{\beta}:=h(\beta)^{1/p}$. We note $c_{\beta}\le1$ by
Bernoulli's inequality $(1+\beta)^{p}\ge1+p\beta$. For $0<m\le M<\infty$ let
\[
  T(t):=\min\{\max\{t,m\},M\},\quad G(t):=tT(t)^{p\beta},\quad w:=UT(U)^{\beta}.
\]
Then $G$ is nondecreasing, locally Lipschitz and $G(0)=0$. Furthermore, $T(t)^{p\beta}\le M^{p\beta}$ if
$\beta\ge0$ and $T(t)^{p\beta}\le m^{p\beta}$ if $\beta<0$, so $G$ is Lipschitz on $[0,\infty)$ and
$\varphi:=G(U)\in\Dsp$. Both sides of \eqref{eq:weakdef} are finite for this $\varphi$, because
\begin{equation}\label{eq:idw}
  U^{p-1}G(U)=U^{p}T(U)^{p\beta}=w^{p},\quad
  U^{p^{*}_{s}-1}G(U)=U^{P_s}w^{p},
\end{equation}
while $\int_{\R^N}|x|^{-sp}U^{p}\dx<\infty$ by \eqref{eq:FS} and $\int_{\R^N} U^{p^{*}_{s}}\dx<\infty$. By density,
$\varphi$ is admissible in \eqref{eq:weakdef}, so
\begin{equation}\label{eq:testedeq}
  \frac12\iint_{\R^{2N}}\frac{\Jp\big(U(x)-U(y)\big)\big(G(U(x))-G(U(y))\big)}{|x-y|^{N+sp}}\dxy
  =\mu\int_{\R^{N}}|x|^{-sp}w^{p}\dx+\int_{\R^{N}}U^{P_s}w^{p}\dx .
\end{equation}

\noi\emph{\textbf{Claim 1.}} The left-hand side of \eqref{eq:testedeq} is at least $\tfrac12c_{\beta}^{p}[w]^{p}_{s,p}$.

Indeed, let $g(t):=\int_{0}^{t}G'(\tau)^{1/p}\d\tau$.
For $a>b\ge0$, applying H\"older inequality,
\[
  g(a)-g(b)=\int_{b}^{a}(G')^{1/p}\le\Big(\int_{b}^{a}G'\Big)^{1/p}(a-b)^{1-\frac1p}
  =\big(G(a)-G(b)\big)^{1/p}(a-b)^{\frac{p-1}{p}} ,
\]
that is $\Jp(a-b)\big(G(a)-G(b)\big)\ge|g(a)-g(b)|^{p}$, and the same holds for $a<b$ by symmetry.
Next, with $W(t):=t\,T(t)^{\beta}$ we have $W'=T^{\beta}$ on $(0,m)\cup(M,\infty)$ and
$W'=(1+\beta)t^{\beta}$ on $(m,M)$, while $G'=T^{p\beta}$ on $(0,m)\cup(M,\infty)$ and
$G'=(1+p\beta)t^{p\beta}$ on $(m,M)$. Hence $g'=(G')^{1/p}\ge c_{\beta}W'$ everywhere, and since $g$ and
$W$ are nondecreasing, $|g(a)-g(b)|\ge c_{\beta}|W(a)-W(b)|$. As $w=W(U)$, Claim 1 follows.
\medskip

\noi\emph{\textbf{Claim 2.}} Let $\lambda>0$, let $\Omega_{\lambda}$, $c_{\lambda}$, $C_{\lambda}$ be as in Lemma~\ref{lem:radial}-(ii), and assume $m\le c_{\lambda}$ and $M\ge C_{\lambda}$. Then
\[
  \int_{\R^{N}}U^{P_s}w^{p}\dx
  \le\lambda\int_{\R^{N}}|x|^{-sp}w^{p}\dx+K(\lambda),
\]
where $K(\lambda)<\infty$ is given in \eqref{eq:Klam} and does not depend on $m$ and $M$.

Indeed, outside $\Omega_{\lambda}$ we have $U^{P_s}\le\lambda|x|^{-sp}$, and $T(U)=U$ on $\Omega_{\lambda}$, so
$w=U^{1+\beta}\le C_{\lambda}^{1+\beta}$ there, and therefore
\begin{equation}\label{eq:Klam}
  \int_{\Omega_{\lambda}}U^{P_s}w^{p}\dx\le C_{\lambda}^{P_s+p(1+\beta)}|\Omega_{\lambda}|=:K(\lambda)<\infty ,
\end{equation}
which proves Claim 2.

Combining Claims 1 and 2 with \eqref{eq:testedeq} and then applying
\eqref{eq:FS} to $w\in\Dsp$, together with $C_H=2\bar\mu$, we obtain
\[
  \frac{c_{\beta}^{p}}{2}[w]^{p}_{s,p}
  \le(\mu+\lambda)\int_{\R^{N}}|x|^{-sp}w^{p}\dx+K(\lambda)
  \le\frac{\mu+\lambda}{2\bar\mu}[w]^{p}_{s,p}+K(\lambda).
\]
Since $c_{\beta}^{p}=h(\beta)>\mu/\bar\mu$, we may fix
$\lambda\in\big(0,c^{p}_{\beta}\bar\mu-\mu\big)$, and then
\[
  [w]^{p}_{s,p}\le\frac{2K(\lambda)}{c_{\beta}^{p}-(\mu+\lambda)/\bar\mu}=:\Lambda<\infty ,
\]
with $\Lambda$ independent of $m$ and $M$. By the definition of $S$,
$\|w\|^{p}_{L^{p^{*}_{s}}}\le\Lambda/S$.

Finally, letting $m\downarrow0$ and $M\uparrow\infty$ we have
$w=U\,T(U)^{\beta}\to U^{1+\beta}$ pointwise in $\R^{N}\setminus\{0\}$, so Fatou's lemma gives
\begin{align*}
\int_{\R^{N}}U^{(1+\beta)p^{*}_{s}}\dx \le \left(\frac{\Lambda}{S}\right)^{\frac{p^{*}_{s}}{p}}.    
\end{align*}
This is
\eqref{eq:Lq} for $q=(1+\beta)p^{*}_{s}$, and $\beta$ ranges over $(\beta_{-},\beta_{+})$.
\end{proof}

\subsubsection{\textbf{The estimates for radial solutions}} We can now state the conclusion.

\begin{theorem}\label{thm:radialmain}
Let $0<\mu<\bar\mu$ and let $U\in\Dsp$ be a positive radially symmetric nonincreasing weak
solution of \eqref{eq:P}. Then $U\in \C(\R^{N}\setminus\{0\})$, $U$ satisfies \eqref{M}, and for every
$q\in(q_{-},q_{+})$,
\begin{equation}\label{eq:radialconc}
  U(x)\le|B_{1}|^{-\frac{1}{q}}\|U\|_{L^{q}(\R^{N})}|x|^{-\frac{N}{q}}\quad\text{for every }x\neq0 ,
\end{equation}
with $q_{\pm}$ as in Theorem~\ref{thm:integ}. In particular there are $\gamma_{1}<\gamma^{*}<\gamma_{1}'$
and $C<\infty$ with
\[
  U(x)\le C|x|^{-\gamma_{1}}\ \text{ near the origin},\quad
  U(x)\le C|x|^{-\gamma_{1}'}\ \text{ at infinity},
\]
and $\gamma_{1}$ may be taken anywhere in $\big(\frac{\gamma^{*}}{1+\beta_{+}},\gamma^{*}\big)$ and
$\gamma_{1}'$ anywhere in $\big(\gamma^{*},\frac{\gamma^{*}}{1+\beta_{-}}\big)$. If $U$ satisfies \eqref{E0},
then $U(x)\le C|x|^{-\gamma_{-}(\mu)}$ near the origin. Also, if $U$ satisfies \eqref{E1}, then
$U(x)\le C|x|^{-\gamma_{+}(\mu)}$ at infinity.
\end{theorem}

\begin{proof}
Lemma~\ref{lem:radial}-(ii),(iii) gives \eqref{M}, $U\in L^{p-1}_{sp}(\R^{N})$ and
$U\in L^{\infty}_{\mathrm{loc}}(\R^{N}\setminus\{0\})$. Hence the right-hand side of \eqref{eq:P} lies
in $L^{\infty}_{\mathrm{loc}}(\R^{N}\setminus\{0\})$, and the interior H\"older estimates of
\cite{DKP16} give $U\in \C(\R^{N}\setminus\{0\})$. Estimate \eqref{eq:radialconc} is
Theorem~\ref{thm:integ} combined with Lemma~\ref{lem:radial}-(i). Taking $q\in(p^{*}_{s},q_{+})$ gives
$N/q<N/p^{*}_{s}=\gamma^{*}$ and taking $q\in(q_{-},p^{*}_{s})$ gives $N/q>\gamma^{*}$, with
$N/q=\gamma^{*}/(1+\beta)$ for $q=(1+\beta)p^{*}_{s}$. The last two statements are
Theorem~\ref{thm:sharpupper}, whose remaining hypothesis \eqref{M} has just been verified.
\end{proof}

\begin{remark}\label{rem:barrier}
The bound $U\le C|x|^{-\gamma_1}$ with $\gamma_1<\gamma^*$ cannot be improved to \eqref{E0} by comparison alone. For a barrier $A|x|^{-\gamma'}$, the ratio $U/(A|x|^{-\gamma'})$ is bounded near $0$ only if $\gamma'\ge\gamma_1$, and then the conclusion is weaker than what we know. The same happens at infinity. For this reason \eqref{E0} and \eqref{E1} are proved in Theorem~\ref{thm:cond} by the ground-state substitution.
\end{remark}

Finally, we record the integrability that corresponds to the two characteristic exponents.

\begin{corollary}\label{cor:sharpint}
Let $U$ be as in Theorem~\ref{thm:radialmain} and assume in addition \eqref{E0} and \eqref{E1}. Then
\begin{equation}\label{eq:sharpint}
  U\in L^{q}(\R^{N})\quad\text{for every } q \text{ with } \frac{N}{\gamma_{+}(\mu)}<q<\frac{N}{\gamma_{-}(\mu)} .
\end{equation}
Conversely, if $U$ is radially symmetric, nonincreasing and satisfies \eqref{eq:sharpint}, then $U$
satisfies \eqref{E0}, \eqref{E1} and \eqref{M}.
\end{corollary}

\begin{proof}
By Theorem~\ref{thm:radialmain} and \eqref{E0}, \eqref{E1}, we have
$U\le C|x|^{-\gamma_{-}}$ on $B_{r_{0}}\setminus\{0\}$ and
$U\le C|x|^{-\gamma_{+}}$ on $\R^{N}\setminus B_{R_{0}}$, while $U$ is bounded on the intermediate
annulus. The first bound is in $L^{q}(B_{r_{0}})$ if and only if $q\gamma_{-}<N$, and the second is in
$L^{q}(\R^{N}\setminus B_{R_{0}})$ if and only if $q\gamma_{+}>N$. This is \eqref{eq:sharpint}.

Conversely, assume \eqref{eq:sharpint}. Given $\eta>0$ with $\eta<\gamma^{*}-\gamma_{-}$, apply
Lemma~\ref{lem:radial}-(i) with $q=N/(\gamma_{-}+\eta)$, which lies in the range \eqref{eq:sharpint} due to Corollary \ref{cor:roots}.
This gives $U(x)\le C_{\eta}|x|^{-\gamma_{-}-\eta}$ for every $x\neq0$, hence \eqref{E0}. Taking
$q=N/(\gamma_{+}-\eta)$ for $0<\eta<\gamma_{+}-\gamma^{*}$ gives $U(x)\le C_{\eta}|x|^{-\gamma_{+}+\eta}$ for every $x\neq0$,
hence \eqref{E1}, and $q=p^{*}_{s}$ gives \eqref{M}.
\end{proof}

\section{Ground-state representations associated with the Hardy exponents}\label{sec:gsrep}
{For one value of $\mu$, Corollary~\ref{cor:roots} gives two power solutions $|x|^{-\gamma_\pm}$ of \eqref{eq:pureHardy}, hence two ground-state substitutions and two remainder inequalities. Throughout this section $p\ge2$ (see Remark~\ref{rem:gsp2}). We recall $\mathcal D_\mu$ and $\mathcal E_\gamma$ from \eqref{eq:intro-defect} and \eqref{eq:intro-Egamma}. By \eqref{eq:FS} and Proposition~\ref{thm:props}-(iv), $\mathcal D_\mu\ge0$ for $0<\mu\le\bar\mu$. For $0<\gamma<\gamma_c$ put $\omega_\gamma(x):=|x|^{-\gamma}$, so that $\pLap\omega_\gamma=\Phi(\gamma)|x|^{-sp}\omega_\gamma^{p-1}$ in $\R^N\setminus\{0\}$ by \eqref{eq:homog}. The substitution is $u=\omega_{\gamma}v$, that is $v(x)=|x|^{\gamma}u(x)$.}

\subsection{Reduction to a pointwise inequality} We now write the whole computation as a pointwise inequality in four real
variables. For $a,b\in\R$ and $c,d>0$ set
\begin{equation}\label{eq:bracket}
  \mathfrak{B}(a,b,c,d):=|a-b|^{p}-\Jp(c-d)\left(\frac{|a|^{p}}{c^{p-1}}-\frac{|b|^{p}}{d^{p-1}}\right).
\end{equation}

We first write the Hardy energy as a double integral of $\mathfrak{B}$.
\begin{lemma}\label{lem:gsreduce}
Let $0<\gamma<\gamma_{c}$ and $u\in\cc(\R^{N}\setminus\{0\})$. Then
\begin{equation}\label{eq:gsreduce}
  \mathcal{D}_{\Phi(\gamma)}(u)
  =\iint_{\R^{2N}}\frac{\mathfrak{B}\big(u(x),u(y),\omega_{\gamma}(x),\omega_{\gamma}(y)\big)}{|x-y|^{N+sp}}\dxy ,
\end{equation}
and the two double integrals into which the right-hand side splits,  converge absolutely.
\end{lemma}

\begin{proof}
We set $\nu:=\Phi(\gamma)$ and $\varphi:=|u|^{p}\omega_{\gamma}^{1-p}=|u|^{p}|x|^{\gamma(p-1)}$,
which is Lipschitz with compact support in $\R^{N}\setminus\{0\}$. By the power-function formula of Del Pezzo and Quaas
\cite[Theorem~1.1]{DPQ25}, restated as Proposition~\ref{prop:rep} with the normalization fixed in
\eqref{eq:op}--\eqref{eq:weakdef}, the function $\omega_{\gamma}$ satisfies \eqref{eq:homog}, that is
$\pLap\omega_{\gamma}=\nu|x|^{-sp}\omega_{\gamma}^{p-1}$ in $\R^{N}\setminus\{0\}$. Lemma~\ref{lem:pv}
shows that the principal value used there agrees with the paired representation. Symmetrizing in
$(x,y)$ as in \eqref{eq:op}--\eqref{eq:weakdef} turns this into the weak formulation, the double
integral so obtained being absolutely convergent, as we check below. Being Lipschitz with compact
support in $\R^{N}\setminus\{0\}$, $\varphi$ lies in $\mathcal{W}^{s,p}_{0}(\R^{N}\setminus\{0\})$ and is
admissible by density. Testing with $\varphi$ gives
\[
  \iint_{\R^{2N}}\frac{\Jp\big(\omega_{\gamma}(x)-\omega_{\gamma}(y)\big)\big(\varphi(x)-\varphi(y)\big)}{|x-y|^{N+sp}}\dxy
  =2\nu\int_{\R^{N}}|x|^{-sp}\omega_{\gamma}^{p-1}\varphi\dx
  =2\nu\int_{\R^{N}}\frac{|u|^{p}}{|x|^{sp}}\dx .
\]
Absolute convergence of this integral is checked region by region. The integrand is symmetric in
$(x,y)$. So in each case below we may assume that $x$ is the variable in the stated position. The integrand
vanishes unless at least one of $x,y$ lies in $\operatorname{supp}\varphi$.

\noi (i) If $|x-y|$ is small, then both variables lie in a fixed neighbourhood of
$\operatorname{supp}\varphi$, on which $\omega_{\gamma}$ and $\varphi$ are Lipschitz, so both differences
are $O(|x-y|)$. The integrand is then $O\big(|x-y|^{p-sp-N}\big)$, and $\int_{0}^{1}r^{p-sp-1}\dr<\infty$
because $sp<p$.

\noi (ii) If $|x-y|$ is bounded below and $x$ is close to the origin, then $\varphi(x)=0$, so the
integrand vanishes unless $y\in\operatorname{supp}\varphi$. Moreover, for such $y$ both $\omega_{\gamma}(y)$ and
$|x-y|^{-N-sp}$ are bounded, so the integrand is
$O\big(\omega_{\gamma}(x)^{p-1}\big)=O\big(|x|^{-\gamma(p-1)}\big)$. This is integrable near the origin
because $\gamma<\gamma_{c}$ forces $\gamma(p-1)<N-sp<N$.

\noi (iii) If $|x|$ is large, then again $\varphi(x)=0$ and $y\in\operatorname{supp}\varphi$, while
$|\omega_{\gamma}(x)-\omega_{\gamma}(y)|$ is bounded and $|x-y|^{-N-sp}\le C|x|^{-N-sp}$, which is
integrable at infinity.

\noi On the remaining set every factor is bounded and the domain has finite measure. Therefore, 
\begin{align*}
    \iint_{\R^{2N}}\frac{\Jp\big(\omega_{\gamma}(x)-\omega_{\gamma}(y)\big)\big(\varphi(x)-\varphi(y)\big)}{|x-y|^{N+sp}}\dxy< \infty. 
\end{align*}
Further, since
$u\in\cc(\R^{N}\setminus\{0\})$, $$\iint_{\R^{2N}} \frac{|u(x)-u(y)|^{p}}{|x-y|^{N+sp}}\dxy=[u]^{p}_{s,p}<\infty.$$ Subtracting term-wise and using
$\varphi=|u|^{p}\omega_{\gamma}^{1-p}$ gives \eqref{eq:gsreduce}.
\end{proof}

Next, we bound $\mathfrak{B}$ from below.
\begin{lemma}\label{lem:gspoint}
Let $a,b\in\R$ and $c,d>0$.
\begin{enumerate}
\item[{\rm(i)}] If $p=2$ then
$\mathfrak{B}(a,b,c,d)=cd\,\big|\frac ac-\frac bd\big|^{2}$.
\item[{\rm(ii)}] Let $p\ge2$ and set
\begin{equation}\label{eq:cp}
  c_{p}:=\min_{0<\tau<\frac12}\Big\{(1-\tau)^{p}-\tau^{p}+p\tau^{p-1}\Big\}\in(0,1] .
\end{equation}
Then
\begin{equation}\label{eq:gspoint}
  \mathfrak{B}(a,b,c,d) \ge c_{p}(cd)^{\frac{p}{2}}\Big|\frac ac-\frac bd\Big|^{p} ,
\end{equation}
\end{enumerate}
\end{lemma}

\begin{proof}
(i) The identity follows by direct calculation. 


(ii) For $p=2$ the expression in \eqref{eq:cp} equals $1-2\tau+\tau^{2}-\tau^{2}+2\tau=1$ for every
$\tau$, so $c_{2}=1$, consistent with (i). By \cite[Lemma~2.6, (2.16)]{FS08}, for $p\ge2$,
$0\le t\le1$ and $A\in\R$,
\begin{equation}\label{eq:FSpoint}
  |A-t|^{p} \ge (1-t)^{p-1}\big(|A|^{p}-t\big)+c_{p}t^{\frac{p}{2}}|A-1|^{p},
\end{equation}
with $c_{p}$ given by \eqref{eq:cp}, which is \cite[(1.13)]{FS08}.

Both $\mathfrak{B}$ and the right-hand side of \eqref{eq:gspoint} are unchanged under the exchange
$(a,c)\leftrightarrow(b,d)$, so we may assume $c\ge d$. We put
\[
  t:=\frac dc\in(0,1],\quad A:=\frac ac,\quad B:=\frac bd ,
\]
so that $a=cA$ and $b=ctB$. Then $|a-b|^{p}=c^{p}|A-tB|^{p}$, $\Jp(c-d)=c^{p-1}(1-t)^{p-1}$,
$|a|^{p}c^{1-p}=c|A|^{p}$ and $|b|^{p}d^{1-p}=ct|B|^{p}$, therefore
\begin{equation}\label{eq:Bnorm}
  \mathfrak{B}(a,b,c,d)=c^{p}\Big(|A-tB|^{p}-(1-t)^{p-1}\big(|A|^{p}-t|B|^{p}\big)\Big).
\end{equation}
Suppose first $B\neq0$. Since
\[
  |A-tB|^{p}=|B|^{p}\Big|\frac AB-t\Big|^{p},\quad
  |A|^{p}-t|B|^{p}=|B|^{p}\Big(\Big|\frac AB\Big|^{p}-t\Big),\quad
  |A-B|^{p}=|B|^{p}\Big|\frac AB-1\Big|^{p},
\]
applying \eqref{eq:FSpoint} with $A/B$ in place of $A$ and multiplying by $c^{p}|B|^{p}$ gives
\[
  \mathfrak{B}(a,b,c,d)\ \ge\ c_{p}\,c^{p}t^{\frac{p}{2}}|A-B|^{p}
  =c_{p}\,(cd)^{\frac{p}{2}}\Big|\frac ac-\frac bd\Big|^{p},
\]
because $c^{p}t^{\frac{p}{2}}=c^{\frac{p}{2}}d^{\frac{p}{2}}$. If $B=A=0$ both sides of \eqref{eq:gspoint} vanish, and if
$B=0\neq A$, dividing \eqref{eq:Bnorm} by $c^{p}|A|^{p}$ turns the claim into
$1-(1-t)^{p-1}\ge c_{p}t^{\frac{p}{2}}$, which follows from \eqref{eq:FSpoint} by letting $A\to\infty$. This
is \eqref{eq:gspoint}.

If $a/c=b/d$, that is, if $u$ is proportional to $\omega_{\gamma}$ (here $a=u(x)$, $b=u(y)$, $c=\omega_{\gamma}(x)$, $d=\omega_{\gamma}(y)$), then both sides of \eqref{eq:gspoint} vanish. Indeed, with $a=\lambda c$ and $b=\lambda d$,
\[
\mathfrak{B}(a,b,c,d)=|\lambda|^{p}|c-d|^{p}-\Jp(c-d)|\lambda|^{p}(c-d)=0 .
\]
Conversely, if \eqref{eq:FSpoint} is strict for $A\neq1$ and $t>0$ when $p>2$ (see \cite[Lemma~2.6]{FS08}), then, since $t=d/c>0$, equality in \eqref{eq:gspoint} with $b\neq0$ forces $a/c=b/d$.

\end{proof}

\subsection{The representation} We can now state the representation for one characteristic exponent.

\begin{theorem}\label{thm:gsrep}
Let $p\ge2$ and $0<\gamma<\gamma_{c}$. Then for every $u\in\Dsp$,
\begin{equation}\label{eq:gsrepmain}
  [u]^{p}_{s,p}-2\Phi(\gamma)\int_{\R^{N}}\frac{|u|^{p}}{|x|^{sp}}\dx
   \ge c_{p}\mathcal{E}_{\gamma}\big(|x|^{\gamma}u\big),
\end{equation}
with $c_{p}$ as in Lemma~\ref{lem:gspoint}. For $p=2$ equality holds, with $c_{2}=1$, whenever
$u\in\cc(\R^{N}\setminus\{0\})$.
\end{theorem}

\begin{proof}
For $u\in\cc(\R^{N}\setminus\{0\})$, \eqref{eq:gsrepmain} follows after combining Lemma~\ref{lem:gsreduce} with Lemma~\ref{lem:gspoint},
noting that $\omega_{\gamma}(x)^{\frac{p}{2}}\omega_{\gamma}(y)^{\frac{p}{2}}=|x|^{- \frac{p\gamma}{2}}|y|^{- \frac{p\gamma}{2}}$ and
$u(x)/\omega_{\gamma}(x)=|x|^{\gamma}u(x)$. In particular, for $p=2$ part (i) of Lemma \ref{lem:gspoint} gives the equality.

For general $u\in\Dsp$ we use that $\cc(\R^{N}\setminus\{0\})$ is dense in $\Dsp$, which holds because
$sp<N$. By the definition of $\Dsp$, we can approximate $u$ in $[\cdot]_{s,p}$ by functions of $\cc(\R^{N})$. To each approximation we apply the cut-off construction at the origin from the
proof of Lemma~\ref{lem:cap}. This removes a neighbourhood of the origin and changes
$[\cdot]_{s,p}$ by an arbitrarily small amount. We choose $u_{n}\in\cc(\R^{N}\setminus\{0\})$ with
$[u_{n}-u]_{s,p}\to0$. Then $[u_{n}]^{p}_{s,p}\to[u]^{p}_{s,p}$, and using \eqref{eq:FS} we get
$\int_{\R^N}|x|^{-sp}|u_{n}|^{p}\to\int_{\R^N}|x|^{-sp}|u|^{p}$. Passing to a subsequence with $u_{n}\to u$ a.e., we
have $|x|^{\gamma}u_{n}\to|x|^{\gamma}u$ a.e., and Fatou's lemma applied to \eqref{eq:intro-Egamma} gives
$\mathcal{E}_{\gamma}(|x|^{\gamma}u)\le\liminf_{n}\mathcal{E}_{\gamma}(|x|^{\gamma}u_{n})$. Now passing to the limit in \eqref{eq:gsrepmain} written for $u_{n}$, we get \eqref{eq:gsrepmain} for every $u \in \Dsp$.
\end{proof}

A single $\mu$ gives two such representations, one for each root of $\Phi=\mu$. This
leads to Theorem~\ref{thm:C}, which is proved in Section~\ref{sec:proofs}.


\begin{corollary}\label{cor:gsHS}
Let $p\ge2$, $0<\mu<\bar\mu$, and let $u\in\Dsp$ be a weak solution of \eqref{eq:P}. Then
\begin{equation}\label{eq:gsHS}
\mathcal{E}_{\gamma_{-}}\big(|x|^{\gamma_{-}}u\big)+\mathcal{E}_{\gamma_{+}}\big(|x|^{\gamma_{+}}u\big)
\le \frac{4}{c_{p}}\int_{\R^{N}}|u|^{p^{*}_{s}}\dx.
\end{equation}
\end{corollary}

\begin{proof}
Since $u\in\Dsp$, we have 
$[u]_{s,p}<\infty$, and $\int_{\R^N}|x|^{-sp}|u|^{p}\dx<\infty$ by \eqref{eq:FS}, 
$\int_{\R^N}|u|^{p^{*}_{s}}\dx<\infty$ by the Sobolev inequality. The choice $\varphi=u$ is admissible in \eqref{eq:weakdef} by density, since every term is finite.
Testing with $u$ gives
$\tfrac12[u]^{p}_{s,p}-\mu\int_{\R^N}|x|^{-sp}|u|^{p}=\int_{\R^N}|u|^{p^{*}_{s}}$, that is
$\mathcal{D}_{\mu}(u)=2\int_{\R^N}|u|^{p^{*}_{s}}$. By Corollary~\ref{cor:roots},
$\Phi(\gamma_{\pm})=\mu$. So Theorem~\ref{thm:gsrep} with $\gamma=\gamma_{-}$ and with
$\gamma=\gamma_{+}$ gives $\mathcal{D}_{\mu}(u)\ge c_{p}\mathcal{E}_{\gamma_{\pm}}(|x|^{\gamma_{\pm}}u)$.
Adding the two inequalities gives \eqref{eq:gsHS}.
\end{proof}

Estimate \eqref{eq:gsHS} requires neither \eqref{E0}, \eqref{E1} nor \eqref{M}, and no symmetry or
monotonicity of $u$.

\begin{remark}\label{rem:gsp2}
For $p=2$, Lemma~\ref{lem:gspoint}-(i) is an identity, so $\mathcal{D}_{\mu}(u)=\mathcal{E}_{\gamma_{-}}(|x|^{\gamma_{-}}u)=\mathcal{E}_{\gamma_{+}}(|x|^{\gamma_{+}}u)$ for $u\in\cc(\R^{N}\setminus\{0\})$. For $p\neq2$ the two weights are different (Proposition~\ref{cor:defect}). We assume $p\ge2$ because we do not know whether \eqref{eq:gspoint} holds for $1<p<2$.
\end{remark}

\section{Verification of the characteristic bounds}\label{sec:verify}
In this section and the next we prove \eqref{E0} and \eqref{E1} for positive radially symmetric nonincreasing solutions when $p\ge2$. We use the substitution $u=\omega_\gamma v$ of Section~\ref{sec:gsrep}, which removes the Hardy term, and then a weighted Moser iteration with finitely many steps. In this section we prove the coercivity estimates needed for the iteration. In Sections~\ref{sec:verify} and \ref{sec:scheme}, $c$ and $C$ denote positive constants which may depend on $N,p,s$ and on the exponent $r$, but never on the truncation level $M$.

\subsection{Ground-state substitution and coercivity}\label{sec:route}
We fix the lower root and write
\begin{equation}\label{eq:omminus}
  \omega:=\omega_{\gamma_{-}}=|x|^{-\gamma_{-}},\quad u=\omega v,\quad v=|x|^{\gamma_{-}}u ,
\end{equation}
so that $\Phi(\gamma_{-})=\mu$ and $\pLap\omega=\mu|x|^{-sp}\omega^{p-1}$ in
$\R^{N}\setminus\{0\}$ by \eqref{eq:homog}. We write
\[
  \mathcal{A}(w,\varphi):=\frac12\iint_{\R^{2N}}
  \frac{\Jp\big(w(x)-w(y)\big)\big(\varphi(x)-\varphi(y)\big)}{|x-y|^{N+sp}}\dxy .
\]

The next identity removes the Hardy term.
\begin{lemma}\label{lem:route}
Let $u$ be a positive weak solution of \eqref{eq:P} and let $v$ be as in \eqref{eq:omminus}. Then for
every $\varphi$ for which all three terms below are defined,
\begin{equation}\label{eq:routeid}
  \mathcal{A}(\omega v,\varphi)-\mathcal{A}\big(\omega,v^{p-1}\varphi\big)
  =\int_{\R^{N}}(\omega v)^{p^{*}_{s}-1}\varphi\dx .
\end{equation}
\end{lemma}

\begin{proof}
Testing \eqref{eq:P} with $\varphi$ and using $u^{p-1}=\omega^{p-1}v^{p-1}$ gives
$$\mathcal{A}(\omega v,\varphi)-\mu\int_{\R^N}|x|^{-sp}\omega^{p-1}v^{p-1}\varphi\dx
=\int_{\R^N}(\omega v)^{p^{*}_{s}-1}\varphi\dx.$$ Testing the equation \eqref{eq:intro-hardy} for $\omega$ with $v^{p-1}\varphi$ gives
$$\mathcal{A}(\omega,v^{p-1}\varphi)-\mu\int_{\R^N}|x|^{-sp}\omega^{p-1}v^{p-1}\varphi\dx=0.$$ Subtracting removes the Hardy term, and \eqref{eq:routeid} holds. 
\end{proof}

Identity \eqref{eq:routeid} differs from Lemma~\ref{lem:gsreduce} in that $\varphi$ is still free, so it
retains the equation and not only the energy. Taking
\begin{equation}\label{eq:routetest}
  \varphi=\omega\,\eta^{p}G(v),
\end{equation}
with $\eta$ a cutoff supported in an annulus and $G$ nondecreasing, the left-hand side of
\eqref{eq:routeid} is the double integral against $|x-y|^{-N-sp}$ of an expression which, for
$\eta\equiv1$, equals
\begin{equation}\label{eq:routebracket}
  \mathfrak{K}(a,b,X,Y):=\Jp(aX-bY)\big[aG(X)-bG(Y)\big]
  -\Jp(a-b)\big[aX^{p-1}G(X)-bY^{p-1}G(Y)\big],
\end{equation}
where $a=\omega(x)$, $b=\omega(y)$, $X=v(x)$, $Y=v(y)$, together with terms carrying derivatives of
$\eta$. Since $\mathfrak{K}(a,b,X,X)=0$, this term measures the variation of $v$.

\subsubsection{\textbf{A nonlinear coercivity inequality}} What is needed is a bound of the form
\begin{equation}\label{eq:routeC}
  \mathfrak{K}(a,b,X,Y)\ge c(ab)^{\frac{p}{2}}\big|H(X)-H(Y)\big|^{p},
\end{equation}
for $a,b>0$ and $X,Y\ge0$, with $c>0$ and $H$ determined by $G$. Three cases are proved.

\begin{enumerate}
\item[{\rm(a)}] If $G(t)=t$ then $\mathfrak{K}(a,b,X,Y)=\mathfrak{B}(aX,bY,a,b)$ with $\mathfrak{B}$ as
in \eqref{eq:bracket}, so Lemma~\ref{lem:gspoint} gives \eqref{eq:routeC} for every $p\ge2$, with
$H(t)=t$ and $c=c_{p}$.
\item[{\rm(b)}] If $p=2$ and $G$ is nondecreasing and locally Lipschitz, then
$\mathfrak{K}(a,b,X,Y)=ab\,(X-Y)\big(G(X)-G(Y)\big)$, and the argument of Claim~1 in the
proof of Theorem~\ref{thm:integ} gives \eqref{eq:routeC} with $H=g$, $g:=\int_{0}^{\cdot}(G')^{\frac12}$ and
$c=1$.
\item[{\rm(c)}] If $G(t)=t^{m}$ with $m=r-p+1$ and $r\ge p>2$, then \eqref{eq:routeC} holds with
$H(t)=t^{\frac{r}{p}}$. This is Lemma~\ref{lem:purepower} below.
\end{enumerate}

\noi Case (c) is the Moser version of Lemma~\ref{lem:gspoint}, which is the case $r=p$. Written in the
variables of its proof it is the extension of \cite[Lemma~2.6, (2.16)]{FS08} from the exponent pair
$(1,p)$ to $(m,r)$, the case $r=p$ being (2.16) itself. The truncated powers $G_{M}$ needed for an iteration are covered by
Proposition~\ref{prop:trunc}.

We also need \eqref{eq:routeC} in case (c) and its truncated versions. These are elementary but long pointwise inequalities in four variables. We prove them in Appendix~\ref{app:coerc}. Lemma~\ref{lem:purepower} treats pure powers. Proposition~\ref{prop:trunc} treats the truncated powers $G_M$ of \eqref{eq:GH}. Proposition~\ref{prop:truncquad} and Corollary~\ref{cor:truncboth} give an additional quadratic lower bound. We need it to control the errors produced by cutoff functions. All constants there are independent of the truncation level $M$.

\subsubsection{\textbf{Consequence of the coercivity estimate}} {By Corollary~\ref{cor:gsHS}, $\mathcal E_{\gamma_-}(v)\le\frac{4}{c_p}\int_{\R^N}|u|^{p^*_s}\dx<\infty$ for every solution, so the energy needed to start an iteration is available without any hypothesis on $u$. With \eqref{eq:routeC}, identity \eqref{eq:routeid} gives a weighted Caccioppoli inequality (Proposition~\ref{prop:cacc}), and the weighted Sobolev inequality of Proposition~\ref{prop:WS} then gives the iteration of Section~\ref{sec:scheme}. The test function $v^{p-1}\varphi$ is admissible by Lemma~\ref{lem:admiss}. At infinity we use $\omega_{\gamma_+}$ instead, since $\Phi(\gamma_+)=\mu$ as well. The energy is essential here. Indeed, $|x|^{-\gamma_+}$ also solves \eqref{eq:pureHardy}, and $v=|x|^{\gamma_--\gamma_+}$ is unbounded at the origin. It is excluded because $|x|^{-\gamma}\in W^{s,p}$ near the origin only when $\gamma<\gamma^*$.}

\section{Weighted Caccioppoli estimates and the finite iteration}\label{sec:scheme}
This section turns the coercivity estimate of Section~\ref{sec:verify} into a Caccioppoli inequality. We then run a Moser iteration for finitely many steps.

\subsection{Cutoff estimates}\label{sec:cutoffs}
We write $\chi:=\frac{p^{*}_{s}}{p}>1$ and note that $N=\gamma^{*}p^{*}_{s}$ by \eqref{eq:Xi}.
For $r\ge p^{*}_{s}$ and $M>0$ put $T_{M}(t):=\min\{t,M\}$, $m:=r-p+1$ and
\begin{equation}\label{eq:GH}
  G_{M}(t):=t\,T_{M}(t)^{m-1},\quad H_{M}(t):=t\,T_{M}(t)^{\frac{m-1}{p}} ,
\end{equation}
so that
\begin{equation}\label{eq:GHid}
  H_{M}(t)^{p}=t^{p-1}G_{M}(t),\quad H_{M}(t)^{p}\le t^{r},\quad H_{M}(t)\uparrow t^{\frac{r}{p}}\ \ (M\uparrow\infty).
\end{equation}
Since also $H_{M}(t)=t^{\frac{r}{p}}$ for $0\le t\le M$, we have $H_{M}(t)\approx t^{\frac{r}{p}}$ as in the notation of Section~\ref{sec:rep}. Moreover $H_{M}'\ge T_{M}^{\frac{r-p}{p}}$, and $T_{M}$ is nondecreasing, so for all $X,Y\ge0$,
\begin{equation}\label{eq:Hslope}
  T_{M}\big(\max\{X,Y\}\big)^{\frac{r-p}{p}}|X-Y| \le \big|H_{M}(X)-H_{M}(Y)\big|.
\end{equation}
For $X>Y$ and below the level $M$, the above inequality is equivalent to $X^{1+\beta}-Y^{1+\beta}\ge X^{\beta}(X-Y)$ with
$\beta=\frac{r-p}{p}$, that is to $X^{\beta}Y\ge Y^{1+\beta}$, and it is an identity above $M$.

Let $\gamma\in\{\gamma_{-},\gamma_{+}\}$, $\omega=\omega_{\gamma}$ and $v=v_{\gamma}:=|x|^{\gamma}U$. The first identity in \eqref{eq:GHid} is what makes the choice $\varphi=\omega\eta^{p}G_{M}(v)$ in
\eqref{eq:routeid} useful. Since $U^{p^{*}_{s}-1}\omega G_{M}(v)=U^{P_s}\omega^{p}v^{p-1}G_{M}(v)$, the
right-hand side of \eqref{eq:routeid} becomes
\begin{equation}\label{eq:rhsform}
  \int_{\R^{N}}(\omega v)^{p^{*}_{s}-1}\varphi\dx=\int_{\R^{N}}U^{P_s}\big(\omega\,\eta\,H_{M}(v)\big)^{p}\dx .
\end{equation}

The next identity separates the effect of the cutoff $\eta$.
\begin{lemma}\label{lem:cutoff}
Let $u=\omega_{\gamma}v$ and $\varphi=\omega_{\gamma}\eta^{p}G_{M}(v)$, and write
$a=\omega_{\gamma}(x)$, $b=\omega_{\gamma}(y)$, $X=v(x)$, $Y=v(y)$. Then
\begin{equation}\label{eq:cutoffid}
\begin{aligned}
  &\Jp\big(u(x)-u(y)\big)\big(\varphi(x)-\varphi(y)\big)
   -\Jp\big(\omega_{\gamma}(x)-\omega_{\gamma}(y)\big)
    \big(v^{p-1}\varphi(x)-v^{p-1}\varphi(y)\big)\\
  &\quad=\eta(x)^{p}\,\mathfrak{K}_{M}(a,b,X,Y)
   +b\,G_{M}(Y)\big[\eta(x)^{p}-\eta(y)^{p}\big]
    \Big[\Jp(aX-bY)-Y^{p-1}\Jp(a-b)\Big],
\end{aligned}
\end{equation}
where $\mathfrak{K}_{M}$ is \eqref{eq:routebracket} with $G=G_{M}$. The bracket in the second term
vanishes whenever $X=Y$.
\end{lemma}

\begin{proof}
We expand the left-hand side of \eqref{eq:cutoffid} using $\varphi(x)=a\eta(x)^{p}G_{M}(X)$ and
$v(x)^{p-1}\varphi(x)=aX^{p-1}\eta(x)^{p}G_{M}(X)$, and likewise at $y$. Subtracting
$\eta(x)^{p}\mathfrak{K}_{M}(a,b,X,Y)$ leaves the quantity
$\Jp(aX-bY)\,bG_{M}(Y)\big[\eta(x)^{p}-\eta(y)^{p}\big]$ from the first group and
$-\Jp(a-b)\,bY^{p-1}G_{M}(Y)\big[\eta(x)^{p}-\eta(y)^{p}\big]$ from the second. If $X=Y$ then $\Jp(aX-bX)=X^{p-1}\Jp(a-b)$, so the bracket vanishes.
\end{proof}

We now combine the coercivity estimates with the cutoff. In the following proposition we use the elementary inequality:
\begin{equation}\label{eq:weightdom}
  a^{p-1}b+ab^{p-1} \le C_{p}\Big[(ab)^{\frac{p}{2}}+ab\,|a-b|^{p-2}\Big],\quad a,b>0,
\end{equation}
which, for $b\le a$ and $t=\frac{b}{a}$, follows from $$\inf_{0<t\le1}\big[t^{\frac{p}{2}-1}+(1-t)^{p-2}\big]>0.$$
\begin{proposition}\label{prop:LC}
Let $p>2$, $r>p$, $M>0$, and let $\mathcal{L}_{e,f}$ denote the left-hand side of \eqref{eq:cutoffid},
that is
\begin{align*}
 \mathcal{L}_{e,f} & := \Jp(aX-bY)\big(ae^{p}G_{M}(X)-bf^{p}G_{M}(Y)\big) \\
  & \quad -\Jp(a-b)\big(aX^{p-1}e^{p}G_{M}(X)-bY^{p-1}f^{p}G_{M}(Y)\big).
\end{align*}
There are $c_r>0$ and $C_r<\infty$, both independent of $M$, such that for all $a,b>0$,
$X,Y\ge0$ and $e,f\in[0,1]$, writing $H_{X}=H_{M}(X)$, $H_{Y}=H_{M}(Y)$,
\begin{equation}\label{eq:LC}
\begin{aligned}
  \mathcal{L}_{e,f} & \ge c_r(ab)^{\frac{p}{2}}\big|eH_{X}-fH_{Y}\big|^{p}\\
  &\quad -C_r(ab)^{\frac{p}{2}}\big(H_{X}+H_{Y}\big)^{p}|e-f|^{p}
  -C_rab|a-b|^{p-2}\big(H_{X}+H_{Y}\big)^{p}|e-f|^{2}.
\end{aligned}
\end{equation}
\end{proposition}

\begin{proof}
Throughout write $S:=H_{X}+H_{Y}$, $\Delta:=|H_{X}-H_{Y}|$, $Q:=ab|a-b|^{p-2}$ and
$W:=a^{p-1}b+ab^{p-1}$, and note $\Delta\le S$.

Applying \eqref{eq:cutoffid} and the same identity with $x$ and $y$
interchanged, and averaging,
\begin{equation}\label{eq:LCsym}
  2\mathcal{L}_{e,f}=(e^{p}+f^{p})\,\mathfrak{K}_{M}(a,b,X,Y)
  +(e^{p}-f^{p})\Big[bG_{M}(Y)R_{Y}+aG_{M}(X)R_{X}\Big],
\end{equation}
where $R_{Y}:=\Jp(aX-bY)-Y^{p-1}\Jp(a-b)$ and $R_{X}:=\Jp(aX-bY)-X^{p-1}\Jp(a-b)$.

We now bound the second term. Since $aX-bY=(a-b)Y+a(X-Y)=(a-b)X+b(X-Y)$, both $R_{X}$ and $R_{Y}$ are
increments of $\Jp$, and
\begin{equation}\label{eq:incr}
  |\Jp(q+z)-\Jp(q)|\le C_{p}\big(|q|^{p-2}|z|+|z|^{p-1}\big)\quad(p\ge2)
\end{equation}
gives
\[
  |R_{Y}|\le C_{p}\Big[a|a-b|^{p-2}Y^{p-2}|X-Y|+a^{p-1}|X-Y|^{p-1}\Big],
\]
and the same with $a,X$ and $b,Y$ interchanged for $R_{X}$. By \eqref{eq:GHid},
$G_{M}(Y)=H_{Y}^{p}Y^{1-p}$ and $H_{Y}/Y=T_{M}(Y)^{\frac{r-p}{p}}$, so
\begin{align*}
    & G_{M}(Y)Y^{p-2}|X-Y|=H_{Y}^{p-1}T_{M}(Y)^{\frac{r-p}{p}}|X-Y|\le H_{Y}^{p-1}\Delta, \text{ and } \\
    & G_{M}(Y)|X-Y|^{p-1}=H_{Y}\Big(T_{M}(Y)^{\frac{r-p}{p}}|X-Y|\Big)^{p-1}\le H_{Y}\Delta^{p-1},
\end{align*}
both by $T_{M}(Y)\le T_{M}(\max\{X,Y\})$ and \eqref{eq:Hslope}. Using the same bounds at $X$,
\begin{equation}\label{eq:LCerr}
  \Big|bG_{M}(Y)R_{Y}+aG_{M}(X)R_{X}\Big| \le C_{p}\Big[Q\,S^{p-1}\Delta+W\,S\,\Delta^{p-1}\Big].
\end{equation}
Moreover $|e^{p}-f^{p}|\le p\max\{e,f\}^{p-1}|e-f|\le p(e^{p}+f^{p})^{\frac{p-1}{p}}|e-f|$, and by
\eqref{eq:weightdom} $W\le C_{p}\big[(ab)^{\frac{p}{2}}+Q\big]$. Hence the second term of \eqref{eq:LCsym} is
bounded in absolute value by $C_{p}\big[(\mathrm{I})+(\mathrm{II})+(\mathrm{III})\big]$, where
\[
  (\mathrm{I})=\Sigma\,Q\,S^{p-1}\Delta,\quad
  (\mathrm{II})=\Sigma\,(ab)^{\frac{p}{2}}S\,\Delta^{p-1},\quad
  (\mathrm{III})=\Sigma\,Q\,S\,\Delta^{p-1},
\]
with $\Sigma:=(e^{p}+f^{p})^{\frac{p-1}{p}}|e-f|$. Since $\Delta\le S$ we have
$\Delta^{p-1}\le S^{p-2}\Delta$, so $(\mathrm{III})\le(\mathrm{I})$ and only two terms remain.

We estimate these two terms by Young's inequality. Let $c_*$ be the constant of
Corollary~\ref{cor:truncboth}, so that
$\mathfrak{K}_{M}\ge c_*\big[(ab)^{\frac{p}{2}}\Delta^{p}+QS^{p-2}\Delta^{2}\big]$. For $(\mathrm{I})$ write
\[
  (\mathrm{I})=\Big[(e^{p}+f^{p})^{\frac12}Q^{\frac12}S^{\frac{p-2}{2}}\Delta\Big]
  \cdot\Big[(e^{p}+f^{p})^{\frac{p-1}{p}-\frac12}|e-f|\,Q^{\frac12}S^{\frac p2}\Big],
\]
so that, by Young's inequality with exponents $2,2$ and $e^{p}+f^{p}\le2$,
\[
  (\mathrm{I})\le\eps\,(e^{p}+f^{p})QS^{p-2}\Delta^{2}+C_{\eps}\,|e-f|^{2}QS^{p} ,
\]
the remaining factor $(e^{p}+f^{p})^{1-\frac2p}$ being bounded because $p>2$. For $(\mathrm{II})$,
Young's inequality with exponents $\tfrac{p}{p-1},p$ applied to
$\big[(e^{p}+f^{p})^{\frac{p-1}{p}}(ab)^{\frac{p-1}{2}}\Delta^{p-1}\big]\cdot
\big[|e-f|(ab)^{\frac12}S\big]$ gives
\[
  (\mathrm{II})\le\eps\,(e^{p}+f^{p})(ab)^{\frac{p}{2}}\Delta^{p}+C_{\eps}\,|e-f|^{p}(ab)^{\frac{p}{2}}S^{p} .
\]
Choosing $\eps$ so that $C_{p}(2\eps)\le c_*/2$ and inserting these into \eqref{eq:LCsym},
\[
  2\mathcal{L}_{e,f}\ \ge\ \tfrac12c_*(e^{p}+f^{p})(ab)^{\frac{p}{2}}\Delta^{p}
  -C\big[(ab)^{\frac{p}{2}}S^{p}|e-f|^{p}+QS^{p}|e-f|^{2}\big].
\]

Finally, from $|eH_{X}-fH_{Y}|\le e\Delta+H_{Y}|e-f|$ and its mirror
image, $|eH_{X}-fH_{Y}|^{p}\le2^{p-1}\big[(e^{p}+f^{p})\Delta^{p}+S^{p}|e-f|^{p}\big]$, so
$(e^{p}+f^{p})\Delta^{p}\ge2^{1-p}|eH_{X}-fH_{Y}|^{p}-S^{p}|e-f|^{p}$. Substituting this and dividing by
$2$ gives \eqref{eq:LC} with $c_r=2^{-p-1}c_*$ and a suitable $C_r$, none of which depends
on $M$.
\end{proof}

\subsubsection{\textbf{Removing the inner cutoff}} The test function \eqref{eq:routetest} may not be admissible near the origin, so we multiply it by a cutoff $\zeta_{\eps}$ vanishing near $0$ and let $\eps\downarrow0$. The scaling is critical, and the smallness comes from Lemma~\ref{lem:radial}-(ii).

\begin{lemma}\label{lem:innercut}
Let $p \ge 2$, $\gamma \in (0,\gamma^{*})$, $\omega=\omega_{\gamma}$, and let $U\in\Dsp$ be a positive radially symmetric
nonincreasing solution of \eqref{eq:P}, $v=|x|^{\gamma}U$ and $H=H_{M}$ as in \eqref{eq:GH}. We put
\[
  \kappa(\rho):=\sup_{0<\tau\le\rho}\tau^{\gamma^{*}}U(\tau),\quad
  A:=N-\tfrac{p\gamma}{2},\quad B:=N-(p-1)\gamma ,
\]
and for $\eps>0$ let $\zeta_{\eps}$ satisfy $0\le\zeta_{\eps}\le1$, $\zeta_{\eps}=0$ on $B_{\eps}$,
$\zeta_{\eps}=1$ outside $B_{2\eps}$ and $|\nabla\zeta_{\eps}|\le \frac{C}{\eps}$. Then $A,B>0$ and
\begin{equation}\label{eq:innercut}
  I_{p,\eps}+I_{2,\eps}\longrightarrow0\quad(\eps\downarrow0),
\end{equation}
where, with $H_{x}:=H(v(x))$ and $\omega_{x}:=\omega(x)$,
\begin{align*}
  &I_{p,\eps}:=\iint_{\R^{2N}}\frac{(\omega_{x}\omega_{y})^{\frac{p}{2}}\big(H_{x}+H_{y}\big)^{p}
  |\zeta_{\eps}(x)-\zeta_{\eps}(y)|^{p}}{|x-y|^{N+sp}}\dxy, \text{ and } \\
  &I_{2,\eps}:=\iint_{\R^{2N}}\frac{\omega_{x}\omega_{y}|\omega_{x}-\omega_{y}|^{p-2}
  \big(H_{x}+H_{y}\big)^{p}|\zeta_{\eps}(x)-\zeta_{\eps}(y)|^{2}}{|x-y|^{N+sp}}\dxy 
\end{align*}
\end{lemma}

\begin{proof}
Note that, $\gamma<\gamma^{*}$ implies $p\gamma<N-sp<N$, which implies $A\ge \tfrac N2$ and $B>N-p\gamma>0$. By Lemma~\ref{lem:radial}-(ii), $\kappa$ is bounded and $\kappa(\rho)\to0$ as $\rho\downarrow0$. Since $H_{M}(t)\le M^{\frac{r-p}{p}}t$ and $v(x)=|x|^{\gamma-\gamma^{*}}|x|^{\gamma^{*}}U(x)$, we get 
\begin{equation}\label{eq:Hkappa}
  H_{x}^{p} \le M^{r-p}\kappa(|x|)^{p}|x|^{p(\gamma-\gamma^{*})}.
\end{equation}
The integrands vanish unless one of $|x|,|y|$ is at most $2\eps$, so we only need to treat the region where $|x|,|y|\le4\eps$ and the region where $|x|\le2\eps$ and $|y|\ge4\eps$, together with its mirror image.

We first treat the region where $|x|,|y|\le4\eps$. The substitution $x=\eps\xi$, $y=\eps\eta$ multiplies $\dxy$ by $\eps^{2N}$, the weight $(\omega_{x}\omega_{y})^{\frac{p}{2}}$ by $\eps^{-p\gamma}$, the kernel by $\eps^{-N-sp}$ and, by \eqref{eq:Hkappa}, the factor $(H_{x}+H_{y})^{p}$ by $M^{r-p}\kappa(4\eps)^{p} \eps^{p(\gamma-\gamma^{*})}$. The total power of $\eps$ is $$ 2N-p\gamma-N-sp+p(\gamma-\gamma^{*})=N-sp-p\gamma^{*}=0,$$ which indicates that no smallness comes from the scaling. What remains is the prefactor $M^{r-p}\kappa(4\eps)^{p}$ times the corresponding integral $\mathcal{I}_p$ for the fixed profile $\zeta_{1}$. We now calculate 
\begin{align*}
    \mathcal{I}_p\Big|_{\{|\xi|\le 4, |\eta|\le 4\}} :=  \iint_{\{|\xi|\le 4, |\eta|\le 4\}} \frac{(\omega_{\xi}\omega_{\eta})^{\frac{p}{2}}\big(H_{\xi}+H_{\eta}\big)^{p} |\zeta_{1}(\xi)-\zeta_{1}(\eta)|^{p}}{|\xi-\eta|^{N+sp}}\d\xi \d \eta
\end{align*}
Using $|\zeta_{1}(\xi)-\zeta_{1}(\eta)|\le C|\xi-\eta|$ and Lemma~\ref{lem:radial}-(ii), we estimate 
\begin{align*}
    \mathcal{I}_p\Big|_{\{|\xi|\le 4, |\eta|\le 4\}} & \le C \iint_{\{|\xi|\le 4, |\eta|\le 4\}} \frac{(\omega_{\xi}\omega_{\eta})^{\frac{p}{2}}\big(H_{\xi}+H_{\eta}\big)^{p}}{|\xi-\eta|^{N+sp-p}}\d\xi \d \eta \\
    & \le C(p)\iint_{\{|\xi|\le 4, |\eta|\le 4\}} \frac{(|\xi|^{-\gamma}|\eta|^{-\gamma})^{\frac{p}{2}}\left(|\xi|^{(\gamma-\gamma^{*})p}+|\eta|^{(\gamma-\gamma^{*})p}\right)}{|\xi-\eta|^{N+sp-p}}\d\xi \d \eta.
\end{align*}
We divide the region into three parts $D_1:=\{ |\xi| \le \frac{|\eta|}{2} \}, D_2:= \{ |\xi| \ge 2 |\eta|\}$, and $D_3:= \{ \frac{|\xi|}{2} \le |\eta| \le 2|\xi|\}$. 

On $D_1$, we have $\frac{1}{2} |\eta| \le |\eta| - |\xi| \le |\xi-\eta| \le \frac{3}{2}|\eta|$ and $|\xi|^{p(\gamma -\gamma^*)} +|\eta|^{p(\gamma-\gamma^{*})} \le 2 |\xi|^{p(\gamma-\gamma^{*})}$. Hence using the fact that $p(\gamma-\gamma^*)-\frac{p\gamma}{2} = sp-N+\frac{p\gamma}{2}$, 
\begin{align*}
    & \iint_{D_1} \frac{(|\xi|^{-\gamma}|\eta|^{-\gamma})^{\frac{p}{2}}\left(|\xi|^{(\gamma-\gamma^{*})p}+|\eta|^{(\gamma-\gamma^{*})p}\right)}{|\xi-\eta|^{N+sp-p}}\d\xi \d\eta \\
    & \le C(N,p,s) \int_{|\eta|\le 4}  |\eta|^{p-sp-N - \frac{\gamma p}{2}} \left(\int_{|\xi| \le \frac{|\eta|}{2}} |\xi|^{sp-N+\frac{p\gamma}{2}} \d \xi \right) \d \eta \\
    & \le C(N,p,s) \int_{|\eta|\le 4} |\eta|^{p-sp-N - \frac{\gamma p}{2}} |\eta|^{sp+\frac{p\gamma}{2}}  \d \eta \le C(N,p,s). 
\end{align*}
On $D_2$, by exchanging the roles of $\xi$ and $\eta$, we similarly obtain 
\begin{align*}
    \iint_{D_2} \frac{(|\xi|^{-\gamma}|\eta|^{-\gamma})^{\frac{p}{2}}\left(|\xi|^{(\gamma-\gamma^{*})p}+|\eta|^{(\gamma-\gamma^{*})p}\right)}{|\xi-\eta|^{N+sp-p}}\d\xi \d\eta \le C(N,p,s).
\end{align*}
On $D_3$, $|\xi|\asymp |\eta|$, $|\xi|^{-\frac{p\gamma}{2}} \asymp |\eta|^{-\frac{p \gamma}{2}}$ and $|\xi-\eta| \le 3|\xi|$. Further, $p(\gamma-\gamma^*) - p \gamma = sp-N$, which yield
\begin{align*}
    &\iint_{D_3} \frac{(|\xi|^{-\gamma}|\eta|^{-\gamma})^{\frac{p}{2}}\left(|\xi|^{(\gamma-\gamma^{*})p}+|\eta|^{(\gamma-\gamma^{*})p}\right)}{|\xi-\eta|^{N+sp-p}}\d\xi \d\eta \\
    &\le C(N,p,s)\int_{|\xi| \le 4} |\xi|^{sp-N} \left( \int_{|z| \le 3|\xi|} |z|^{p-N-sp} \dz \right) \d \xi \\
    & \le C(N,p,s)\int_{|\xi| \le 4} |\xi|^{p-N} \d \xi \le C(N,p,s). 
\end{align*}
Therefore, on the region $\{|x| \le 4\eps,|y|\le4\eps\}$, the integral value of $I_{p,\eps}$ is at most $CM^{r-p}\kappa(4\eps)^{p}$.

For $I_{2,\eps}$ we use the same substitution. The weight
$\omega_{x}\omega_{y}|\omega_{x}-\omega_{y}|^{p-2}$ also scales like $\eps^{-p\gamma}$, and here total effective power of $\eps$ is also zero. So the remaining prefactor is $M^{r-p}\kappa(4\eps)^{p}$ which multiplies the corresponding integral $\mathcal{I}_2$ for the fixed profile $\zeta_{1}$. For $\mathcal{I}_2$, we similarly divide the region into $D_1$, $D_2$, and $D_3$. On $D_1$, using $|\xi|^{-\gamma}|\eta|^{-\gamma} \left| |\xi|^{-\gamma} - |\eta|^{-\gamma}\right|^{p-2} \le |\xi|^{-\gamma (p-1)} |\eta|^{- \gamma}$, and  $(\gamma-\gamma^{*})p - (p-1)\gamma = \gamma - N + sp$, 
\begin{align*}
    &\iint_{D_1} \frac{|\xi|^{-\gamma}|\eta|^{-\gamma} \left| |\xi|^{-\gamma} - |\eta|^{-\gamma}\right|^{p-2}\left(|\xi|^{(\gamma-\gamma^{*})p}+|\eta|^{(\gamma-\gamma^{*})p}\right)}{|\xi-\eta|^{N+sp-2}} \d\xi \d\eta \\
    & \le C(N,p,s) \int_{|\eta|\le 4}  |\eta|^{-\gamma - N - sp+2} \left(\int_{|\xi| \le \frac{|\eta|}{2}} |\xi|^{\gamma - N +sp} \d\xi \right) \d \eta \\
    & \le C(N,p,s)  \int_{|\eta|\le 4} |\eta|^{- N +2} \d \eta \le C(N,p,s). 
\end{align*}
On $D_2$, by exchanging the roles of $\xi$ and $\eta$, we similarly obtain the integral value is less than some $C(N,p,s)$. On $D_3$, using the mean value theorem, $||\xi|^{-\gamma} - |\eta|^{-\gamma}| \le C(\gamma) |\xi|^{-\gamma -1} |\xi - \eta|$. Further, $|\xi|^{-\gamma} \asymp |\eta|^{-\gamma}$ and  $|\xi|^{(\gamma-\gamma^*)p} \asymp |\eta|^{(\gamma-\gamma^*)p}$ gives 
\begin{align*}
    &\iint_{D_3} \frac{|\xi|^{-\gamma}|\eta|^{-\gamma} \left| |\xi|^{-\gamma} - |\eta|^{-\gamma}\right|^{p-2}\left(|\xi|^{(\gamma-\gamma^{*})p}+|\eta|^{(\gamma-\gamma^{*})p}\right)}{|\xi-\eta|^{N+sp-2}} \d\xi \d\eta \\
    & \le C(N,p,s) \iint_{D_3} \frac{|\xi|^{-2\gamma - (\gamma +1)(p-2) + p(\gamma - \gamma^*)}}{|\xi-\eta|^{N+sp-2-(p-2)}} \d\xi \d\eta \\
    & \le C(N,p,s) \int_{|\xi| \le 4} |\xi|^{2-N-p+sp} \left( \int_{|z| \le 3 |\xi|}  |z|^{p-N-sp} \dz \right) \d \xi \\ & \le C(N,p,s) \int_{|\xi| \le 4} |\xi|^{2-N} \d\xi \le C(N,p,s).
\end{align*}
Therefore, on the region $\{|x| \le 4\eps,|y|\le4\eps\}$, the integral value of $I_{p,\eps}$ is at most $O\big(\kappa(4\eps)^{p}\big)$.

Next,  we treat the region where $|x|\le2\eps$ and $|y|\ge4\eps$. Then $|x-y|\simeq|y|$ and $|\zeta_{\eps}(x)-\zeta_{\eps}(y)|\le1$. The contribution of $H_{x}^{p}$ is bounded, after integrating in $y$, by $CM^{r-p}\kappa(2\eps)^{p}$ as above. For the contribution of $H_{y}^{p}$ to $I_{p,\eps}$,
\eqref{eq:Hkappa} and $\int_{|x|\le2\eps}|x|^{-\frac{p\gamma}{2}}\dx=C\eps^{A}$ give
\begin{align}\label{bound-I}
    CM^{r-p}\,\eps^{A}\int_{4\eps}^{\infty}\kappa(\rho)^{p}\rho^{-A-1}\d\rho ,
\end{align}
because the radial exponent is
$-\tfrac{p\gamma}{2}+p(\gamma-\gamma^{*})-N-sp+N-1=-A-1$. For $I_{2,\eps}$ we
use $\omega_{x}\ge\omega_{y}$ and $|\omega_{x}-\omega_{y}|\le\omega_{x}$, so that
$\omega_{x}\omega_{y}|\omega_{x}-\omega_{y}|^{p-2}\le\omega_{x}^{p-1}\omega_{y}$. Since
$\int_{|x|\le2\eps}|x|^{-(p-1)\gamma}\dx=C\eps^{B}$ and the radial exponent is now
$-\gamma+p(\gamma-\gamma^{*})-N-sp+N-1=-B-1$, the same bound \eqref{bound-I} holds with $A$ replaced by $B$.

It remains to show that these two terms tend to $0$. Both have the form
$$J(\eps):=\eps^{a}\int_{4\eps}^{\infty}\kappa(\rho)^{p}\rho^{-a-1}\d\rho, \quad \text{ with } a>0.$$ For any
$\delta>0$,
\[
  J(\eps) \le  \sup_{\rho\le\delta}\kappa(\rho)^{p}\,\eps^{a}\int_{4\eps}^{\delta}\rho^{-a-1}\d\rho
  +\|\kappa\|_{\infty}^{p}\,\eps^{a}\int_{\delta}^{\infty}\rho^{-a-1}\d\rho
   \le  \frac{4^{-a}}{a}\sup_{\rho\le\delta}\kappa(\rho)^{p}+\frac{\|\kappa\|_{\infty}^{p}}{a}
  \Big(\frac{\eps}{\delta}\Big)^{a},
\]
so $\limsup_{\eps\downarrow0}J(\eps)\le C\sup_{\rho\le\delta}\kappa(\rho)^{p}$ for every $\delta>0$,
and letting $\delta\downarrow0$ gives $J(\eps)\to0$. Together with the near estimates, this proves \eqref{eq:innercut}.
\end{proof}

We use the following weighted Sobolev inequality of Abdellaoui and Bentifour \cite[Theorem~1.5]{AB17}.
\begin{proposition}\label{prop:WS}
Let $0<\gamma<\gamma^{*}$. Then there is $S_{\gamma}>0$ such that
\begin{equation}\label{eq:WS}
  S_{\gamma}\Big(\int_{\R^{N}}|x|^{-\gamma p^{*}_{s}}|w|^{p^{*}_{s}}\dx\Big)^{\frac{p}{p^{*}_{s}}}\le\mathcal{E}_{\gamma}(w)
  \quad\text{for all }w\in\cc(\R^{N}\setminus\{0\}).
\end{equation}
\end{proposition}

\noi Indeed, with $\beta=\frac{p\gamma}{2}$ the weight in \eqref{eq:intro-Egamma} is $|x|^{-\beta}|y|^{-\beta}$, the
weight $|x|^{-\frac{2\beta p^{*}_{s}}{p}}$ of \cite[Theorem~1.5]{AB17} is $|x|^{-\gamma p^{*}_{s}}$, and the condition
$0<\beta<\frac{N-sp}{2}$ imposed in \cite[Theorem~1.5]{AB17} is exactly $0<\gamma<\gamma^{*}$.

Proposition~\ref{prop:WS} covers only $\gamma<\gamma^{*}$.The upper root is obtained from it via inversion, as shown in the following lemma.

\begin{lemma}\label{lem:inv}
For $w\in\cc(\R^{N}\setminus\{0\})$ set $W(X):=w\big(\frac{X}{|X|^{2}}\big)$. Then for every $\gamma>0$,
\begin{equation}\label{eq:invid}
  \mathcal{E}_{\gamma}(w)=\mathcal{E}_{2\gamma^{*}-\gamma}(W), \; \text{ and } \;
  \int_{\R^{N}}|x|^{-\gamma p^{*}_{s}}|w|^{p^{*}_{s}} \dx=\int_{\R^{N}}|X|^{-(2\gamma^{*}-\gamma)p^{*}_{s}}|W|^{p^{*}_{s}}\d X .
\end{equation}
Consequently \eqref{eq:WS} for $\gamma$ is equivalent to \eqref{eq:WS} for $2\gamma^{*}-\gamma$.
Moreover, if $p\ge2$ then $\gamma_{c}\le2\gamma^{*}$, so $\gamma\in(\gamma^{*},\gamma_{c})$ implies
$2\gamma^{*}-\gamma\in(0,\gamma^{*})$, and \eqref{eq:WS} at the upper root follows from
Proposition~\ref{prop:WS}.
\end{lemma}

\begin{proof}
Under $x=\frac{X}{|X|^{2}}$, $y=\frac{Y}{|Y|^{2}}$ we have $|x-y|=\frac{|X-Y|}{|X||Y|}, \dx=|X|^{-2N} \d X$, and $\dy=|Y|^{-2N} \d Y$. Hence, with
$\beta:=\frac{p\gamma}{2}$,
\begin{align*}
  \iint_{\R^{2N}}\frac{|w(x)-w(y)|^{p}}{|x|^{\beta}|y|^{\beta}|x-y|^{N+sp}}\dxy
  =\iint_{\R^{2N}}\frac{|W(X)-W(Y)|^{p}}{|X|^{N-sp-\beta}|Y|^{N-sp-\beta}|X-Y|^{N+sp}}\d X \d Y,  
\end{align*}
and $N-sp-\beta=p\gamma^{*}-\frac{p\gamma}{2}=\frac p2(2\gamma^{*}-\gamma)$, which is the exponent $\beta$ associated with $2\gamma^{*}-\gamma$. For the second identity,
\begin{align*}
  \int_{\R^N} |x|^{-\gamma p^{*}_{s}}|w|^{p^{*}_{s}}\dx = \int_{\R^N}  |X|^{\gamma p^{*}_{s}-2N}|W|^{p^{*}_{s}} \d X,  
\end{align*}
and
$\gamma p^{*}_{s}-2N=-(2\gamma^{*}-\gamma)p^{*}_{s}$ because $N=\gamma^{*}p^{*}_{s}$. Finally, $\frac{\gamma_{c}}{\gamma^{*}}=\frac{p}{p-1}\le2$ exactly when $p\ge2$.
\end{proof}

The following lemma converts weighted integrability of $v_{\gamma}$ into a pointwise bound, and is where
the losses in \eqref{E0} and \eqref{E1} appear.

\begin{lemma}\label{lem:radexp}
Let $U\ge0$ be radially symmetric and nonincreasing, let $\gamma>0$ and $r>p^{*}_{s}$.
\begin{enumerate}
\item[{\rm(i)}] If $\int_{B_{\rho_{0}}}|x|^{-\gamma p^{*}_{s}}v_{\gamma}^{r}\dx<\infty$, then
$U(\rho)\le C\rho^{-\gamma-p^{*}_{s}\frac{\gamma^{*}-\gamma}{r}}$ for $0<\rho<\rho_{0}$.
\item[{\rm(ii)}] If $\int_{\R^{N}\setminus B_{R_{0}}}|x|^{-\gamma p^{*}_{s}}v_{\gamma}^{r}\dx<\infty$, then
$U(\rho)\le C\rho^{-\gamma+p^{*}_{s}\frac{\gamma-\gamma^{*}}{r}}$ for $\rho>2R_{0}$.
\end{enumerate}
\end{lemma}

\begin{proof}
Since $v_{\gamma}=|x|^{\gamma}U$, the integrand is $U^{r}|x|^{\gamma(r-p^{*}_{s})}$, and $\gamma(r-p^{*}_{s})>0$. For
(i), monotonicity gives
$\int_{B_{\rho}}U^{r}|x|^{\gamma(r-p^{*}_{s})}\dx\ge cU(\rho)^{r}\rho^{N+\gamma(r-p^{*}_{s})}$, so
$U(\rho)\le C\rho^{-\gamma-\frac{N-\gamma p^{*}_{s}}{r}}$, and $N-\gamma p^{*}_{s}=p^{*}_{s}(\gamma^{*}-\gamma)$ because
$N=\gamma^{*}p^{*}_{s}$. For (ii), the same bound on the annulus $\{\frac{\rho}{2}<|x|<\rho\}$, where
$|x|^{\gamma(r-p^{*}_{s})}\ge(\frac{\rho}{2})^{\gamma(r-p^{*}_{s})}$, gives
$U(\rho)\le C\rho^{-\gamma-\frac{N-\gamma p^{*}_{s}}{r}}=C\rho^{-\gamma+p^{*}_{s}\frac{\gamma-\gamma^{*}}{r}}$.
\end{proof}

\subsubsection{\textbf{Admissibility of the test pair}} Before estimating the cutoff terms we record that, for
each fixed $M$, both test functions occurring in \eqref{eq:routeid} are admissible. 

\begin{lemma}\label{lem:admiss}
Let $\gamma\in(0,\gamma_{c})$, $r\ge p^{*}_{s}$, $M>0$, and let $U\in\Dsp$ be a positive weak solution of
\eqref{eq:P} with $U\in L^{\infty}_{\mathrm{loc}}(\R^{N}\setminus\{0\})$. Let $v=v_{\gamma}=|x|^{\gamma}U$
and let $\eta$ be Lipschitz with $0\le\eta\le1$ and $\operatorname{supp}\eta\Subset\R^{N}\setminus\{0\}$.
We put $\varphi:=\omega_{\gamma}\,\eta^{p}G_{M}(v)$. Then
\[
  \varphi\ \text{ and }\ v^{p-1}\varphi=\omega_{\gamma}\,\eta^{p}H_{M}(v)^{p}
\]
both belong to $W^{s,p}(\R^{N})\cap L^{\infty}(\R^{N})$ and have compact support in
$\R^{N}\setminus\{0\}$. Consequently, $\varphi$ is admissible in \eqref{eq:weakdef} for $U$,
$v^{p-1}\varphi$ is admissible in \eqref{eq:weakdef} for $\omega_{\gamma}$, and \eqref{eq:routeid} holds
for this $\varphi$ with all three terms finite.
\end{lemma}

\begin{proof}
We write $\Sigma:=\operatorname{supp}\eta$, a compact subset of $\R^{N}\setminus\{0\}$. On $\Sigma$ the
weight $\omega_{\gamma}$ is smooth and bounded above and below, and $U$ is bounded. Hence $v$ is bounded
on $\Sigma$, and $v\in W^{s,p}(\Sigma)$ because $U\in\Dsp$ and $|x|^{\gamma}$ is smooth and bounded on a
neighbourhood of $\Sigma$.

By \eqref{eq:GH}, $G_{M}(t)=t\,T_{M}(t)^{r-p}$ and $H_{M}(t)^{p}=t^{p-1}G_{M}(t)=t^{p}T_{M}(t)^{r-p}$.
Both are nondecreasing, vanish at $0$, and satisfy
$G_{M}(t)\le M^{r-p}t$ and $H_{M}(t)^{p}\le M^{r-p}t^{p}$. Moreover $G_{M}$ is Lipschitz on
$[0,\infty)$ with constant $(r-p+1)M^{r-p}$, and $t\mapsto H_{M}(t)^{p}$ is Lipschitz on every bounded
interval. If we compose a Lipschitz function vanishing at $0$ with a bounded $W^{s,p}$ function, we again get a
bounded $W^{s,p}$ function. Multiplication by the bounded Lipschitz function $\omega_{\gamma}\eta^{p}$,
which is supported in $\Sigma$, preserves this property. Hence $\varphi,v^{p-1}\varphi\in
W^{s,p}(\R^{N})\cap L^{\infty}(\R^{N})$ with support in $\Sigma$.

For the second equation we need $\omega_{\gamma}$ to be a legitimate weak solution against such test
functions. Since $\gamma<\gamma_{c}$ we have $0<\gamma(p-1)<N-sp<N$, so
$$\int_{\R^{N}}(1+|y|)^{-N-sp}|y|^{-\gamma(p-1)}\dy<\infty,$$ that is
$\omega_{\gamma}\in L^{p-1}_{sp}(\R^{N})\cap W^{s,p}_{\mathrm{loc}}(\R^{N}\setminus\{0\})$. This is the
integrability required in \eqref{eq:op}--\eqref{eq:weakdef} and already used in \eqref{eq:homog}. Both
sides of \eqref{eq:weakdef} for $\omega_{\gamma}$ are then finite for a bounded $W^{s,p}$ test function
supported in $\Sigma$. For $U$ the three terms of \eqref{eq:weakdef} are finite because
$[U]_{s,p}<\infty$, $\int_{\R^N} |x|^{-sp}U^{p}\dx<\infty$ by \eqref{eq:FS} and $\varphi$ is bounded with
support away from the origin. Subtracting as in the proof of Lemma~\ref{lem:route} gives
\eqref{eq:routeid}.
\end{proof}

\subsubsection{\textbf{Removing the cutoff at infinity}} At the upper root we also cut off at a large radius $\Lambda$ and let $\Lambda\to\infty$. The argument is the analogue of Lemma~\ref{lem:innercut} at infinity.

\begin{lemma}\label{lem:outercut}
Let $p\ge2$, $\gamma\in(0,\gamma_{c})$, $r\ge p^{*}_{s}$ and $M>0$, let $\omega=\omega_{\gamma}$, and let
$U\in\Dsp$ be a positive radially symmetric nonincreasing solution of \eqref{eq:P}, $v=|x|^{\gamma}U$ and
$H=H_{M}$. We put
\[
  \kappa_{\infty}(\rho):=\sup_{\tau\ge\rho}\tau^{\gamma^{*}}U(\tau),\quad
  A:=N-\tfrac{p\gamma}{2},\quad B:=N-(p-1)\gamma ,
\]
and for $\Lambda>0$ let $\theta_{\Lambda}(x):=\theta_{1}(\frac{x}{\Lambda})$ with $0\le\theta_{1}\le1$,
$\theta_{1}=1$ on $B_{1}$, $\theta_{1}=0$ outside $B_{2}$ and $|\nabla\theta_{1}|\le C_{0}$. Then $A,B>0$,
$\kappa_{\infty}$ is bounded, nonincreasing and $\kappa_{\infty}(\rho)\to0$ as $\rho\to\infty$, and
\begin{equation}\label{eq:outercut}
  I^{\infty}_{p,\Lambda}+I^{\infty}_{2,\Lambda}\longrightarrow0\quad(\Lambda\to\infty),
\end{equation}
where, with $H_{x}:=H(v(x))$ and $\omega_{x}:=\omega(x)$,
\begin{align*}
    & I^{\infty}_{p,\Lambda}:=\iint_{\R^{2N}}\frac{(\omega_{x}\omega_{y})^{\frac{p}{2}}\big(H_{x}+H_{y}\big)^{p} |\theta_{\Lambda}(x)-\theta_{\Lambda}(y)|^{p}}{|x-y|^{N+sp}}\dxy, \text{ and } \\
    & I^{\infty}_{2,\Lambda}:=\iint_{\R^{2N}}\frac{\omega_{x}\omega_{y}|\omega_{x}-\omega_{y}|^{p-2}
  \big(H_{x}+H_{y}\big)^{p}|\theta_{\Lambda}(x)-\theta_{\Lambda}(y)|^{2}}{|x-y|^{N+sp}}\dxy. 
\end{align*}
\end{lemma}

\begin{proof}
We first check that $A,B>0$. Since $\gamma<\gamma_{c}$ we have $(p-1)\gamma<N-sp<N$, so $B>sp>0$. For $A$,
$\tfrac{p\gamma}{2}<\tfrac{p\gamma_{c}}{2}=\tfrac{p(N-sp)}{2(p-1)}$, and
$\tfrac{p(N-sp)}{2(p-1)}<N$ is equivalent to $N(p-2)+sp^{2}>0$, which holds for $p\ge2$.

By
Lemma~\ref{lem:radial}-(ii), $\kappa_{\infty}$ is finite, nonincreasing, and tends to $0$ at infinity.
Since $H_{M}(t)\le M^{\frac{r-p}{p}}t$ and $v(x)=|x|^{\gamma-\gamma^{*}}\big(|x|^{\gamma^{*}}U(x)\big)$,
\begin{equation}\label{eq:Hkappainf}
  H_{x}^{p} \le  M^{r-p}\,\kappa_{\infty}(|x|)^{p}\,|x|^{p(\gamma-\gamma^{*})} .
\end{equation}
If $\theta_{\Lambda}(x)\neq\theta_{\Lambda}(y)$ then $\max\{|x|,|y|\}\ge\Lambda$ and
$\min\{|x|,|y|\}\le2\Lambda$. It is therefore enough to treat the region where $|x|,|y|\ge\frac{\Lambda}{2}$ and
the region where $|x|\ge\Lambda$ and $|y|\le\frac{\Lambda}{2}$, together with the symmetric estimate obtained by interchanging $x$ and $y$.

We first treat the region where $|x|,|y|\ge\frac{\Lambda}{2}$. The substitution $x=\Lambda\xi$, $y=\Lambda\zeta$ with $|\xi|,|\zeta|\ge \frac12$
multiplies $\dxy$ by $\Lambda^{2N}$, the weight $(\omega_{x}\omega_{y})^{\frac{p}{2}}$ by $\Lambda^{-p\gamma}$,
the kernel by $\Lambda^{-N-sp}$ and, by \eqref{eq:Hkappainf} and the monotonicity of $\kappa_{\infty}$,
the factor $(H_{x}+H_{y})^{p}$ by $M^{r-p}\kappa_{\infty}(\frac{\Lambda}{2})^{p}\Lambda^{p(\gamma-\gamma^{*})}$.
The total power of $\Lambda$ is
\[
  2N-p\gamma-N-sp+p(\gamma-\gamma^{*})=N-sp-p\gamma^{*}=0 ,
\]
exactly as in Lemma~\ref{lem:innercut}. The scaling is critical and no smallness comes from it. What
remains is $M^{r-p}\kappa_{\infty}(\frac{\Lambda}{2})^{p}$ times the fixed integral
\[
  \mathcal{I}_{p}^{\infty}:=\iint_{\{|\xi|\ge\frac12,|\zeta|\ge\frac12\}}
  \frac{(|\xi||\zeta|)^{-\frac{p\gamma}{2}}\big(|\xi|^{p(\gamma-\gamma^{*})}+|\zeta|^{p(\gamma-\gamma^{*})}\big)
  |\theta_{1}(\xi)-\theta_{1}(\zeta)|^{p}}{|\xi-\zeta|^{N+sp}}\d\xi\d\zeta ,
\]
which is finite. Indeed, near the diagonal $|\theta_{1}(\xi)-\theta_{1}(\zeta)|\le C_{0}|\xi-\zeta|$
turns the kernel into $|\xi-\zeta|^{p-N-sp}$, integrable because $sp<p$, while all weights are bounded
on $\{\tfrac12\le\abs{\cdot}\le3\}$. Away from the diagonal one variable lies in $\{\frac12 \le \abs{\cdot}\le 2\}$. If
the other, say $\zeta$, is large, then $|\xi-\zeta|\simeq|\zeta|$ and the two radial exponents are
\[
  -\frac{p\gamma}{2}+p(\gamma-\gamma^{*})-N-sp+N-1=-A-1, \text{ and }  -\frac{p\gamma}{2}-N-sp+N-1=-\frac{p\gamma}{2}-sp-1 ,
\]
both $<-1$, using $p\gamma^{*}=N-sp$ and $A>0$. Hence this part of $I^{\infty}_{p,\Lambda}$ is at most
$CM^{r-p}\kappa_{\infty}(\frac{\Lambda}{2})^{p}$.

For $I^{\infty}_{2,\Lambda}$ the weight $\omega_{x}\omega_{y}|\omega_{x}-\omega_{y}|^{p-2}$ similarly scales like
$\Lambda^{-p\gamma}$. Near the diagonal
$|\omega_{x}-\omega_{y}|\le C\Lambda^{-\gamma-1}|x-y|$ and
$|\theta_{\Lambda}(x)-\theta_{\Lambda}(y)|\le C_{0}|x-y|/\Lambda$ again turn the kernel into
$|x-y|^{p-N-sp}$. Away from the diagonal, if $|\zeta|\ge2|\xi|$ then $\omega_{\zeta}\le\omega_{\xi}$ and
$|\omega_{\xi}-\omega_{\zeta}|\le\omega_{\xi}$, so the weight is at most
$\omega_{\xi}^{p-1}\omega_{\zeta}=|\xi|^{-(p-1)\gamma}|\zeta|^{-\gamma}$, and the radial exponents become
$-\gamma+p(\gamma-\gamma^{*})-N-sp+N-1=-B-1$ and $-\gamma-sp-1$, both $<-1$ because $B>0$.

Next we treat the region where $|x|\ge\Lambda$ and $|y|\le\frac{\Lambda}{2}$. Then $|x-y|\ge|x|/2$ and
$|\theta_{\Lambda}(x)-\theta_{\Lambda}(y)|\le1$. For the contribution of $H_{x}^{p}$ to
$I^{\infty}_{p,\Lambda}$, \eqref{eq:Hkappainf} and $\int_{|y|\le\frac{\Lambda}{2}}|y|^{- \frac{p\gamma}{2}}\dy=C\Lambda^{A}$
give
\[
  CM^{r-p}\Lambda^{A}\int_{\Lambda}^{\infty}\kappa_{\infty}(\rho)^{p}\rho^{-A-1}\d\rho
   \le  \frac{C}{A}\,M^{r-p}\kappa_{\infty}(\Lambda)^{p},
\]
because the radial exponent is $-\tfrac{p\gamma}{2}+p(\gamma-\gamma^{*})-N-sp+N-1=-A-1$ and
$\kappa_{\infty}$ is nonincreasing. This tends to $0$.

For the contribution of $H_{y}^{p}$ we use $\int_{|x|\ge\Lambda}|x|^{-\frac{p\gamma}{2}-N-sp}\dx=C\Lambda^{-A'}$
with $A':=\tfrac{p\gamma}{2}+sp>0$, so that this contribution is at most
\[
  CM^{r-p}\,\Lambda^{-A'}\int_{|y|\le\frac{\Lambda}{2}}\big(|y|^{\gamma^{*}}U(y)\big)^{p}
  |y|^{-\frac{p\gamma}{2}+p(\gamma-\gamma^{*})}\dy ,
\]
whose $y$-density has radial exponent
$-\tfrac{p\gamma}{2}+p\gamma-p\gamma^{*}+N-1=\tfrac{p\gamma}{2}+sp-1=A'-1>-1$. We fix
$\Lambda_{0}>0$. Splitting the $y$-integral at $\Lambda_{0}$, using
$|y|^{\gamma^{*}}U(y)\le C$ on $B_{\Lambda_{0}}$ from Lemma~\ref{lem:radial}-(i) and
$|y|^{\gamma^{*}}U(y)\le\kappa_{\infty}(\Lambda_{0})$ on $B_{\frac{\Lambda}{2}}\setminus B_{\Lambda_{0}}$, we
obtain the bound
\[
  CM^{r-p}\Big[C(\Lambda_{0})\Lambda^{-A'}+\kappa_{\infty}(\Lambda_{0})^{p}\Big],
\]
so its $\limsup$ as $\Lambda\to\infty$ is at most $CM^{r-p}\kappa_{\infty}(\Lambda_{0})^{p}$. Letting
$\Lambda_{0}\to\infty$ gives $0$.

For $I^{\infty}_{2,\Lambda}$ in this region, $|y|\le \frac{\Lambda}{2}<|x|$ gives $\omega_{y}\ge\omega_{x}$ and
$|\omega_{x}-\omega_{y}|\le\omega_{y}$, so
$\omega_{x}\omega_{y}|\omega_{x}-\omega_{y}|^{p-2}\le\omega_{x}\omega_{y}^{p-1}
=|x|^{-\gamma}|y|^{-(p-1)\gamma}$. The two contributions are estimated exactly as above with $A$ replaced
by $B$. For $H_{x}^{p}$ the radial exponent in $x$ is
$-\gamma+p(\gamma-\gamma^{*})-N-sp+N-1=-B-1$ while
$\int_{|y|\le\frac{\Lambda}{2}}|y|^{-(p-1)\gamma}\dy=C\Lambda^{B}$. For $H_{y}^{p}$ we have
$\int_{|x|\ge\Lambda}|x|^{-\gamma-N-sp}\dx=C\Lambda^{-\gamma-sp}$ and the $y$-density has radial exponent
$-(p-1)\gamma+p\gamma-p\gamma^{*}+N-1=\gamma+sp-1>-1$, so the same splitting at $\Lambda_{0}$ applies.
Adding the four bounds proves \eqref{eq:outercut}.
\end{proof}

\subsubsection{\textbf{The remainder of the cutoff at scale $R$}} The cutoff $\eta_{R}$ is \emph{not} removed.
It is the localization of the Moser step, and we need its remainder to be finite and bounded
independently of the truncation parameter $M$. This follows from the second identity in
\eqref{eq:GHid}, $H_{M}(t)^{p}\le t^{r}$, which makes every bound below $M$-free.

\begin{lemma}\label{lem:cutrem}
Let $p\ge2$, $\gamma\in(0,\gamma_{c})$, $r\ge p^{*}_{s}$ and $M>0$, and let $U,v,H=H_{M}$ be as in
Lemma~\ref{lem:outercut}. We set
\[
  \tau:=\max\Big\{\tfrac{p\gamma}{2},\,(p-1)\gamma\Big\},\quad \sigma:=\gamma .
\]
Then
\begin{equation}\label{eq:expchain}
  0<\tau<\min\{N, \gamma p^{*}_{s}\}\quad\text{and}\quad \gamma p^{*}_{s}<\sigma+N+sp .
\end{equation}
Let $R>0$ and let $\eta_{R}$ be Lipschitz with $0\le\eta_{R}\le1$, $|\nabla\eta_{R}|\le \frac{C_{0}}{R}$,
constant on $B_{R}$ and constant on $\R^{N}\setminus B_{2R}$. Then, with
\[
  \Pi_{R}(x):=\begin{cases}
  \left(\frac{|x|}{R}\right)^{-\tau}, & |x|\le4R,\\[2pt]
  \left(\frac{|x|}{R}\right)^{-\sigma-N-sp}, & |x|>4R,
  \end{cases}
\]
there is $C=C(N,p,s,\gamma,C_{0})$, independent of $M,R$ and $r$, such that
\begin{equation}\label{eq:cutrem}
  \mathcal{R}_{p}[\eta_{R}]+\mathcal{R}_{2}[\eta_{R}]
   \le  C R^{-p\gamma-sp}\int_{\R^{N}}v(x)^{r}\,\Pi_{R}(x)\dx ,
\end{equation}
where $\mathcal{R}_{p}[\eta_{R}]$ and $\mathcal{R}_{2}[\eta_{R}]$ are $I^{\infty}_{p,\Lambda}$ and
$I^{\infty}_{2,\Lambda}$ of Lemma~\ref{lem:outercut} with $\theta_{\Lambda}$ replaced by $\eta_{R}$.
Moreover,
\begin{align*}
    \Pi_{R}(x)\le4^{\gamma p^{*}_{s}}R^{\gamma p^{*}_{s}}|x|^{-\gamma p^{*}_{s}}, \text{ for every } x \neq 0.
\end{align*}
Furthermore, the right-hand side of \eqref{eq:cutrem} is finite whenever
$\int_{\Sigma'}|x|^{-\gamma p^{*}_{s}}v^{r}\dx<\infty$, where $\Sigma':=B_{4R}$ if $\gamma<\gamma^{*}$ and
$\Sigma':=\R^{N}\setminus B_{R/2}$ if $\gamma>\gamma^{*}$.
\end{lemma}

\begin{proof}
We first prove \eqref{eq:expchain}. As in Lemma~\ref{lem:outercut}, $\tfrac{p\gamma}{2}<N$
and $(p-1)\gamma<N-sp<N$, so $\tau<N$. Since $p^{*}_{s}>\frac{p}{2}$ and $p^{*}_{s}>p-1$ we get $\tau<\gamma p^{*}_{s}$. For the last
inequality, $\gamma p^{*}_{s}<\gamma+N+sp$ is $\gamma(p^{*}_{s}-1)<N+sp$, and
$\gamma_{c}(p^{*}_{s}-1)=\frac{N-sp}{p-1}\frac{Np-N+sp}{N-sp}=\frac{Np-N+sp}{p-1}$, while
$\tfrac{Np-N+sp}{p-1}\le N+sp$ is equivalent to $2sp\le sp^{2}$, that is to $p\ge2$. As $\gamma<\gamma_{c}$, we must have $\gamma(p^{*}_{s}-1)<N+sp$.

Next, we prove the bound $\Pi_{R}\le4^{\gamma p^{*}_{s}}R^{\gamma p^{*}_{s}}|x|^{-\gamma p^{*}_{s}}$, which follows from \eqref{eq:expchain}.
For $|x|\le R$ we have $(\frac{|x|}{R})^{-\tau}\le(\frac{|x|}{R})^{-\gamma p^{*}_{s}}$ because $\tau<\gamma p^{*}_{s}$ and $\frac{|x|}{R}\le1$.
For $R<|x|\le4R$ the same comparison costs the factor $4^{\gamma p^{*}_{s}-\tau}$. For $|x|>4R$,
$(\frac{|x|}{R})^{-\sigma-N-sp}\le(\frac{|x|}{R})^{-\gamma p^{*}_{s}}$ since $\frac{|x|}{R}>1$ and $\sigma+N+sp>\gamma p^{*}_{s}$.

Finally, we prove \eqref{eq:cutrem}. By \eqref{eq:GHid}, $H_{x}^{p}\le v(x)^{r}$ uniformly in
$M$, so $(H_{x}+H_{y})^{p}\le2^{p-1}\big(v(x)^{r}+v(y)^{r}\big)$, and by the symmetry of the integrands
in $(x,y)$,
\begin{align*}
    & \mathcal{R}_{p}[\eta_{R}]+\mathcal{R}_{2}[\eta_{R}]
   \le  2^{p}\int_{\R^{N}}v(x)^{r}\,\mathcal{N}[\eta_{R}](x)\dx, \text{ where } \\
   & \mathcal{N}[\eta_{R}](x):=\int_{\R^{N}}
  \frac{(\omega_{x}\omega_{y})^{\frac{p}{2}}|\eta_{R}(x)-\eta_{R}(y)|^{p}
  +\omega_{x}\omega_{y}|\omega_{x}-\omega_{y}|^{p-2}|\eta_{R}(x)-\eta_{R}(y)|^{2}}{|x-y|^{N+sp}}\dy. 
\end{align*}
Since $\eta_{R}(x)=\eta_{1}(\frac{x}{R})$ for a fixed profile $\eta_{1}$, and both weights are homogeneous of
degree $-p\gamma$, the substitution $y=Rz$ gives the exact scaling
$\mathcal{N}[\eta_{R}](R\xi)=R^{-p\gamma-sp}\,\mathcal{N}_{1}(\xi)$. So \eqref{eq:cutrem} follows from
$$\mathcal{N}_{1}(\xi)\le C\Pi_{1}(\xi), \quad \xi \in \R^N,$$ and we check this bound in three regions.

Firstly, let $|\xi|\le\frac12$. Then $\eta_{1}(\xi)$ is the interior constant, so
$\eta_{1}(z)\neq\eta_{1}(\xi)$ forces $|z|\ge1\ge2|\xi|$ and hence $|\xi-z|\ge \frac{|z|}{2}$. Bounding
$|\eta_{1}(\xi)-\eta_{1}(z)|\le1$, the first term is at most
\[
  C|\xi|^{-\frac{p\gamma}{2}}\int_{\{|z|\ge1 \}}|z|^{-\frac{p\gamma}{2}-N-sp}\dz=C|\xi|^{-\frac{p\gamma}{2}} .
\]
For the second, $|\xi|<|z|$ gives $\omega_{\xi}>\omega_{z}$ and $|\omega_{\xi}-\omega_{z}|\le\omega_{\xi}$,
so the weight is at most $\omega_{\xi}^{p-1}\omega_{z}$ and the term is at most
$$C|\xi|^{-(p-1)\gamma}\int_{\{|z|\ge1\}}|z|^{-\gamma-N-sp}\dz=C|\xi|^{-(p-1)\gamma}.$$ Hence
$\mathcal{N}_{1}(\xi)\le C|\xi|^{-\tau}$.

Now let $\frac12\le|\xi|\le4$, where all weights are bounded above and below. Near the diagonal we have
$|\eta_{1}(\xi)-\eta_{1}(z)|^{p}\le C|\xi-z|^{p}$, respectively
$|\omega_{\xi}-\omega_{z}|^{p-2}|\eta_{1}(\xi)-\eta_{1}(z)|^{2}\le C|\xi-z|^{p}$. This turns the kernel into
$|\xi-z|^{p-N-sp}$, which is integrable because $sp<p$. Away from the diagonal we argue as in the other two regions. Hence
$\mathcal{N}_{1}(\xi)\le C$, and $\Pi_{1}\ge4^{-\tau}$ there.

Finally, let $|\xi|>4$. Then $\eta_{1}(\xi)$ is the exterior constant, so
$\eta_{1}(z)\neq\eta_{1}(\xi)$ forces $|z|\le2\le \frac{|\xi|}{2}$ and $|\xi-z|\ge \frac{|\xi|}{2}$. The first term is at
most $$C|\xi|^{-\frac{p\gamma}{2}-N-sp}\int_{\{|z|\le2\}}|z|^{-\frac{p\gamma}{2}}\dz,$$ which is finite because
$\tfrac{p\gamma}{2}<N$. The second, using $|\omega_{\xi}-\omega_{z}|\le\omega_{z}$, is at most
$$C|\xi|^{-\gamma-N-sp}\int_{\{|z|\le2\}}|z|^{-(p-1)\gamma}\dz,$$ which is finite because $(p-1)\gamma<N$. Since
$\sigma=\gamma\le\tfrac{p\gamma}{2}$, this gives $\mathcal{N}_{1}(\xi)\le C|\xi|^{-\sigma-N-sp}$.

It remains to check that the right-hand side of \eqref{eq:cutrem} is finite. By the bound on $\Pi_{R}$, the part of the integral in \eqref{eq:cutrem} over
$\Sigma'$ is at most $C R^{\gamma p^{*}_{s}}\int_{\Sigma'}|x|^{-\gamma p^{*}_{s}}v^{r}\dx$. For the part over
$\R^{N}\setminus\Sigma'$ we use the unconditional bound $U(x)\le C|x|^{-\gamma^{*}}$ of
Lemma~\ref{lem:radial}-(i), that is $v(x)\le C|x|^{\gamma-\gamma^{*}}$. If $\gamma<\gamma^{*}$, then
$\R^{N}\setminus\Sigma'=\R^{N}\setminus B_{4R}$ and using \eqref{eq:expchain}, the radial exponent of $v^{r}\Pi_{R}$ is
$r(\gamma-\gamma^{*})-\sigma-N-sp+N-1<-1$ because $r(\gamma-\gamma^{*})\le0$ and $\sigma+sp>0$. If
$\gamma>\gamma^{*}$, then $\R^{N}\setminus\Sigma'=B_{R/2}$ and again using \eqref{eq:expchain}, the radial exponent is
$r(\gamma-\gamma^{*})-\tau+N-1>-1$ because $\tau<N$. In both cases, integrals give finite values.
\end{proof}

\subsubsection{\textbf{Comparison of $\mathcal{E}_{\gamma}$ with the Gagliardo seminorm}} The next lemma shows that $\mathcal{E}_{\gamma}$ is finite on functions supported in a compact annulus, and that \eqref{eq:WS} applies to them.

\begin{lemma}\label{lem:dens}
Let $p\ge2$ and $\gamma\in(0,\gamma_{c})$, so that $\tfrac{p\gamma}{2}<N$ by Lemma~\ref{lem:outercut}. Let
$K=\{a\le|x|\le b\}$ with $0<a<b<\infty$ and put $K':=\{\frac{a}{2}\le|x|\le2b\}$. Then there is
$C_{K}=C(N,p,s,\gamma,a,b)$ such that
\begin{equation}\label{eq:Egamloc}
  \mathcal{E}_{\gamma}(h) \le  C_{K}\Big([h]_{s,p}^{p}+\|h\|_{L^{p}(\R^{N})}^{p}\Big)
  \quad\text{for every }h\in W^{s,p}(\R^{N})\text{ with }\operatorname{supp}h\subseteq K .
\end{equation}
Consequently, if $w\in W^{s,p}(\R^{N})\cap L^{\infty}$ is supported in $K$, then there exists
$w_{n}\in\cc(\R^{N}\setminus\{0\})$, supported in $K'$, with
\begin{equation}\label{eq:densconv}
  \lim_{n \ra \infty} \mathcal{E}_{\gamma}(w_{n}) = \mathcal{E}_{\gamma}(w),\quad
  \lim_{n \ra \infty} \int_{\R^{N}}|x|^{-\gamma p^{*}_{s}}|w_{n}|^{p^{*}_{s}}\dx = \int_{\R^{N}}|x|^{-\gamma p^{*}_{s}}|w|^{p^{*}_{s}}\dx,
\end{equation}
and therefore \eqref{eq:WS} holds for $w$.
\end{lemma}

\begin{proof}
We first prove \eqref{eq:Egamloc}. We split the domain of \eqref{eq:intro-Egamma} into three parts. If $x,y\in K'$, both weights lie between $(2b)^{- \frac{p\gamma}{2}}$ and $( \frac{a}{2} )^{- \frac{p\gamma}{2}}$, so this part is at
most $C_{K}[h]^{p}_{s,p}$. If $x\in K'$ and $y\notin K'$ (and, symmetrically, the reverse, which doubles the constant), then
$h(y)=0$ because $\operatorname{supp}h\subseteq K\subset K'$, so the numerator is $|h(x)|^{p}$ and may be
taken to vanish unless $x\in K$. Since $\operatorname{dist}(K,\R^{N}\setminus K')\ge\min\{ \frac{a}{2} ,b\}>0$,
there is no diagonal singularity, and this part is at most
\[
  \int_{K}|h(x)|^{p}|x|^{-\frac{p\gamma}{2}}
  \bigg(\int_{\R^{N}\setminus K'}\frac{|y|^{-\frac{p\gamma}{2}}}{|x-y|^{N+sp}}\dy\bigg)\dx .
\]
The inner integral is bounded uniformly for $x\in K$. Indeed $\R^{N}\setminus K'=\{|y|< \frac{a}{2} \}\cup\{|y|>2b\}$.
On $\{|y|< \frac{a}{2} \}$ we have $|x-y|\ge  \frac{a}{2} $, so that piece is at most
$$\left( \frac{a}{2} \right)^{-N-sp}\int_{ \{ |y|< \frac{a}{2} \} }|y|^{- \frac{p\gamma}{2}}\dy<\infty,$$
where the finiteness comes using $\frac{p\gamma}{2}<N$. 
On $\{|y|>2b\}$ we have $|x-y| \ge \frac{|y|}{2}$, so that piece is at most
$$2^{N+sp}\int_{\{|y|>2b\}}|y|^{-\frac{p\gamma}{2}-N-sp}\dy<\infty,$$ the kernel alone giving convergence since
$\tfrac{p\gamma}{2}+sp>0$. As $|x|^{- \frac{p\gamma}{2}}\le(a)^{- \frac{p\gamma}{2}}$ on $K$, this part is at most
$C_{K}\|h\|^{p}_{L^{p}}$.
If $x,y\notin K'$, then $h(x)=h(y)=0$ and the contribution vanishes. Adding the three parts gives
\eqref{eq:Egamloc}.

Next, we prove \eqref{eq:densconv}. Let $\rho_{n}$ be a standard mollifier of radius
$\delta_{n}<\min\{ \frac{a}{2} ,b\}$, $\delta_{n}\to0$, and $w_{n}:=w*\rho_{n}$. Then
$w_{n}\in\cc(\R^{N}\setminus\{0\})$, $\operatorname{supp}w_{n}\subseteq K'$, and $w_{n}\to w$ in
$W^{s,p}(\R^{N})$. Since $h_{n}:=w_{n}-w$ is supported in $K'$, \eqref{eq:Egamloc} applied with $K'$ in
place of $K$ gives
\[
  \mathcal{E}_{\gamma}(h_{n}) \le  C_{K'}\Big([h_{n}]^{p}_{s,p}+\|h_{n}\|^{p}_{L^{p}}\Big)
  \ \longrightarrow\ 0 .
\]
Now $\mathcal{E}_{\gamma}^{1/p}$ is a seminorm, being the $L^{p}\big(\R^{2N},\d\nu\big)$ norm of
$\big(h(x)-h(y)\big)|x-y|^{-\frac{N+sp}{p}}|x|^{-\frac{\gamma}{2}}|y|^{-\frac{\gamma}{2}}$,
so the triangle inequality gives
$\big|\mathcal{E}_{\gamma}(w_{n})^{1/p}-\mathcal{E}_{\gamma}(w)^{1/p}\big|\le\mathcal{E}_{\gamma}(h_{n})^{1/p}\to0$,
which is the first limit in \eqref{eq:densconv}. For the second, $w_{n}\to w$ in $W^{s,p}(\R^{N})$ implies
$w_{n}\to w$ in $L^{p^{*}_{s}}(\R^{N})$ by the Sobolev inequality, and the weight $|x|^{-\gamma p^{*}_{s}}$ is bounded
above and below on $K'$, which contains every support involved. Finally, applying \eqref{eq:WS} to each
$w_{n}$ and passing to the limit gives \eqref{eq:WS} for $w$.
\end{proof}

\subsubsection{\textbf{The weighted Caccioppoli inequality}}

We now prove the weighted Caccioppoli inequality, with the auxiliary cutoff still present.
\begin{proposition}\label{prop:cacc}
Let $p\ge2$, $0<\mu<\bar\mu$, $\gamma\in\{\gamma_{-}(\mu),\gamma_{+}(\mu)\}$, $r\ge p^{*}_{s}$ and
$M>0$, and let $U\in\Dsp$ be a positive radially symmetric nonincreasing weak solution of \eqref{eq:P},
$v=v_{\gamma}=|x|^{\gamma}U$. We take $\eta_{R}$ as in Lemma~\ref{lem:cutrem} with the choice corresponding to the end under consideration. We put
\begin{equation}\label{eq:etaends}
\begin{aligned}
  \gamma=\gamma_{-}:\quad &\eta_{R}=1\ \text{on }B_{R},\ \ \eta_{R}=0\ \text{outside }B_{2R},
  &&\eta:=\zeta_{\eps}\,\eta_{R},\\
  \gamma=\gamma_{+}:\quad &\eta_{R}=0\ \text{on }B_{R},\ \ \eta_{R}=1\ \text{outside }B_{2R},
  &&\eta:=\eta_{R}\,\theta_{\Lambda},
\end{aligned}
\end{equation}
with $\zeta_{\eps}$ as in Lemma~\ref{lem:innercut} and $\theta_{\Lambda}$ as in
Lemma~\ref{lem:outercut}, and $w_{\eps,\Lambda}:=\eta H_{M}(v)$. We omit the unused parameter in each case. We write
\begin{equation}\label{eq:Raux}
  \mathcal{R}_{\mathrm{aux}}:=\begin{cases}
  I_{p,\eps}+I_{2,\eps}, & \gamma=\gamma_{-},\\[2pt]
  I^{\infty}_{p,\Lambda}+I^{\infty}_{2,\Lambda}, & \gamma=\gamma_{+}.
  \end{cases}
\end{equation}
Then \eqref{eq:cacc} below holds with the constants $c_r>0$ and $C_r<\infty$ of Proposition~\ref{prop:LC}, which are independent of $M$, $\eps$ and $\Lambda$,
\begin{equation}\label{eq:cacc}
  c_r\,\mathcal{E}_{\gamma}\big(w_{\eps,\Lambda}\big)
   \le  \int_{\R^{N}}U^{P_s}\big(\omega_{\gamma}w_{\eps,\Lambda}\big)^{p}\dx
  +C_r\Big(\mathcal{R}_{\mathrm{aux}}
  +\mathcal{R}_{p}[\eta_{R}]+\mathcal{R}_{2}[\eta_{R}]\Big).
\end{equation}
Moreover $w_{\eps,\Lambda}$ is bounded, belongs to $W^{s,p}(\R^{N})$ and is supported in a compact
annulus contained in $\R^{N}\setminus\{0\}$, and $\mathcal{E}_{\gamma}(w_{\eps,\Lambda})<\infty$.
\end{proposition}

\begin{proof}
In either case of \eqref{eq:etaends} the cutoff $\eta$ is Lipschitz, takes values in $[0,1]$ and is
supported in a compact annulus. This annulus is $\overline{B_{2R}}\setminus B_{\eps}$ at the lower root, and
$\overline{B_{2\Lambda}}\setminus B_{R}$ at the upper root. By Lemma~\ref{lem:admiss}, the pair
$\varphi=\omega_{\gamma}\eta^{p}G_{M}(v)$, $v^{p-1}\varphi$ is therefore admissible and
\eqref{eq:routeid} holds. The same lemma gives the stated properties of $w_{\eps,\Lambda}$, and
$\mathcal{E}_{\gamma}(w_{\eps,\Lambda})<\infty$ by \eqref{eq:Egamloc} of Lemma~\ref{lem:dens}.

Writing $a=\omega_{\gamma}(x)$, $b=\omega_{\gamma}(y)$, $X=v(x)$, $Y=v(y)$, $e=\eta(x)$, $f=\eta(y)$ and
integrating \eqref{eq:LC} against $|x-y|^{-N-sp}\dxy$, the left-hand side of \eqref{eq:routeid} is
bounded below by $c_r\mathcal{E}_{\gamma}(w_{\eps,\Lambda})$ minus the two cutoff remainders, while
the right-hand side is $\int_{\R^N} U^{P_s}(\omega_{\gamma}w_{\eps,\Lambda})^{p}\dx$ by \eqref{eq:rhsform}. For
$p=2$ this is as in Remark~\ref{rem:pis2}, the two remainders then coinciding. Since $\eta$ is a
product of two functions with values in $[0,1]$,
\[
  |\eta(x)-\eta(y)| \le  |\eta_{R}(x)-\eta_{R}(y)|+
  \begin{cases}|\zeta_{\eps}(x)-\zeta_{\eps}(y)|, & \gamma=\gamma_{-},\\
  |\theta_{\Lambda}(x)-\theta_{\Lambda}(y)|, & \gamma=\gamma_{+},\end{cases}
\]
so by convexity the two remainders are bounded by $2^{p-1}$, respectively $2$, times the sum of the
corresponding two expressions, which are exactly $\mathcal{R}_{p}[\eta_{R}]+\mathcal{R}_{2}[\eta_{R}]$
and $\mathcal{R}_{\mathrm{aux}}$. Every constant is independent of $M$, $\eps$ and $\Lambda$, using
Corollary~\ref{cor:truncboth} and Proposition~\ref{prop:LC}.
\end{proof}

\noi In Proposition~\ref{prop:cacc} the auxiliary cutoff is kept, because \eqref{eq:WS} is applied only to functions with compact support in $\R^{N}\setminus\{0\}$. We remove it in the next corollary.

The next corollary gives one step of the iteration.
\begin{corollary}\label{cor:onestep}
In the setting of Proposition~\ref{prop:cacc}, put $\Sigma:=\operatorname{supp}\eta_{R}$, so that
$\Sigma\subseteq B_{2R}$ if $\gamma=\gamma_{-}$ and $\Sigma\subseteq\R^{N}\setminus B_{R}$ if
$\gamma=\gamma_{+}$, and write
\begin{equation}\label{eq:Rel}
  \mathcal{R}_{\eps,\Lambda}:=\mathcal{R}_{\mathrm{aux}}
  +\mathcal{R}_{p}[\eta_{R}]+\mathcal{R}_{2}[\eta_{R}] ,
\end{equation}
with $\mathcal{R}_{\mathrm{aux}}$ as in \eqref{eq:Raux}. We assume $R$ has been chosen so that
\begin{equation}\label{eq:smallthr}
  \Big(\int_{\Sigma}U^{p^{*}_{s}}\dx\Big)^{\frac{P_s}{p^{*}_{s}}} \le  \frac12c_rS_{\gamma} .
\end{equation}
If $\int_{\Sigma'}|x|^{-\gamma p^{*}_{s}}v^{r}\dx<\infty$, with $\Sigma'$ as in Lemma~\ref{lem:cutrem}, then
\begin{equation}\label{eq:onestep}
  \left(\int_{\R^{N}}|x|^{-\gamma p^{*}_{s}}\left|\eta_{R}H_{M}(v)\right|^{p^{*}_{s}}\dx\right)^{\frac{p}{p^{*}_{s}}}
   \le\frac{2C_r}{c_rS_{\gamma}}
  \left(\mathcal{R}_{p}[\eta_{R}]+\mathcal{R}_{2}[\eta_{R}]\right) < \infty,
\end{equation}
with a bound independent of $M$.
\end{corollary}

\begin{proof}
We fix the auxiliary parameter and $M$, and set
\[
  Y_{\eps,\Lambda}:=\left(\int_{\R^{N}}|x|^{-\gamma p^{*}_{s}}\left|w_{\eps,\Lambda}\right|^{p^{*}_{s}}\dx\right)^{\frac{p}{p^{*}_{s}}},
\]
which is finite because $w_{\eps,\Lambda}$ is bounded with support in a compact annulus. By Proposition~\ref{prop:cacc}, $w_{\eps,\Lambda}\in
W^{s,p}(\R^{N})\cap L^{\infty}$ is supported in a compact annulus $K\Subset\R^{N}\setminus\{0\}$, so
Lemma~\ref{lem:dens} applies and gives \eqref{eq:WS} for $w_{\eps,\Lambda}$, that is
\[
  S_{\gamma} Y_{\eps,\Lambda} \le  \mathcal{E}_{\gamma}\big(w_{\eps,\Lambda}\big) .
\]
At $\gamma=\gamma_{+}$ the inequality \eqref{eq:WS} follows from Lemma~\ref{lem:inv}. 

With exponents $\frac{p^{*}_{s}}{P_s}$ and $\frac{p^{*}_{s}}{p}$, and $\omega_{\gamma}^{p^{*}_{s}}=|x|^{-\gamma p^{*}_{s}}$,
\[
  \int_{\R^{N}}U^{P_s}\big(\omega_{\gamma}w_{\eps,\Lambda}\big)^{p}\dx
   \le  \Big(\int_{\Sigma}U^{p^{*}_{s}}\dx\Big)^{\frac{P_s}{p^{*}_{s}}}\,Y_{\eps,\Lambda} ,
\]
since $\operatorname{supp}w_{\eps,\Lambda}\subseteq\operatorname{supp}\eta\subseteq\Sigma$. In
particular the first factor is independent of the auxiliary parameter. Combining the above with \eqref{eq:cacc},
\[
c_rS_{\gamma} Y_{\eps,\Lambda} \le  \Big(\int_{\Sigma}U^{p^{*}_{s}}\dx\Big)^{\frac{P_s}{p^{*}_{s}}}Y_{\eps,\Lambda}
+C_r \mathcal{R}_{\eps,\Lambda},
\]
By \eqref{eq:smallthr}, the first term on the right is at most half of the left-hand side. Since $Y_{\eps,\Lambda}<\infty$, we may subtract it. This gives
\begin{equation}\label{eq:Ybound}
  Y_{\eps,\Lambda} \le \frac{2C_r}{c_rS_{\gamma}} \mathcal{R}_{\eps,\Lambda}.
\end{equation}

 The cutoff is removed only in \eqref{eq:Ybound}, and there is a single limit to take at each end. We let $\eps\downarrow0$ at
$\gamma_{-}$, where $\mathcal{R}_{\mathrm{aux}}=I_{p,\eps}+I_{2,\eps}\to0$ by \eqref{eq:innercut} and
$\gamma_{-}<\gamma^{*}$ is exactly the hypothesis of Lemma~\ref{lem:innercut}. Also
$\Lambda\uparrow\infty$ at $\gamma_{+}$, where
$\mathcal{R}_{\mathrm{aux}}=I^{\infty}_{p,\Lambda}+I^{\infty}_{2,\Lambda}\to0$ by \eqref{eq:outercut},
which holds for every $\gamma\in(0,\gamma_{c})$. What remains in either case is
$\mathcal{R}_{p}[\eta_{R}]+\mathcal{R}_{2}[\eta_{R}]$, finite and independent of $M$ by
Lemma~\ref{lem:cutrem}. On the left, $w_{\eps,\Lambda}\to\eta_{R}H_{M}(v)$ pointwise on
$\R^{N}\setminus\{0\}$ along that limit, so Fatou's lemma gives
\[
  \left(\int_{\R^{N}}|x|^{-\gamma p^{*}_{s}}\big|\eta_{R}H_{M}(v)\big|^{p^{*}_{s}}\dx \right)^{\frac{p}{p^{*}_{s}}}
   \le  \liminf_{\eps \ra 0, \Lambda \ra \infty} Y_{\eps,\Lambda} ,
\]
and \eqref{eq:onestep} follows from \eqref{eq:Ybound}.
\end{proof}

\subsection{Proof of \eqref{E0} and \eqref{E1}}\label{sec:Mproof}
We now combine the previous estimates. The iteration is run for a finite number of steps.
\begin{theorem}\label{thm:cond}
Let $p\ge2$ and $0<\mu<\bar\mu$, and let $U\in\Dsp$ be a
positive radially symmetric nonincreasing weak solution of \eqref{eq:P}. Then $U$ satisfies \eqref{E0}
and \eqref{E1}, and therefore
\[
  U(x)\le C|x|^{-\gamma_{-}(\mu)}\ \text{ as }|x|\to0,
  \quad
  U(x)\le C|x|^{-\gamma_{+}(\mu)}\ \text{ as }|x|\to\infty .
\]
\end{theorem}

\begin{proof}
{Take $\gamma=\gamma_{-}$ near the origin and $\gamma=\gamma_{+}$ at infinity, and write $v=v_{\gamma}$.

\noi\emph{\textbf{One step of the iteration.}} Let $r\ge p^{*}_{s}$. We claim the following:

There is $R=R(r)>0$ such that, with $\Sigma'$ as in Lemma~\ref{lem:cutrem}, and with $\Sigma''=B_{R}$ at the origin and $\Sigma''=\R^{N}\setminus B_{2R}$ at infinity,
\begin{equation}\label{eq:onestepgain}
  v\in L^{r}\big(\Sigma',|x|^{-\gamma p^{*}_{s}}\dx\big) 
  \Longrightarrow 
  v\in L^{r\chi}\big(\Sigma'',|x|^{-\gamma p^{*}_{s}}\dx\big).
\end{equation}
Moreover, $R(r)$ may be replaced by any smaller radius at the origin, and by any larger radius at infinity. 

Indeed, fix $M>0$. By Lemma~\ref{lem:admiss} and Proposition~\ref{prop:cacc}, the test function $\varphi=\omega_{\gamma}\eta^{p}G_{M}(v)$, with $\eta$ as in \eqref{eq:etaends}, gives \eqref{eq:cacc} with constants independent of $M$ and of the auxiliary parameter. Since $U\in L^{p^{*}_{s}}(\R^{N})$, we have $\int_{B_{2R}}U^{p^{*}_{s}}\to0$ as $R\to0$ and $\int_{\R^{N}\setminus B_{R}}U^{p^{*}_{s}}\to0$ as $R\to\infty$, so we can choose $R=R(r)$ satisfying \eqref{eq:smallthr}, and \eqref{eq:smallthr} remains true for smaller (respectively larger) radii. We have to choose the radius after $r$. By Proposition~\ref{prop:trunc}-(iii), our constant $c_r$ is of order at most $r^{-p}$. Corollary~\ref{cor:onestep} and Lemma~\ref{lem:cutrem} now give
\[
  \Big(\int_{\R^{N}}|x|^{-\gamma p^{*}_{s}}\big|\eta_{R}H_{M}(v)\big|^{p^{*}_{s}}\dx\Big)^{\frac{p}{p^{*}_{s}}}
   \le  C(R,r)\Big(\int_{\Sigma'}|x|^{-\gamma p^{*}_{s}}v^{r}\dx+1\Big),
\]
with $C(R,r)$ independent of $M$. By \eqref{eq:GHid}, $H_{M}(t)\uparrow t^{\frac{r}{p}}$ as $M\uparrow\infty$. Since $r\frac{p^{*}_{s}}{p}=r\chi$ and $\eta_{R}\equiv1$ on $\Sigma''$, monotone convergence in $M$ proves the claim.

\noi\emph{\textbf{Iteration begins.}} To iterate, we consider $r_{k}:=p^{*}_{s}\chi^{k}$. The exponent $r_{0}=p^{*}_{s}$ is admissible, since $\int_{\R^N}|x|^{-\gamma p^{*}_{s}}v^{p^{*}_{s}}\dx=\int_{\R^N} U^{p^{*}_{s}}\dx<\infty$. Given $\eta>0$, we fix $k$ with $p^{*}_{s}(\gamma^{*}-\gamma_{-})/r_{k}<\eta$ and $p^{*}_{s}(\gamma_{+}-\gamma^{*})/r_{k}<\eta$. We apply the claim with $r=r_{0},r_{1},\dots,r_{k-1}$. We choose each radius after the exponent. At the origin we take it so small, and at infinity so large, that the new domain $\Sigma'$ lies in the previous domain $\Sigma''$. This gives $v\in L^{r_{k}}$ with weight $|x|^{-\gamma p^{*}_{s}}$ near the origin, respectively near infinity. We use only finitely many radii. This is why the iteration is stopped after $k$ steps. Lemma~\ref{lem:radexp} now gives $U(\rho)\le C_{\eta}\rho^{-\gamma_{-}-\eta}$ for small $\rho$ and $U(\rho)\le C_{\eta}\rho^{-\gamma_{+}+\eta}$ for large $\rho$, which are \eqref{E0} and \eqref{E1}. Since $U$ satisfies \eqref{M} by Lemma~\ref{lem:radial}-(ii), Theorem~\ref{thm:sharpupper} applies at both ends.}
\end{proof}

\subsection{Optimality of the characteristic exponents}\label{sec:optimal}
For $p\ge2$, Theorem~\ref{thm:cond} and Theorem~\ref{thm:radialmain} give the upper bounds
\eqref{eq:intro-final-bounds}. We now show that their exponents cannot be improved. For every
$\alpha<\gamma_{-}(\mu)$ and every $\beta\in(\gamma_{+}(\mu),\gamma_{c})$, a positive radially symmetric
nonincreasing solution satisfies $U\ge c_{\alpha}|x|^{-\alpha}$ near the origin and
$U\ge c_{\beta}|x|^{-\beta}$ at infinity. In both cases we compare $U$ with a truncated power. The power is
a subsolution of the Hardy equation, because $\Phi(\alpha)<\mu$ and $\Phi(\beta)<\mu$.

We first check that truncated powers have finite energy.
\begin{lemma}\label{lem:barenergy}
Let $R>0$, and let $\chi\in \C^{\infty}(\R^{N})$ be radially symmetric and nonincreasing, with $\chi=1$ on
$B_{R}$ and $\chi=0$ outside $B_{2R}$. Then the following hold:
\begin{enumerate}
\item[{\rm(i)}] If $0<\alpha<\gamma^{*}$, then $\chi|x|^{-\alpha}\in\Dsp$.
\item[{\rm (ii)}] If $\beta>\gamma^{*}$, then $\min\{1,|x|^{-\beta}\}\in\Dsp$.
\end{enumerate}
\end{lemma}

\begin{proof}
In both cases $w$ denotes the function, and $w$ is radially nonincreasing and nonnegative. By symmetry,
$[w]_{s,p}^{p}$ is twice the integral over the pairs with $|y|\ge|x|$. For such a pair,
$0\le w(x)-w(y)\le w(x)$. We divide the pairs according to the annulus containing the inner point $x$.

(ii) Here $w\in L^{p^{*}_{s}}$, because $\beta p^{*}_{s}>N$ exactly when $\beta>\gamma^{*}$. The function
$w$ is Lipschitz and $0\le w\le1$. So the pairs with $x\in B_{1}$ contribute at most
$\int_{B_{1}}\int\min\{1,\beta|h|\}^{p}|h|^{-N-sp}\,\mathrm{d}h\dx<\infty$. Now let $k\ge0$ and
$x\in A_{k}:=\{2^{k}<|x|\le2^{k+1}\}$. Then $w(x)\le2^{-k\beta}$. Also $|\nabla w|\le C2^{-k(\beta+1)}$ on
$\{|z|\ge2^{k-1}\}$. If $|x-y|<2^{k-1}$, the segment from $x$ to $y$ stays in $\{|z|\ge2^{k-1}\}$, so
$|w(x)-w(y)|\le C2^{-k(\beta+1)}|x-y|$. These pairs contribute at most
\[
  C2^{-kp(\beta+1)}\int_{A_{k}}\int_{\{|h|<2^{k-1}\}}|h|^{p-N-sp}\,\mathrm{d}h\dx\le C2^{k(N-sp-p\beta)} ,
\]
since $p>sp$. If $|x-y|\ge2^{k-1}$, we use $w(x)-w(y)\le2^{-k\beta}$. These pairs contribute at most
\[
  2^{-kp\beta}\int_{A_{k}}\int_{\{|h|\ge2^{k-1}\}}|h|^{-N-sp}\,\mathrm{d}h\dx\le C2^{k(N-sp-p\beta)} .
\]
This covers all pairs, including those between different annulus. Since
$N-sp-p\beta=-p(\beta-\gamma^{*})<0$, the sum over $k\ge0$ is finite.

(i) Here $w\le|x|^{-\alpha}$ and $w=0$ outside $B_{2R}$. So $w\in L^{p^{*}_{s}}$, because
$\alpha p^{*}_{s}<N$ exactly when $\alpha<\gamma^{*}$. Let $k\ge0$ and
$x\in A_{k}:=\{2^{-k-1}R<|x|\le2^{-k}R\}$, and put $r_{k}:=2^{-k}R$. Then $w(x)\le Cr_{k}^{-\alpha}$.
Also $|\nabla w|\le Cr_{k}^{-\alpha-1}$ on $\{|z|\ge r_{k}/4\}$, with $C$ depending on $\chi$. If
$|x-y|<r_{k}/4$, the segment from $x$ to $y$ stays in $\{|z|\ge r_{k}/4\}$. As in (ii), these pairs
contribute at most $Cr_{k}^{-p(\alpha+1)}r_{k}^{N}r_{k}^{p-sp}=Cr_{k}^{N-sp-p\alpha}$. The pairs with
$|x-y|\ge r_{k}/4$ contribute at most $Cr_{k}^{-p\alpha}r_{k}^{N}r_{k}^{-sp}=Cr_{k}^{N-sp-p\alpha}$.
Since $N-sp-p\alpha=p(\gamma^{*}-\alpha)>0$, the sum over $k\ge0$ is finite. Finally, the pairs with
$|x|>R$ have both points in $\{|z|>R\}$. There $w$ is Lipschitz, bounded, and zero outside $B_{2R}$. So
these pairs give a finite contribution.
\end{proof}

\subsubsection{\textbf{At the origin}}
Proposition~\ref{thm:comp} does not apply directly, since the ratio of the power to $U$ may be unbounded near $0$ and \eqref{H1} may fail. We therefore localize its proof away from the origin. Only \eqref{H1} changes.

\begin{lemma}\label{lem:loccomp}
Let $\rho>0$, $\nu\in\R$, and let $h\ge0$ be measurable. Let $u,v\in \mathcal{D}^{s,p}(\R^{N})$ with $u\ge0$
and $v>0$. Let $u$ be a weak subsolution and $v$ a weak supersolution of $\pLap w=\nu h(x)w^{p-1}$ in
$B_{\rho}\setminus\{0\}$, and let $u\le v$ almost everywhere in $\R^{N}\setminus B_{\rho}$. Let
$\sigma_{\delta}$, $\varphi_{1,\delta}$, $\varphi_{2,\delta}$ be as in Proposition~\ref{thm:comp}, and
assume \eqref{H2}. We fix a radially symmetric nondecreasing $\zeta\in \C^{\infty}(\R^{N})$ with $\zeta=0$ on
$B_{1}$ and $\zeta=1$ outside $B_{2}$, and put $\eta_{\eps}(x):=\zeta(\frac{x}{\eps})^{p}$ and
$K:=|x-y|^{-N-sp}$. In place of \eqref{H1}, we assume the following:
\begin{enumerate}
\item[{\rm (a)}] $\eta_{\eps}\varphi_{1,\delta},\ \eta_{\eps}\varphi_{2,\delta}\in
\mathcal{W}^{s,p}_{0}(B_{\rho}\setminus\{0\})$ for all $\eps,\delta\in(0,1]$,
\item[{\rm (b)}] $v$ is radially symmetric and nonincreasing,
\item[{\rm (c)}] as $\eps\downarrow0$,
\[
  \mathcal{R}_{\eps}:=\iint_{\R^{2N}}\big|\eta_{\eps}(x)-\eta_{\eps}(y)\big|\,
  \big|u(x)-u(y)\big|^{p-1}\big(u(x)+u(y)\big)K\dxy\ \longrightarrow\ 0,
\]
\end{enumerate}
and we assume \eqref{H2}. Then $u\le v$ almost everywhere in $\R^{N}$.
\end{lemma}

\begin{proof}
We follow the proof of Proposition~\ref{thm:comp} and indicate only the changes. Let
$\Omega^{+}:=\{u>v\}\subset B_{\rho}$, and let $A_{\delta},B_{\delta},g_{\delta}$ be as there. The
Picone claim is applied to $u+\delta$ and $v+\delta$. These are positive even where $u$ vanishes, so
$A_{\delta}-B_{\delta}\ge0$.

By (a) and density, $\eta_{\eps}\varphi_{1,\delta}$ and
$\eta_{\eps}\varphi_{2,\delta}$ are admissible for $u+\delta$ and $v+\delta$. Testing and subtracting,
\begin{equation}\label{eq:locAB}
  \frac12\iint_{\R^{2N}}\Big[\Jp(\Delta_{xy}u)\,\Delta_{xy}(\eta_{\eps}\varphi_{1,\delta})
  -\Jp(\Delta_{xy}v)\,\Delta_{xy}(\eta_{\eps}\varphi_{2,\delta})\Big]K\dxy
  \le\nu\int_{\Omega^{+}}h\,\eta_{\eps}\sigma_{\delta}^{+}g_{\delta}\dx .
\end{equation}
We put $m_{\eps}:=\frac12\big(\eta_{\eps}(x)+\eta_{\eps}(y)\big)$ and
$d_{\eps}:=\frac12\big(\eta_{\eps}(x)-\eta_{\eps}(y)\big)$. Then
$\Delta_{xy}(\eta_{\eps}\varphi)=m_{\eps}\Delta_{xy}\varphi+d_{\eps}\big(\varphi(x)+\varphi(y)\big)$, so
the integrand on the left of \eqref{eq:locAB} is
\[
  m_{\eps}\big(A_{\delta}-B_{\delta}\big)
  +d_{\eps}\Jp(\Delta_{xy}u)\big(\varphi_{1,\delta}(x)+\varphi_{1,\delta}(y)\big)
  -d_{\eps}\Jp(\Delta_{xy}v)\big(\varphi_{2,\delta}(x)+\varphi_{2,\delta}(y)\big).
\]

The last term has a sign. The function $\eta_{\eps}$ is radially nondecreasing and
$v$ is radially nonincreasing. If $|x|\le|y|$, then $d_{\eps}\le0\le\Delta_{xy}v$. If $|x|\ge|y|$, then
$d_{\eps}\ge0\ge\Delta_{xy}v$. So $d_{\eps}\Jp(\Delta_{xy}v)\le0$ for all $x,y$, and the last term above
is nonnegative. We drop it. No information on $v$ near the origin is needed.

For the middle term we use that, for $a\ge b\ge0$,
$(a+\delta)^{p}-(b+\delta)^{p}\le p(a+\delta)^{p-1}(a-b)$. Hence
$0\le\varphi_{1,\delta}\le p(u-v)^{+}\le pu$. The middle term above is therefore bounded in absolute value
by $\frac p2\big|\eta_{\eps}(x)-\eta_{\eps}(y)\big|\,|\Delta_{xy}u|^{p-1}\big(u(x)+u(y)\big)$, and
\eqref{eq:locAB} gives
\begin{equation}\label{eq:locAB2}
  \frac12\iint_{\R^{2N}}m_{\eps}\big(A_{\delta}-B_{\delta}\big)K\dxy
  \le\nu\int_{\Omega^{+}}h\,\eta_{\eps}\sigma_{\delta}^{+}g_{\delta}\dx+\frac p4\,\mathcal{R}_{\eps}.
\end{equation}

Now let $\eps\downarrow0$ with $\delta$ fixed. Off the origin, $\eta_{\eps}\to1$
and $m_{\eps}\to1$. Since $A_{\delta}-B_{\delta}\ge0$, Fatou's lemma applies on the left of
\eqref{eq:locAB2}. On the right, $0\le\eta_{\eps}\le1$ and $0\le g_{\delta}\le1$, and
$\sigma_{\delta}^{+}\le\sigma^{+}+p\delta(u+1)^{p-1}$ as in the proof of Proposition~\ref{thm:comp}. So
the integrand is dominated by $h\big[\sigma^{+}+p(u+1)^{p-1}\big]$, which is integrable on $\Omega^{+}$
by \eqref{H2}, and dominated convergence applies. Finally $\mathcal{R}_{\eps}\to0$ by (c). The limit is
exactly \eqref{eq:ABdelta}.

From \eqref{eq:ABdelta} on, the proof of Proposition~\ref{thm:comp} applies unchanged. There,
positivity of $u$ is only used in the strict inequality on $\Omega^{+}\times(\R^{N}\setminus\Omega^{+})$.
For $x\in\Omega^{+}$ we have $u(x)>v(x)>0$, and for $y\notin\Omega^{+}$ we have $0\le u(y)\le v(y)$. This
is all that argument needs.
\end{proof}

We now show that $\gamma_-$ cannot be improved at the origin.
\begin{proposition}\label{prop:optorigin}
Let $p \ge 2$, $0<\mu<\bar\mu$, and let $U\in\Dsp$ be a positive radially symmetric nonincreasing weak
solution of \eqref{eq:P}. For every $\alpha\in(0,\gamma_{-}(\mu))$ there are $c_{\alpha},r_{\alpha}>0$
such that
\begin{equation}\label{eq:optorigin}
  U(x)\ge c_{\alpha}|x|^{-\alpha},\quad\text{ for }0<|x|<r_{\alpha}.
\end{equation}
Consequently, no bound $U(x)=O(|x|^{-\beta})$ as $|x|\to0$ holds with $\beta<\gamma_{-}(\mu)$. In
particular, for $p\ge2$ the exponent $\gamma_{-}(\mu)$ in \eqref{eq:intro-final-bounds} cannot be
replaced by a smaller one.
\end{proposition}

\begin{proof}
We fix $\alpha\in(0,\gamma_{-})$ and write $\omega:=|x|^{-\alpha}$.
 By Proposition~\ref{thm:props}, $\Phi$ is strictly increasing
on $(0,\gamma^{*})$. Since $\alpha<\gamma_{-}<\gamma^{*}$, the number
$\delta_{\alpha}:=\mu-\Phi(\alpha)$ is positive. By \eqref{eq:homog}, due to Del Pezzo and
Quaas \cite{DPQ25},
\[
  \pLap\omega=\big(\mu-\delta_{\alpha}\big)|x|^{-sp}\omega^{p-1}\quad\text{in }\R^{N}\setminus\{0\}.
\]
 We fix $R>0$ and a radially symmetric nonincreasing
$\chi\in \C^{\infty}(\R^{N})$ with $\chi=1$ on $B_{R}$ and $\chi=0$ outside $B_{2R}$. We put
$\psi:=\chi\omega$. Then $\psi\le\omega$, and $\psi=\omega$ on $B_{R}$. Let $0<|x|<R$. For every $y$,
$\psi(x)-\psi(y)\ge\omega(x)-\omega(y)$, with equality unless $|y|>R$. Since $\Jp$ is increasing,
\begin{align*}
    & \pLap\psi(x)=\pLap\omega(x)+E(x), \text{ where } \\
    & E(x):=\int_{\{|y|>R\}}\Big[\Jp\big(\omega(x)-\psi(y)\big)-\Jp\big(\omega(x)-\omega(y)\big)\Big]
      \frac{\dy}{|x-y|^{N+sp}} \ge 0. 
\end{align*}
So the truncation increases the operator. This error has the wrong sign, and we show that it is small.
Let $\theta\in(0,\frac12]$ with $\theta^{\alpha}\le\frac12$, and let $0<|x|\le\theta R$. For $|y|>R$,
\[
  \omega(y)\le R^{-\alpha}\le\theta^{\alpha}\omega(x)\le\tfrac12\,\omega(x),
\]
so both arguments of $\Jp$ above lie in $\big[\frac12\omega(x),\omega(x)\big]$. On this interval
$\Jp'\le C_{p}\,\omega(x)^{p-2}$, with $C_{p}:=(p-1)\max\{1,2^{2-p}\}$. Also
$\omega(y)-\psi(y)\le|y|^{-\alpha}$ and $|x-y|\ge\frac12|y|$. Hence
\[
  0\le E(x)\le2^{N+sp}C_{p}\,\omega(x)^{p-2}\int_{\{|y|>R\}}|y|^{-\alpha-N-sp}\dy
  =C\,|x|^{-\alpha(p-2)}R^{-\alpha-sp},
\]
and so
\[
  \frac{E(x)}{|x|^{-sp}\omega(x)^{p-1}} \le  C\left(\frac{|x|}{R}\right)^{\alpha+sp} \le  C\theta^{\alpha+sp}.
\]
We fix $\theta=\theta_{\alpha}$ so small that the right-hand side is at most $\frac12\delta_{\alpha}$. Then
\begin{equation}\label{eq:psisub}
  \pLap\psi \le  \Big(\mu-\frac{\delta_{\alpha}}{2}\Big)|x|^{-sp}\psi^{p-1} \le  \mu|x|^{-sp}\psi^{p-1}
  \qquad\text{in }B_{\theta_{\alpha}R}\setminus\{0\}.
\end{equation}
This also holds in the weak sense. Let $0\le\varphi\in\cc(B_{\theta_{\alpha}R}\setminus\{0\})$. If both
points lie outside $B_{R}$, both integrands in \eqref{eq:weakdef} vanish. If exactly one point, say
$y$, lies outside $B_{R}$, the two integrands differ by
$\big[\Jp(\omega(x)-\psi(y))-\Jp(\omega(x)-\omega(y))\big]\varphi(x)|x-y|^{-N-sp}$. Counting both orders,
the left-hand sides of \eqref{eq:weakdef} for $\psi$ and for $\omega$ differ by $\int\varphi E\dx$. This
integral converges absolutely, since $|x-y|\ge\frac12|y|$ for $x\in\operatorname{supp}\varphi$ and
$|y|>R$.
Since \eqref{eq:homog} also holds in the weak sense, this gives \eqref{eq:psisub}
in the weak sense. By homogeneity, the same
holds for $\lambda\psi$, for every $\lambda>0$.

 We put $\rho:=\theta_{\alpha}R$. By
Theorem~\ref{thm:radialmain}, $U$ is continuous on $\R^{N}\setminus\{0\}$. Since $U$ is positive and
radially nonincreasing, $U(x)\ge U(2R)>0$ for $0<|x|\le2R$, where $U(2R)$ denotes the value on the
sphere of radius $2R$. We choose $\lambda>0$ with $\lambda(\rho/2)^{-\alpha}<U(2R)$, and put
$w:=\lambda\psi$. Then $w<U$ for $\rho/2\le|x|\le2R$, and $w=0<U$ for $|x|>2R$. Hence
\[
  \Omega^{+}:=\{w>U\}\subset B_{\rho/2},\qquad w\le U\ \text{ in }\R^{N}\setminus B_{\rho}.
\]
 We take $u=w$, $v=U$, $\nu=\mu$ and
$h=|x|^{-sp}$. By the definition of weak solution, $U\in L^{p-1}_{sp}(\R^{N})$. Also,
$w\in L^{p-1}_{sp}(\R^{N})$, since $w$ has compact support and $(p-1)\alpha<N$. By \eqref{eq:psisub}, $w$ is a weak
subsolution in $B_{\rho}\setminus\{0\}$. Since $U$ solves \eqref{eq:P} and $U^{p^{*}_{s}-1}\ge0$, it is a
weak supersolution there. By the choice of $\lambda$, $w\le U$ outside $B_{\rho}$.

We check \eqref{H2}. On $\Omega^{+}\subset B_{\rho/2}$ we have $\sigma^{+}\le w^{p}\le\lambda^{p}|x|^{-p\alpha}$
and $(w+1)^{p-1}\le C\big(|x|^{-\alpha(p-1)}+1\big)$. Both $sp+p\alpha<N$ and $sp+\alpha(p-1)<N$ hold,
because $\alpha<\gamma^{*}$. So \eqref{H2} holds.

We now check the remaining hypothesis of Lemma \ref{lem:loccomp}.

(a) Fix $\eps,\delta\in(0,1]$. The function $\eta_{\eps}\varphi_{i,\delta}$ vanishes outside the
compact set $\overline{B_{\rho/2}}\setminus B_{\eps}$. On $V:=B_{\rho}\setminus\overline{B_{\eps/2}}$,
both $w$ and $U$ lie in $W^{s,p}(V)$, $w$ is bounded, and $U\ge U(\rho)>0$. So the maps $F_{i,\delta}$
of Lemma~\ref{rem:Hcheck} are Lipschitz on the range of $(w,U)$ over $V$, and
$\varphi_{i,\delta}\in W^{s,p}(V)$. Multiplying by the smooth function $\eta_{\eps}$, and extending by
zero as in Lemma~\ref{rem:Hcheck}, gives
$\eta_{\eps}\varphi_{i,\delta}\in\mathcal{W}^{s,p}_{0}(B_{\rho}\setminus\{0\})$.

(b) $U$ is radially symmetric and nonincreasing by assumption.

(c) By H\"older inequality,
$\mathcal{R}_{\eps}\le2\,[w]_{s,p}^{p-1}Q_{\eps}^{1/p}$, where
\[
  Q_{\eps}:=\iint_{\R^{2N}}\big|\eta_{\eps}(x)-\eta_{\eps}(y)\big|^{p}\big(w(x)+w(y)\big)^{p}K\dxy .
\]
First, $[w]_{s,p}<\infty$, from Lemma~\ref{lem:barenergy}-(i). Let
$\eps\le\rho/8$. The integrand of $Q_{\eps}$ vanishes unless $\min\{|x|,|y|\}<2\eps$. Suppose
$|x|,|y|<4\eps$. There $w=\lambda|x|^{-\alpha}$, and the substitution $x=\eps\xi$, $y=\eps\eta$ shows
that this part of $Q_{\eps}$ equals $\eps^{N-sp-p\alpha}$ times
\[
  \lambda^{p}\iint_{B_{4}\times B_{4}}\big|\zeta(\xi)^{p}-\zeta(\eta)^{p}\big|^{p}
  \big(|\xi|^{-\alpha}+|\eta|^{-\alpha}\big)^{p}\,\frac{\mathrm{d}\xi\,\mathrm{d}\eta}{|\xi-\eta|^{N+sp}} .
\]
This integral is finite, because $|\zeta(\xi)^{p}-\zeta(\eta)^{p}|\le C|\xi-\eta|$ with $p>sp$, and
$p\alpha<N$. Suppose $|x|<2\eps$ and $|y|\ge4\eps$, or the same with $x,y$ exchanged. There
$|x-y|\ge\frac12|y|$ and $|\eta_{\eps}(x)-\eta_{\eps}(y)|\le1$. The part with $w(x)^{p}$ is at most
$$C\int_{B_{2\eps}}|x|^{-p\alpha}\dx\int_{\{|y|\ge4\eps\}}|y|^{-N-sp}\dy\le C\eps^{N-sp-p\alpha}.$$ The part
with $w(y)^{p}$ is at most $$C\eps^{N}\int_{\{|y|\ge4\eps\}}w(y)^{p}|y|^{-N-sp}\dy
\le C\eps^{N}\big(\eps^{-p\alpha-sp}+1\big)\le C\eps^{N-sp-p\alpha}.$$ 
Therefore,
\[
  Q_{\eps} \le  C\eps^{N-sp-p\alpha}=C\eps^{p(\gamma^{*}-\alpha)} \longrightarrow\ 0,
\]
since $\alpha<\gamma^{*}$. So $\mathcal{R}_{\eps}\to0$.

Lemma~\ref{lem:loccomp} gives $w\le U$ almost everywhere. Both functions
are continuous on $\R^{N}\setminus\{0\}$, so $w\le U$ at every $x\neq0$. For $0<|x|<\rho$ we have
$\chi(x)=1$, so $U(x)\ge\lambda|x|^{-\alpha}$. This is \eqref{eq:optorigin} with $c_{\alpha}=\lambda$ and
$r_{\alpha}=\rho$.
Finally, suppose $U(x)\le C|x|^{-\beta}$ near $0$ for some $\beta<\gamma_{-}$. We choose
$\alpha\in(\beta,\gamma_{-})$. Then $c_{\alpha}|x|^{-\alpha}\le C|x|^{-\beta}$, that is
$c_{\alpha}\le C|x|^{\alpha-\beta}$, for all small $|x|$. The right-hand side tends to $0$, which is a
contradiction.
\end{proof}

\subsubsection{\textbf{At infinity}}
At infinity the barrier is bounded. For a bounded subsolution the Picone test functions of
Proposition~\ref{thm:comp} are admissible without any cutoff. Only \eqref{H2} must be replaced. On an
unbounded set $\{u>v\}$ the term $(u+1)^{p-1}$ in \eqref{H2} need not be integrable against
$|x|^{-sp}$. The next lemma shows that the Hardy integrability of $u$ is enough.

\begin{lemma}\label{lem:extcomp}
Let $R>0$, $D:=\R^{N}\setminus\overline{B_{R}}$ and $\nu\in\R$. Let $u,v\in\Dsp$ with $v>0$ and $0<u\le M$ for a constant $M$. Let $u$ be a weak subsolution and $v$ a weak supersolution of $\pLap w=\nu|x|^{-sp}w^{p-1}$ in $D$, and let $\{u>v\}\subset\R^{N}\setminus B_{R'}$ for some $R'>R$. Then $u\le v$ almost everywhere in $\R^{N}$.
\end{lemma}

\begin{proof}
We follow the proof of Proposition~\ref{thm:comp} with $h=|x|^{-sp}$. However, neither \eqref{H1} nor \eqref{H2} is assumed here. Let $\Omega^{+}:=\{u>v\}$, and let $\sigma_{\delta}$, $\varphi_{i,\delta}$, $A_{\delta}$, $B_{\delta}$, $g_{\delta}$ be as given in Proposition~\ref{thm:comp}.

We first check admissibility. We fix $\delta\in(0,1]$. The maps $F_{i,\delta}$ of Lemma~\ref{rem:Hcheck} vanish when $b\ge a$. So their nonzero part lies in $\{0\le b<a\le M\}$, inside the compact square $[0,M]^{2}$. On this square the denominators are at least $\delta^{p-1}$. So $F_{1,\delta}$ and $F_{2,\delta}$ are Lipschitz there, with a constant $L_{\delta}$. Since they vanish when $b\ge a$, they are
Lipschitz on $[0,M]\times[0,\infty)$ with the same constant. Hence $[\varphi_{i,\delta}]_{s,p}\le L_{\delta}\big([u]_{s,p}+[v]_{s,p}\big)$. Also
$0\le\varphi_{i,\delta}\le p((M+1) \delta ^{-1})^{p-1}u$, so $\varphi_{i,\delta}\in
L^{p^{*}_{s}}(\R^{N})$. By the characterization of $\Dsp$ in \cite{BGV21} (see Notations),  $\varphi_{i,\delta}\in\Dsp$, and $\varphi_{i,\delta}=0$ on $B_{R'}$. We fix
$\theta\in \mathcal{C}^{\infty}(\R^{N})$ with $\theta=0$ on $B_{(R+R')/2}$ and $\theta=1$ outside $B_{R'}$.
Multiplication by $\theta$ is continuous on $\Dsp$. Indeed,
$$[\theta f]_{s,p}\le\|\theta\|_{\infty}[f]_{s,p}+\left(\int_{\R^N}|f(y)|^{p}\kappa(y)\dy\right)^{\frac1p}  \text{ with } \kappa(y):=\int_{\R^N}\frac{|\theta(x)-\theta(y)|^{p}}{|x-y|^{N+sp}}\dx.$$ The function $\kappa$ is bounded and
$O(|y|^{-N-sp})$ as $|y|\to\infty$, so $\kappa\in L^{\frac{N}{sp}}(\R^{N})$, and H\"older inequality gives
$$\int_{\R^N} |f|^{p}\kappa\le\|\kappa\|_{L^{\frac{N}{sp}}}\|f\|_{L^{p^{*}_{s}}}^{p}.$$ So if $\phi_{k}\in\cc(\R^{N})$
tend to $\varphi_{i,\delta}$ in $\Dsp$, then $\theta\phi_{k}\in\cc(D)$ tend to
$\theta\varphi_{i,\delta}=\varphi_{i,\delta}$. Both sides of the weak inequalities for $u$ and $v$ are continuous in $\Dsp$. For the left-hand sides we use H\"older inequality. For the right-hand sides we use H\"older inequality and \eqref{eq:FS}, extended to $\Dsp$ by density. Hence, by density, $\varphi_{1,\delta}$ and $\varphi_{2,\delta}$ are admissible, and \eqref{eq:ABdelta} holds.

Next we let $\delta\downarrow0$. The right-hand side of \eqref{eq:ABdelta} is
$$\nu\int_{\Omega^{+}}|x|^{-sp}\big[u^{p-1}\varphi_{1,\delta}-v^{p-1}\varphi_{2,\delta}\big]\dx.$$ Here
$z^{p-1}\varphi_{i,\delta}=\big(z/(z+\delta)\big)^{p-1}\sigma_{\delta}^{+}$, with $z=u$ for $i=1$ and $z=v$ for $i=2$. On $\Omega^{+}$, the mean value theorem gives
$\sigma_{\delta}\le p(u+\delta)^{p-1}(u-v)$. Since $t\mapsto t/(t+\delta)$ is increasing and $v<u$, we get
$\big(z/(z+\delta)\big)^{p-1}\le\big(u/(u+\delta)\big)^{p-1}$ on $\Omega^{+}$. Hence
\[
  0\le z^{p-1}\varphi_{i,\delta}\le pu^{p-1}(u-v)\le pu^{p} \text{ on }\Omega^{+}, \; \delta\in(0,1].
\]
This bound shows that the term $(u+1)^{p-1}$ in \eqref{H2} comes only from the estimate used in the proof of Proposition~\ref{thm:comp}.
By \eqref{eq:FS}, extended to $\Dsp$, the function $|x|^{-sp}u^{p}$ is integrable. As $\delta\downarrow 0$,
both $z^{p-1}\varphi_{i,\delta}$ tend to $\sigma^{+}$, because $u,v>0$. So the integrand tends to $0$, and
by dominated convergence the right-hand side of \eqref{eq:ABdelta} tends to $0$. From here the proof of Proposition~\ref{thm:comp} applies unchanged. It uses $|\R^{N}\setminus D|>0$, which holds. Therefore, $u\le v$ almost everywhere in $\R^{N}$.
\end{proof}

We now show that $\gamma_+$ cannot be improved at infinity.
\begin{proposition}\label{prop:optinf}
Let $0<\mu<\bar\mu$ and let $U\in\Dsp$ be a positive radially symmetric nonincreasing weak
solution of \eqref{eq:P}. For every $\beta\in(\gamma_{+}(\mu),\gamma_{c})$ there is $c_{\beta}>0$ such
that
\begin{equation}\label{eq:optinf}
  U(x)\ge c_{\beta}|x|^{-\beta}\qquad\text{for }|x|\ge1.
\end{equation}
Consequently, no bound $U(x)=O(|x|^{-\eta})$ as $|x|\to\infty$ holds with $\eta>\gamma_{+}(\mu)$. In
particular, for $p\ge2$ the exponent $\gamma_{+}(\mu)$ in \eqref{eq:intro-final-bounds} cannot be
replaced by a larger one.
\end{proposition}

\begin{proof}
We fix $\beta\in(\gamma_{+},\gamma_{c})$ and write $\omega:=|x|^{-\beta}$.
By Proposition~\ref{thm:props}, $\Phi$ is strictly decreasing
on $(\gamma^{*},\frac{N}{p-1})$. Since $\gamma^{*}<\gamma_{+}<\beta$, the number
$d_{\beta}:=\mu-\Phi(\beta)$ is positive. By \eqref{eq:homog}, due to Del Pezzo and Quaas
\cite{DPQ25}, $$\pLap\omega=(\mu-d_{\beta})|x|^{-sp}\omega^{p-1} \text{ in } \R^{N}\setminus\{0\}.$$
We put $\psi:=\min\{1,\omega\}$. Then $\psi\le\omega$, and
$\psi=\omega$ outside $B_{1}$. Let $|x|\ge2$. For every $y$,
$\psi(x)-\psi(y)\ge\omega(x)-\omega(y)$, with equality unless $|y|<1$. Since $\Jp$ is increasing,
\begin{align*}
    &\pLap\psi(x)=\pLap\omega(x)+E(x), \text{ where } \\
    & E(x):=\int_{B_{1}}\Big[\Jp\big(\omega(x)-1\big)-\Jp\big(\omega(x)-\omega(y)\big)\Big]
  \frac{\dy}{|x-y|^{N+sp}} \ge 0.
\end{align*}
For $y\in B_{1}$ we have $\omega(x)<1<\omega(y)$, so both arguments
of $\Jp$ are negative, and the bracket lies between $0$ and
$\big(\omega(y)-\omega(x)\big)^{p-1}\le|y|^{-\beta(p-1)}$. Since $|x-y|\ge\frac12|x|$,
\[
  0\le E(x)\le2^{N+sp}|x|^{-N-sp}\int_{B_{1}}|y|^{-\beta(p-1)}\dy=C|x|^{-N-sp},
\]
and the integral is finite because $\beta(p-1)<\gamma_{c}(p-1)=N-sp<N$. Hence
\[
  \frac{E(x)}{|x|^{-sp}\omega(x)^{p-1}} \le  C|x|^{\beta(p-1)-N} \longrightarrow 0,
  \quad \text{ as } |x|\to\infty.
\]
We fix $R_{\beta}\ge2$ such that the left-hand side is at most $\frac12d_{\beta}$ for $|x|\ge R_{\beta}$.
Then
\begin{equation}\label{eq:psisubinf}
  \pLap\psi \le  \Big(\mu-\frac{d_{\beta}}{2}\Big)|x|^{-sp}\psi^{p-1} \le  \mu|x|^{-sp}\psi^{p-1}
  \qquad\text{in }D:=\R^{N}\setminus\overline{B_{R_{\beta}}} .
\end{equation}
This also holds in the weak sense, by the argument used for \eqref{eq:psisub}. For
$0\le\varphi\in\cc(D)$, the integrands of \eqref{eq:weakdef} for $\psi$ and for $\omega$ differ only when
one point lies in $B_{1}$, where $\varphi$ vanishes, and the two left-hand sides differ by
$\int\varphi E\dx$. This gives \eqref{eq:psisubinf} in the
weak sense. By homogeneity, the same holds for $\lambda\psi$, for every $\lambda>0$.

 By Theorem~\ref{thm:radialmain}, $U$ is
continuous on $\R^{N}\setminus\{0\}$. Since $U$ is positive and radially nonincreasing,
$U(x)\ge U(2R_{\beta})>0$ for $0<|x|\le2R_{\beta}$. We put $\lambda:=\frac12U(2R_{\beta})$ and
$w:=\lambda\psi$. Since $\psi\le1$, we have $w<U$ on $\overline{B_{2R_{\beta}}}\setminus\{0\}$. Hence
$\{w>U\}\subset\R^{N}\setminus B_{2R_{\beta}}$.

We take $u=w$, $v=U$, $\nu=\mu$, $R=R_{\beta}$ and $R'=2R_{\beta}$. We have $0<w\le\lambda$. Also $w\in\Dsp$, by Lemma~\ref{lem:barenergy}-(ii). By \eqref{eq:psisubinf}, $w$ is a weak subsolution in $D$. Since $U$ solves \eqref{eq:P} and $U^{p^{*}_{s}-1}\ge0$, it is a weak
supersolution in $D$. Applying Lemma~\ref{lem:extcomp} gives $w\le U$ almost everywhere in $\R^N$. By continuity, $w\le U$ at every $x\neq0$. 

For $|x|\ge1$ we have $\psi(x)=|x|^{-\beta}$, so $U(x)\ge\lambda|x|^{-\beta}$.
This is \eqref{eq:optinf} with $c_{\beta}=\lambda$. Finally, suppose $U(x)\le C|x|^{-\eta}$ for large $|x|$, with $\eta>\gamma_{+}$. We choose
$\beta\in\big(\gamma_{+},\min\{\eta,\gamma_{c}\}\big)$. Then $c_{\beta}|x|^{-\beta}\le C|x|^{-\eta}$, that
is $c_{\beta}\le C|x|^{\beta-\eta}$, for all large $|x|$. The right-hand side tends to $0$, which is a
contradiction.
\end{proof}

As a consequence, we find the exact range of integrability.
\begin{corollary}\label{cor:nonint}
Let $0<\mu<\bar\mu$, and let $U$ be as in Proposition~\ref{prop:optinf}. Then
$U\notin L^{q}(\R^{N})$ if $q>\frac{N}{\gamma_{-}(\mu)}$ or if $0<q<\frac{N}{\gamma_{+}(\mu)}$.
\end{corollary}

\begin{proof}
Let $q>N/\gamma_{-}$. We choose $\alpha\in(N/q,\gamma_{-})$. By Proposition~\ref{prop:optorigin},
$U^{q}\ge c_{\alpha}^{q}|x|^{-\alpha q}$ on $B_{r_{\alpha}}$. Since $\alpha q>N$, this is not integrable on
$B_{r_{\alpha}}$. Now let $0<q<N/\gamma_{+}$. We choose $\beta\in\big(\gamma_{+},\min\{N/q,\gamma_{c}\}\big)$.
By Proposition~\ref{prop:optinf}, $U^{q}\ge c_{\beta}^{q}|x|^{-\beta q}$ for $|x|\ge1$. Since $\beta q<N$,
this is not integrable on $\{|x|\ge1\}$.
\end{proof}

\begin{remark}
 Propositions~\ref{prop:optorigin} and \ref{prop:optinf} and Corollary~\ref{cor:nonint} hold for every $p>1$. In Theorem~\ref{thm:D}, the restriction $p\ge2$ comes only from the upper bounds.
\end{remark}

\section{Proofs of the main theorems}\label{sec:proofs}
In this section we complete the proofs of the main results stated in Section~\ref{subsec:mainresults}. Theorem~\ref{thm:integ} is proved in Section~\ref{sec:radial}, where its integrability estimate is used directly in the subsequent analysis of radial solutions.

\begin{proof}[{\bf Proof of Theorem \ref{thm:A}:}]
We split the statement into five parts.

\noi (i)  \emph{The paired representation.} This is Proposition~\ref{prop:rep}. By Lemma~\ref{lem:pv}, the principal value in \eqref{eq:Phidef} is well defined.

\noi (ii) \emph{The sign of $\Phi'$.} Proposition~\ref{thm:props}-(ii) gives the formula
\eqref{eq:dPhi}. In this formula all factors are positive, except $r^{\gamma}-r^{\beta(\gamma)}$. For
$0<r<1$, this factor has the sign of $\beta(\gamma)-\gamma$. Hence $\Phi'(\gamma)$ has the sign
of $\beta(\gamma)-\gamma$.

\noi (iii) \emph{Monotonicity and the maximizer.} We have $\beta(\gamma)-\gamma=p(\gamma^{*}-\gamma)$. So,
by (ii), $\Phi$ is strictly increasing on $(0,\gamma^{*})$ and strictly decreasing on
$(\gamma^{*},\frac{N}{p-1})$. Therefore $\gamma^{*}$ is the unique maximizer. This is also
Proposition~\ref{thm:props}-(iii).

\noi (iv) \emph{The maximum.} By Proposition~\ref{thm:props}-(iv), $2\bar\mu=C_H$.

\noi (v) \emph{The two roots.} Let $0<\mu<\bar\mu$. Corollary~\ref{cor:roots} gives two roots
$\gamma_{-}(\mu)<\gamma^{*}<\gamma_{+}(\mu)$ in $(0,\gamma_{c})$. There is no other root. Indeed,
$\Phi<0$ on $(\gamma_{c},\frac{N}{p-1})$ by Proposition~\ref{thm:props}-(i). Also,
$\Phi$ is strictly monotone on each side of $\gamma^{*}$, by (iii). Corollary~\ref{cor:roots}
also gives the monotonicity of $\gamma_{\pm}$ in $\mu$, their limits, and the two profiles
$|x|^{-\gamma_{\pm}(\mu)}$ of \eqref{eq:intro-hardy}.
\end{proof}

\begin{proof}[{\bf Proof of Theorem \ref{thm:C}:}]
Let $u\in\Dsp$. By Corollary~\ref{cor:roots}, both $\gamma_{-}$ and $\gamma_{+}$ lie in
$(0,\gamma_{c})$, and $\Phi(\gamma_{\pm})=\mu$. So we apply Theorem~\ref{thm:gsrep} with
$\gamma=\gamma_{-}$. Its left-hand side is then exactly $\mathcal{D}_{\mu}(u)$. This gives
\[
  \mathcal{D}_{\mu}(u) \ge c_{p}\mathcal{E}_{\gamma_{-}}\big(|x|^{\gamma_{-}}u\big).
\]
In the same way, with $\gamma=\gamma_{+}$,
\[
  \mathcal{D}_{\mu}(u) \ge c_{p}\mathcal{E}_{\gamma_{+}}\big(|x|^{\gamma_{+}}u\big).
\]
We use the larger of the two right-hand sides. This is \eqref{eq:intro-two-groundstates}.
\end{proof}

\begin{proof}[{\bf Proof of Theorem \ref{thm:D}:}]
Let $U$ be as in Theorem~\ref{thm:D}. By Definition~\ref{def:optimal}, we need an upper bound and lower bounds at each end.

We start with the upper bounds. By Theorem~\ref{thm:cond}, $U$ satisfies \eqref{E0} and \eqref{E1}. By Theorem~\ref{thm:radialmain}, $U$ is continuous on $\R^{N}\setminus\{0\}$ and satisfies \eqref{M}. The same theorem then gives
$U(x)\le C|x|^{-\gamma_{-}(\mu)}$ near the origin and $U(x)\le C|x|^{-\gamma_{+}(\mu)}$ 
at infinity. This is \eqref{eq:intro-final-bounds}.

Next, we prove the lower bounds. Near the origin, Proposition~\ref{prop:optorigin} gives $U(x)\ge c_{\alpha}|x|^{-\alpha}$ for every $\alpha\in(0,\gamma_{-})$. At infinity, Proposition~\ref{prop:optinf} gives $U(x)\ge c_{\beta}|x|^{-\beta}$ for every $\beta\in(\gamma_{+},\gamma_{c})$. Therefore, in view of Definition~\ref{def:optimal}, 
$\gamma_{-}$ and $\gamma_{+}$ are the optimal exponents of $U$ at the origin and at infinity.

Finally we prove the integrability statement. We have shown that $U$ satisfies \eqref{E0} and \eqref{E1}. So Corollary~\ref{cor:sharpint} applies, and it gives $U\in L^{q}(\R^{N})$ for $N/\gamma_{+}<q<N/\gamma_{-}$. Furthermore, Corollary~\ref{cor:nonint} gives $U\notin L^{q}(\R^{N})$ for $0<q<N/\gamma_{+}$ and for $q>N/\gamma_{-}$.
\end{proof}

\appendix
\section{The principal value}\label{app:pv}
The principal value in \eqref{eq:Phidef} was taken over the sets $\{|\log r|>\eps\}$. The next lemma
shows that the usual exclusion of a ball around the singular point gives the same value.

\begin{lemma}\label{lem:pv}
Let $\gamma\in\big(0,\tfrac{N}{p-1}\big)$ and put $g(z):=\Jp\big(1-|z|^{-\gamma}\big)$ and
$K(z):=|e_{1}-z|^{-N-sp}$. Then both limits
\begin{equation*}
     \lim_{\eps\downarrow0}\int_{\{|\log|z||>\eps\}}gK\dz \; \text{ and } \;
  \lim_{\eps\downarrow0}\int_{\R^{N}\setminus B_{\eps}(e_{1})}gK\dz
\end{equation*}
exist and are equal. If $p>\frac{1}{1-s}$ the integral converges absolutely and no exclusion is needed.
\end{lemma}

\begin{proof}
We write $z=e_{1}+\zeta$ and $\zeta_{1}:=\zeta\cdot e_{1}$. By the estimates at the end of the proof of
Proposition~\ref{prop:rep}, the integrand is absolutely integrable outside any ball $B_{2\delta}(e_{1})$. So only the behaviour near $\zeta=0$ matters. We fix $\delta<\tfrac14$ and a
radial $\chi\in \cc(B_{2\delta})$ with $\chi\equiv1$ on $B_{\delta}$.

From $|e_{1}+\zeta|=1+\zeta_{1}+O(|\zeta|^{2})$ we get $1-|z|^{-\gamma}=\gamma\zeta_{1}+O(|\zeta|^{2})$.
We put $h(\zeta):=\gamma^{p-1}\Jp(\zeta_{1})\chi(\zeta)$. Since $h$ is odd in $\zeta_{1}$ and
$\chi(\zeta)|\zeta|^{-N-sp}$ is even in $\zeta_{1}$,
\begin{equation}\label{eq:oddvanish}
  \int_{E}h(\zeta) |\zeta|^{-N-sp} d\zeta=0
\end{equation}
for every measurable $E$ invariant under $\zeta_{1}\mapsto-\zeta_{1}$.

Next we show that $R(\zeta):=\big[g(e_{1}+\zeta)\chi(\zeta)-h(\zeta)\big]|\zeta|^{-N-sp}$ lies in
$L^{1}(B_{2\delta})$. We set $a:=1-|z|^{-\gamma}$ and $b:=\gamma\zeta_{1}$, so $|a-b|\le C|\zeta|^{2}$. If $p\ge2$, then $|a|+|b|\le C|\zeta|$ gives
$$|\Jp(a)-\Jp(b)|\le C(|a|+|b|)^{p-2}|a-b|\le C|\zeta|^{p},$$ and
$\int_{B_{2\delta}}|\zeta|^{p-N-sp}d\zeta<\infty$ because $p>sp$. 

Now let $1<p<2$. We choose $M>0$ large and split $B_{2\delta}$ into the regions $\{|\zeta_1|>M|\zeta|^2\}$ and $\{|\zeta_1|\le M|\zeta|^2\}$. On the first region we have $a=b+O(|\zeta|^2)$ and $b=\gamma\zeta_1$. So the choice of $M$ gives
$|a|\asymp |b|\asymp |\zeta_1|$, and the mean value theorem yields $|J_p(a)-J_p(b)| \le C|\zeta_1|^{p-2}|\zeta|^2.$ Writing $\zeta=\rho\omega$, we obtain
\begin{equation*}
\int_{B_{2\delta}}|\zeta_1|^{p-2}|\zeta|^{2-N-sp} \d\zeta =\left(\int_{\mathbb S^{N-1}}|\omega_1|^{p-2} \d\omega\right)\int_0^{2\delta}\rho^{p-1-sp} \d\rho<\infty.
\end{equation*}
On the second region, $|a|+|b|\le C|\zeta|^2.$ Since $J_p$ is H\"older continuous of order $p-1$, we have $|J_p(a)-J_p(b)|
 \le C|\zeta|^{2(p-1)}.$

On this second region we write $\zeta=(t,y)\in\mathbb R\times \mathbb R^{N-1}$, with $t=\zeta_1$ and $r=|y|$. Here $|t|\le M|\zeta|^2=M(t^2+r^2)$. Taking $\delta>0$ smaller if necessary, we obtain $|t|\le C r^2 \, \text{ and }\, |\zeta|\asymp r.$
 Therefore, 
 \begin{equation}
 \begin{aligned}
&\int_{\{|\zeta|<2\delta,\; |\zeta_1|\le M|\zeta|^2\}}
|J_p(a)-J_p(b)|\,|\zeta|^{-N-sp}\,d\zeta \\
&\quad\le C\int_0^{2\delta} r^{2(p-1)}r^{-N-sp}\,r^{N-2}\left(\int_{|t|\le Cr^2}dt\right)\dr  \\ 
&\quad\le C\int_0^{2\delta} r^{2p-2-sp}\dr<\infty,
\end{aligned}
 \end{equation}
since $2p-1-sp=(p-1)+(p-sp)>0$. Together with the estimate on the first region, this proves that $R\in L^1(B_{2\delta})$.

It remains to compare the two exclusions. 
Since $R\in L^1(B_{2\delta})$, the choice of exclusion does not affect its limit. So we only need to consider $h$.
We continue with the notation $\zeta=(t,y)$, $t=\zeta_1$ and $r=|y|$.
The ball exclusion is symmetric under $t \mapsto -t$. Hence the integral of $h$ over it is zero by \eqref{eq:oddvanish}. The same holds for the symmetric slab $B_\varepsilon:=\{|t|<\varepsilon\}.$ For the annular exclusion we set
\begin{gather*}
 A_\varepsilon:= \bigl\{|\log|e_1+\zeta||<\varepsilon\bigr\},   S_\varepsilon:=A_\varepsilon\triangle B_\varepsilon.
\end{gather*}
We only need to prove 
\begin{equation}\label{eq:Sdiff}
    \int_{S_\varepsilon} |h(\zeta)| |\zeta|^{-N-sp} d\zeta\longrightarrow0 \text{ as }\varepsilon\downarrow0.
\end{equation}
For fixed $y$, let $A_\varepsilon^y$ and $S_\varepsilon^y$ denote the corresponding sections in the $t$-variable. Since
$ |e_1+\zeta|^2=(1+t)^2+r^2$, for $\delta$ and $\varepsilon$ small we have
$ A_\varepsilon^y =\left(-1+\sqrt{e^{-2\varepsilon}-r^2},-1+\sqrt{e^{2\varepsilon}-r^2}\right).$
A Taylor expansion, uniform for $0\le r\le2\delta$, gives 
\begin{gather*}
\left|-1+\sqrt{e^{-2\varepsilon}-r^2}+\varepsilon\right|
 +\left|-1+\sqrt{e^{2\varepsilon}-r^2}-\varepsilon\right| \le C(\varepsilon^2+r^2).
\end{gather*}
Consequently,
\begin{equation}\label{eq:i}
 |S_\varepsilon^y| \le C\min\{\varepsilon,\varepsilon^2+r^2\},
\end{equation}
and every $(t,y)\in S_\varepsilon$ satisfies $|t|\le C(\varepsilon+r^2).$ 

We split the $r$-integration into three regions. If $0\le r\le\varepsilon$, then $S_\varepsilon^y$ lies in a $C\varepsilon^2$-neighbourhood of $t=\pm\varepsilon$. Hence
\begin{gather*}
 |t|\asymp\varepsilon,  |\zeta|\ge c\varepsilon,  |S_\varepsilon^y|\le C\varepsilon^2.
\end{gather*}
Since $|h|\le C|t|^{p-1}$, the integral over this region is bounded by 
\begin{equation}\label{eq:ii}
 C\varepsilon^{p-1}\varepsilon^{-N-sp}\varepsilon^2 \int_0^\varepsilon r^{N-2} dr \le C\varepsilon^{p-sp}=o_{\eps}(1).
\end{equation}
If $\varepsilon\le r\le\sqrt{\varepsilon}$, then by \eqref{eq:i} we have
$ |S_\varepsilon^y|\le Cr^2, |t|\le C\varepsilon, |\zeta|\ge r.$ So this part is bounded by
\begin{equation}\label{eq:iii}
C\varepsilon^{p-1}\int_\varepsilon^{\sqrt{\varepsilon}}r^{-sp} dr=o_{\eps}(1).
\end{equation}
Indeed, this quantity is $O\big(\varepsilon^{p-1+(1-sp)/2}\big)$ if $sp<1$, $O\big(\varepsilon^{p-1}|\log\varepsilon|\big)$ if $sp=1$, and $O\big(\varepsilon^{p-sp}\big)$ if $sp>1$.
Finally, if $\sqrt{\varepsilon}\le r\le2\delta$, then $ |S_\varepsilon^y|\le C\varepsilon, |t|\le Cr^2, |\zeta|\ge r.$ So the integral over this region is bounded by 
\begin{equation}\label{eq:iv}
C\varepsilon \int_{\sqrt{\varepsilon}}^{2\delta} r^{2p-sp-4} dr = o_{\eps}(1).
\end{equation}
Indeed, if $2p-sp>3$ this is $O(\varepsilon)$, if $2p-sp=3$ it is $O(\varepsilon|\log\varepsilon|)$, and if $2p-sp<3$, it is 
$O\!\left(\varepsilon^{(2p-sp-1)/2}\right)=o_{\eps}(1)$, since  $2p-sp-1=(p-1)+(p-sp)>0.$ Combining \eqref{eq:ii}--\eqref{eq:iv}, we get \eqref{eq:Sdiff}. Thus the annular exclusion and the usual ball exclusion give the same principal value.  Finally, $|g|\le C|\zeta|^{p-1}$ near $e_{1}$, so $|gK|\le C|\zeta|^{p-1-N-sp}$, which is integrable when $p-1-sp>0$. This completes the proof.
\end{proof}

\section{Further properties of power solutions}\label{app:power}
In this appendix we collect some results on power solutions of \eqref{eq:pureHardy} and \eqref{eq:P}. They are of independent interest, but we do not need them for the main theorems. They are a barrier estimate, the cylinder formulation, a touching lemma, compactness under scaling, and a blow-up dichotomy for the lower bound.

\subsection{A barrier estimate}
The next lemma is a simple barrier estimate near the origin. We do not use it in the main text.

\begin{lemma}\label{lem:improve}
Let $\gamma\in(\gamma_{-},\gamma^{*})$ and $A>0$. Then, for every
$\gamma'$ satisfying
\[
  \gamma_{-}<\gamma'\le\gamma
  \; \text{and } \;
  \gamma'\ge \gamma-\frac{P_s}{p-1}(\gamma^{*}-\gamma),
\]
there exists $A'>0$ such that $A'|x|^{-\gamma'}$ is a supersolution near the origin of the equation obtained from \eqref{eq:P}
by replacing the critical term with $A^{p_{s}^{*}-1}|x|^{-\gamma(p_{s}^{*}-1)}$.
For the limiting value of $\gamma'$, $\gamma^{*}-\gamma'=\left(1+\frac{P_s}{p-1}\right)(\gamma^{*}-\gamma).$
\end{lemma}

\begin{proof}
The required inequality is
$A'^{p-1}(\Phi(\gamma')-\mu)|x|^{-\gamma'(p-1)-sp}\ge A^{p_{s}^{*}-1}|x|^{-\gamma(p_{s}^{*}-1)}$
for small $|x|$, which holds when $\gamma'(p-1)+sp\ge\gamma(p_{s}^{*}-1)$ and
$\Phi(\gamma')>\mu$. By \eqref{eq:Xiid} the first is
$\gamma'\ge\gamma-\frac{P_s}{p-1}(\gamma^{*}-\gamma)$. For the second, note that
$\Phi(\gamma')>\mu$ holds precisely when $\gamma_{-}<\gamma'<\gamma_{+}$, by
Proposition~\ref{thm:props} and Corollary~\ref{cor:roots}. This is guaranteed by
$\gamma_{-}<\gamma'\le\gamma<\gamma^{*}<\gamma_{+}$. Such $\gamma'$ exist because $\gamma>\gamma_{-}$
and $\gamma-\frac{P_s}{p-1}(\gamma^{*}-\gamma)<\gamma$. The final identity follows by direct calculation.
\end{proof}

\subsection{Cylinder formulation}\label{sec:cylinder} 
We now rewrite the Hardy equation on the cylinder. Under this change of variables, dilations in $\R^N$ become translations in the cylinder variable. We fix $\mu>0$ and $\gamma\in\big(0,\tfrac{N}{p-1}\big)$.

\begin{lemma}\label{lem:cyl}
Let $x=e^{-t}\theta, t\in\R,\, \theta\in\Sph^{N-1}$ and define $v(t,\theta):=e^{-\gamma t}w(e^{-t}\theta).$
Equivalently, $w(x)=|x|^{-\gamma}v(t,\theta).$ Then $w$ solves \eqref{eq:pureHardy} in $\R^N\setminus\{0\}$ if and only if
\begin{equation}\label{eq:cyl}
A_\gamma v=\mu v^{p-1} \, \text{ on }\R\times\Sph^{N-1},
\end{equation}
where
\begin{equation*}
A_\gamma v(t,\theta)=\mathrm{P.V.}
\iint_{\Sph^{N-1}\times\R}
\Jp\bigl(v(t,\theta)-e^{\gamma\xi}v(t+\xi,\sigma)\bigr)
\frac{e^{-N\xi}\,d\xi\,d\sigma}
{\bigl(1-2e^{-\xi}\theta\cdot\sigma+e^{-2\xi}\bigr)^{\frac{N+sp}{2}}}.
\end{equation*}
The operator $A_\gamma$ is invariant under translations in $t$ and rotations on the sphere. It also satisfies
$A_\gamma(\lambda v)=\lambda^{p-1}A_\gamma v \text{ for }\lambda>0$ and $A_\gamma \mathbf{1}=\Phi(\gamma)$.
\end{lemma}

\begin{proof}
We write $y=e^{-\tau}\sigma,\,\xi:=\tau-t$.
Then $dy=e^{-N\tau}\,d\tau\,d\sigma$ and $|x-y|^2=e^{-2t}\bigl(1-2e^{-\xi}\theta\cdot\sigma+e^{-2\xi}\bigr).$
Also, $w(x)-w(y)=e^{\gamma t}\bigl(v(t,\theta)-e^{\gamma\xi}v(t+\xi,\sigma)\bigr).$
Substituting these identities into $\pLap w$ gives $\pLap w(x)=e^{\gamma(p-1)t+spt}A_\gamma v(t,\theta)$. On the other hand, the right-hand side of \eqref{eq:pureHardy} is $$\mu e^{\gamma(p-1)t+spt}v(t,\theta)^{p-1}.$$ Thus \eqref{eq:pureHardy} is equivalent to \eqref{eq:cyl}. 
Finally, substituting $\rho=e^{-\xi}$ in $A_\gamma\mathbf{1}$ and comparing with
\eqref{eq:Phidef} gives $A_\gamma\mathbf{1}=\Phi(\gamma)$.
\end{proof}

In particular, the power function $A|x|^{-\gamma}$ corresponds to the constant $v\equiv A$.

\subsection{Touching lemma}
We first show what happens when a power function touches a solution of \eqref{eq:pureHardy}.

\begin{lemma}\label{lem:touch}
Let $\gamma\in\big(0,\tfrac{N}{p-1}\big)$ and let $w>0$ be a continuous weak solution of
\eqref{eq:pureHardy} in $\R^{N}\setminus\{0\}$ with $w\in L^{p-1}_{sp}(\R^{N})$. Suppose $A>0$ satisfies
\begin{equation*}
      w(x) \le A|x|^{-\gamma}\;\text{on }\R^{N}\setminus\{0\},
  \quad w(x_{0})=A|x_{0}|^{-\gamma}\ \text{ for some }x_{0}\neq0 .
\end{equation*}
Then $\Phi(\gamma)\le\mu$ and if $\Phi(\gamma)=\mu$ then $w\equiv A|x|^{-\gamma}$. The
same holds with both inequalities reversed and $\Phi(\gamma)\ge\mu$.
\end{lemma}

\begin{proof}
We set $v(x):=A|x|^{-\gamma}$, so $\pLap v=\Phi(\gamma)|x|^{-sp}v^{p-1}$ by \eqref{eq:homog}. By the
equivalence of weak and viscosity solutions \cite{KKL19}, $w$ is a viscosity subsolution of
\eqref{eq:pureHardy}. We fix $0<r<|x_{0}|/2$ and set $V_{r}:=v$ in $B_{r}(x_{0})$ and $V_{r}:=w$ in $\R^N \setminus B_{r}(x_{0})$.
Then $V_{r}\in L^{p-1}_{sp}(\R^{N})$, $V_{r}$ is smooth in $B_{r}(x_{0})$ with
$\nabla V_{r}(x_{0})=\nabla v(x_{0})\neq0$, and $V_{r}$ touches $w$ from above at $x_{0}$. The principal
value $\pLap V_{r}(x_{0})$ exists. For $p\ge2$, second-order cancellation ensures its existence, and for $p \in (1,2)$ we use $\int_{\Sph^{N-1}}|e \cdot \omega|^{p-2}\d\omega<\infty$, valid since $p>1$. Hence
\begin{equation}\label{eq:visc}
  \pLap V_{r}(x_{0})\le\mu |x_{0}|^{-sp}w(x_{0})^{p-1}.
\end{equation}
On the other hand $V_{r}\le v$ everywhere with $V_{r}(x_{0})=v(x_{0})$, so
$V_{r}(x_{0})-V_{r}(y)\ge v(x_{0})-v(y)$ for every $y$, and since $\Jp$ is increasing,
\begin{equation}\label{eq:compare}
  \pLap V_{r}(x_{0})\ge\pLap v(x_{0})
   = \Phi(\gamma) |x_{0}|^{-sp}w(x_{0})^{p-1}.
\end{equation}
Comparing \eqref{eq:visc} and \eqref{eq:compare} give $\Phi(\gamma)\le\mu$. If
$\Phi(\gamma)=\mu$, then both are equalities, so
\[
  \int_{\R^{N}\setminus B_{r}(x_{0})}
  \Big[\Jp\big(v(x_{0})-w(y)\big)-\Jp\big(v(x_{0})-v(y)\big)\Big]\frac{\dy}{|x_{0}-y|^{N+sp}}=0 .
\]
The integrand is nonnegative because $w\le v$ and $\Jp$ is increasing, so it vanishes almost
everywhere and $w=v$ in $\R^N \setminus B_{r}(x_{0})$. Letting $r\downarrow0$ gives $w\equiv v$ a.e. in $\R^N$.
\end{proof}

Applying Lemma~\ref{lem:touch} at both extrema gives the following result.

\begin{corollary}\label{cor:attained}
Let $w$ satisfy $c|x|^{-\gamma}\le w\le C|x|^{-\gamma}$ on $\R^{N}\setminus\{0\}$ and let $v$ be as in
Lemma~\ref{lem:cyl}. If both $\sup v$ and $\inf v$ over $\R\times\Sph^{N-1}$ are attained, then
$\Phi(\gamma)=\mu$, so $\gamma\in\{\gamma_{-}(\mu),\gamma_{+}(\mu)\}$, and
$w\equiv C_{0}|x|^{-\gamma}$ for some $C_{0}>0$. If $\Phi(\gamma)=\mu$ is known, it is enough that either $\sup v$ or $\inf v$ is attained.
\end{corollary}

\begin{proof}
Applying Lemma~\ref{lem:touch} with $A=\sup v$ to get $\Phi(\gamma)\le\mu$, and with $A=\inf v$ and reversed inequalities, we get $\Phi(\gamma)\ge\mu$. Then using the second part of Lemma~\ref{lem:touch} the conclusion holds. 
\end{proof}
\subsection{Scaling and compactness}
We record the scaling of \eqref{eq:pureHardy} and prove a compactness result for rescaled solutions.

The next lemma shows how the Hardy equation behaves under scaling.
\begin{lemma}\label{lem:rescale}
For $\lambda>0$ put $w_{\lambda}(x):=\lambda^{\gamma}w(\lambda x)$. If $w$ solves \eqref{eq:pureHardy}
in $\R^{N}\setminus\{0\}$ then so does $w_{\lambda}$, and if $w$ satisfies the bound $c|x|^{-\gamma}\le w\le C|x|^{-\gamma}$, then  $w_{\lambda}$ satisfies the same bound
$c|x|^{-\gamma}\le w_{\lambda}\le C|x|^{-\gamma}$. On the cylinder, this scaling corresponds to a translation in the $t$-variable. More precisely, if $\lambda=e^{-t_\lambda}$, then $w_\lambda$ corresponds to $v(t+t_\lambda,\theta)$.
\end{lemma}

\begin{proof}
Since $\pLap\big(w(\lambda \cdot)\big)(x)=\lambda^{sp}\big(\pLap w\big)(\lambda x)$ and
$\pLap(\kappa u)=\kappa^{p-1}\pLap u$ for $\kappa>0$,
\begin{equation*}
     \pLap w_{\lambda}(x)=\lambda^{\gamma(p-1)+sp}\big(\pLap w\big)(\lambda x)
  =\lambda^{\gamma(p-1)+sp}\mu|\lambda x|^{-sp}w(\lambda x)^{p-1}=\mu|x|^{-sp}w_{\lambda}(x)^{p-1}.
\end{equation*}
The bound is preserved because $\lambda^{\gamma}|\lambda x|^{-\gamma}=|x|^{-\gamma}$.
\end{proof}

The following lemma shows that the family produced by Lemma~\ref{lem:rescale} is precompact.

\begin{lemma}\label{lem:compact}
Let $\{w_{j}\}$ be positive continuous weak solutions of \eqref{eq:pureHardy} in $\R^{N}\setminus\{0\}$, such that for all $j$, $c|x|^{-\gamma}\le w_{j}\le C|x|^{-\gamma}$ with some $c,C$. Then a subsequence of $\{w_{j}\}$ converges locally uniformly on $\R^{N}\setminus\{0\}$ to a positive continuous weak solution
$w_{\infty}$ of \eqref{eq:pureHardy} with the same bound. Equivalently, if $v$ weakly solves \eqref{eq:cyl} with $c\le v\le C$ and $\{ t_{j}\} \subset\R$, then a subsequence of $v(\cdot+t_{j},\cdot)$ converges locally uniformly to a solution $\tilde v$ of \eqref{eq:cyl} with $c\le\tilde v\le C$.
\end{lemma}

\begin{proof}
We observe that Lemma~\ref{lem:rescale} identifies the two formulations. For $\rho>1$, denote
$A_{\rho}:=B_{\rho}\setminus\overline{B_{1/\rho}}$, and define $f_{j}:=\mu|x|^{-sp}w_{j}^{p-1}$. Constants below may depend on $N,p,s,\mu,\gamma,c,C,\rho$, but not on $j$.
From the hypothesis, 
\begin{align}\label{uni-bound-1}
    \|w_{j}\|_{L^{\infty}(A_{4\rho})}\le C(4\rho)^{\gamma} \text{ and } \|f_{j}\|_{L^{\infty}(A_{4\rho})}\le\mu(4\rho)^{sp}\big(C(4\rho)^{\gamma}\big)^{p-1}.
\end{align}
For the tails, if $|x_{0}|\in[\tfrac{1}{2\rho},2\rho]$ and $0<r<\tfrac{1}{8\rho}$, then
\begin{equation}\label{uni-bound-2}
    \int_{\R^{N}\setminus B_{r}(x_{0})}\frac{w_{j}(y)^{p-1}}{|x_{0}-y|^{N+sp}}\dy
  \le C^{p-1}\int_{\R^{N}\setminus B_{r}(x_{0})}\frac{|y|^{-\gamma(p-1)}}{|x_{0}-y|^{N+sp}}\dy
  =:T(\rho,r) < \infty, 
\end{equation}
where finiteness of $T(\rho,r)$ comes since $\gamma(p-1)<N$ near the origin and $\gamma(p-1)+sp>0$ at infinity.
We observe that for every $y \in B_{2r}(x_{0})$, $|y| \le |x_0| + |y-x_0| \le 2\rho + 1/4\rho < 4\rho$ and $|y| \ge |x_0| -|x_0-y| \ge \frac12\rho-1/4\rho = 1/4\rho$, which implies $B_{2r}(x_{0})\subset A_{4\rho}$.
Let $\eta\in \cc(B_{2r}(x_{0}))$ with
$0\le\eta\le1$, $\eta\equiv1$ on $B_{r}(x_{0})$, $|\nabla\eta|\le2/r$. The Caccioppoli inequality for
weak solutions of $\pLap u=f$ with $f\in L^{\infty}$ (see \cite[Theorem~1.4]{DKP16} and
\cite[S3]{Coz17}) and \eqref{uni-bound-1}, \eqref{uni-bound-2} give
$$\iint_{B_{2r}\times B_{2r}} \frac{|w_{j}(x)\eta(x)-w_{j}(y)\eta(y)|^{p}}{|x-y|^{N+sp}}\dxy\le\Lambda_{1},$$
where $\Lambda_{1}$ does not depend on $j$. Covering $\overline{A_{2\rho}}$ by
finitely many such balls, we obtain 
\begin{equation}\label{eq:unifWsp}
  \sup_{j}\Big([w_{j}]^{p}_{W^{s,p}(A_{2\rho})}+\|w_{j}\|^{p}_{L^{p}(A_{2\rho})}\Big)=:\Lambda_{2}<\infty .
\end{equation}
By the interior H\"older estimates of \cite{DKP16}, there exists $\beta\in(0,1)$ such that $\|w_{j}\|_{C^{0,\beta}(\overline{A_{\rho}})}\le\Lambda_{3}$.
By the Arzel\`{a}--Ascoli theorem and a diagonal argument, a subsequence of $\{w_j\}$ converges
uniformly on compact subsets of $\R^{N}\setminus\{0\}$ to a continuous function $w_{\infty}$ with the same
two-sided bound.
We put $Q_{\rho}:=A_{2\rho}\times A_{2\rho}$, and
\[
  D_{j}(x,y):=\frac{w_{j}(x)-w_{j}(y)}{|x-y|^{\frac{N+sp}{p}}}, \quad
  \Theta_{j}(x,y):=|D_{j}(x,y)|^{p-2}D_{j}(x,y)=\frac{\Jp\big(w_{j}(x)-w_{j}(y)\big)}{|x-y|^{\frac{N+sp}{p'}}} .
\]
By \eqref{eq:unifWsp}, $\|D_{j}\|^{p}_{L^{p}(Q_{\rho})}\le\Lambda_{2}$, and hence
$\|\Theta_{j}\|^{p'}_{L^{p'}(Q_{\rho})}\le\Lambda_{2}$. Since $w_j \ra w_{\infty}$ uniformly on
$\overline{A_{2\rho}}$, we get $D_{j}\to D_{\infty}$ and hence $\Theta_{j}\to\Theta_{\infty}$ almost
everywhere on $Q_{\rho}$. Since $|Q_{\rho}|<\infty$, a sequence bounded in $L^{p'}(Q_{\rho})$
and convergent almost everywhere converges weakly in $L^{p'}(Q_{\rho})$. Hence
\begin{equation}\label{eq:weakTheta}
  \Theta_{j}\rightharpoonup \Theta_{\infty} \text{ in } L^{p'}(Q_{\rho}).
\end{equation}
So weak compactness in $L^{p'}$ is enough, and we need no pointwise bound.
We fix $\phi\in \cc(A_{\rho})$. Since the integrand
$\Jp(w_{j}(x)-w_{j}(y))(\phi(x)-\phi(y))|x-y|^{-N-sp}$ is symmetric in $(x,y)$ and $\phi$ vanishes
outside $A_{\rho}$, the weak formulation \eqref{eq:weakdef} splits as
\begin{equation*}
\frac12\underbrace{\iint_{Q_{\rho}}\Theta_{j} d_{\phi}\dxy}_{\mathrm{I}_{j}}
+\underbrace{\iint_{A_{\rho} \times \R^{N}\setminus A_{2\rho}}
\frac{\Jp\big(w_{j}(x)-w_{j}(y)\big) \phi(x)}{|x-y|^{N+sp}}\dy\dx}_{\mathrm{II}_{j}}
= \int_{A_{\rho}}f_{j} \phi\dx,
\end{equation*}
with $d_{\phi}(x,y):=\big(\phi(x)-\phi(y)\big)|x-y|^{-(N+sp)/p}$. As $\phi$ is Lipschitz and $sp<p$,
$d_{\phi}\in L^{p}(Q_{\rho})$, so \eqref{eq:weakTheta} gives
$$\mathrm{I}_{j}\underset{j \ra \infty}{\longrightarrow} \iint_{Q_{\rho}}\Theta_{\infty}d_{\phi} \dxy.$$ 
In $\mathrm{II}_{j}$ there is no singularity since $x\in\operatorname{supp}\phi$ and $y\notin A_{2\rho}$ imply $|x-y|\ge d(\rho)>0$. Moreover
$\big|\Jp(w_{j}(x)-w_{j}(y))\phi(x)\big|\le\Lambda_{4}\big(1+|y|^{-\gamma(p-1)}\big)$. This bound is independent of $j$ and 
is integrable against $|x-y|^{-N-sp}$. Hence dominated convergence applies and we get 
\begin{align*}
\mathrm{II}_{j} \underset{j \ra \infty}{\longrightarrow} \iint_{A_{\rho} \times \R^{N}\setminus A_{2\rho}}
\frac{\Jp\big(w_{\infty}(x)-w_{\infty}(y)\big)}{|x-y|^{N+sp}}\phi(x)\dy\dx.
\end{align*}
Finally, $f_{j}\to f_{\infty}$ uniformly on $\operatorname{supp}\phi$.
By weak lower semicontinuity in \eqref{eq:unifWsp}, $w_{\infty}\in W^{s,p}(A_{2\rho})$ for
every $\rho$, and $w_{\infty}\in L^{p-1}_{sp}(\R^{N})$ by the two-sided bound. Passing to the limit above
gives \eqref{eq:weakdef} for $w_{\infty}$ against $\phi$, where $\rho$ and $\phi$ are arbitrary.
\end{proof}

\subsection{A blow-up dichotomy} Comparison gives no lower bound. Indeed, a power subsolution $A|x|^{-\gamma}$ of \eqref{eq:pureHardy} needs $\gamma\notin[\gamma_-,\gamma_+]$. If $\gamma>\gamma_+$, it is too large at the origin. If $\gamma<\gamma_-$, it decays too slowly at infinity. Also, the minimum of two subsolutions is not a subsolution. The next result describes what happens if a lower bound fails. For radial solutions, lower bounds are proved in Section~\ref{sec:optimal}.
\begin{proposition}\label{prop:dichotomy}
Let $u$ be a positive weak solution of \eqref{eq:P}, continuous on $\R^{N}\setminus\{0\}$ and
satisfying \eqref{M}, \eqref{E0} and \eqref{E1}, and let $r_{0},R_{0},C$ be as in
Theorem~\ref{thm:sharpupper}. For $0<r<r_{0}$ set $u_{r}(y):=r^{\gamma_{-}}u(ry)$, so
that
\begin{equation}\label{eq:blowup}
  \pLap u_{r}=\mu|y|^{-sp}u_{r}^{p-1}+r^{\kappa_{0}}u_{r}^{p_{s}^{*}-1},
  \quad \kappa_{0}=P_s(\gamma^{*}-\gamma_{-})>0,
\end{equation}
and $u_{r}\le C'|y|^{-\gamma_{-}}$ on $\R^{N}\setminus\{0\}$ uniformly in $r$. Then one of the
following holds.
\begin{enumerate}
\item[\rm(a)] $\displaystyle\liminf_{x\to0}|x|^{\gamma_{-}}u(x)>0$, so that $u\asymp|x|^{-\gamma_{-}}$
near the origin.
\item[\rm(b)] there is $r_{j}\downarrow0$ with $u_{r_{j}}\to0$ locally uniformly on
$\R^{N}\setminus\{0\}$.
\end{enumerate}
\end{proposition}

\begin{proof}
Equation \eqref{eq:blowup} follows from the scaling in Lemma~\ref{lem:rescale}. By \eqref{eq:Xiid} the critical term is multiplied by $r^{\gamma_{-}(p-1)+sp-\gamma_{-}(p_{s}^{*}-1)}=r^{P_s(\gamma^{*}-\gamma_{-})}$.
For the uniform bound, split according to the size of $r|y|$. If $r|y|<r_{0}$, then
Theorem~\ref{thm:sharpupper} gives $u(ry)\le C(r|y|)^{-\gamma_{-}}$, hence
$u_{r}(y)\le C|y|^{-\gamma_{-}}$. If $r_{0}\le r|y|\le R_{0}$, then $u$ is bounded on that annulus by
continuity while $(r|y|)^{-\gamma_{-}}$ is bounded below there, so the same bound holds after enlarging
the constant. If $r|y|>R_{0}$, then Theorem~\ref{thm:sharpupper} gives
$u(ry)\le C(r|y|)^{-\gamma_{+}}\le CR_{0}^{\gamma_{-}-\gamma_{+}}(r|y|)^{-\gamma_{-}}$, because
$\gamma_{-}<\gamma_{+}$. In all three cases $u_{r}(y)\le C'|y|^{-\gamma_{-}}$ with $C'$ independent of
$r$, by using \eqref{E1}.

Suppose (a) fails. Then there exists $r_{j}\downarrow0$ and $|y_{j}|=1$ such that $u_{r_{j}}(y_{j})\to0$. Since
$u_{r_{j}}\ge0$ is a weak supersolution of $\pLap w\ge0$ on $\R^{N}\setminus\{0\}$, the weak Harnack
inequality of \cite[Theorem ~1.2]{DKP14} applies on balls $B_{2\sigma}(y_{j})$ that stay away from the origin.

Hence, for some admissible $t>0$, 
\[
  \Big(\fint_{B_{\sigma}(y_{j})}u_{r_{j}}^{t}\Big)^{1/t}\le C\inf_{B_{2\sigma}(y_{j})}u_{r_{j}}
  \le Cu_{r_{j}}(y_{j})\to0 .
\]
The constant $C$ is independent of $j$ since the coefficients in \eqref{eq:blowup} are uniformly bounded on fixed annuli. The
uniform H\"older bounds of \cite{DKP16,BLS18,GL24} then give uniform convergence to zero in a neighbourhood of $y_{j}$. 
By covering the annulus $\{\tfrac12<|y|<2\}$ with finitely many overlapping balls, we obtain
$u_{r_{j}}\to0$ uniformly on $\{\tfrac12<|y|<2\}$.

The compactness argument of Lemma~\ref{lem:compact} also applies to \eqref{eq:blowup}. Indeed, the extra term $r^{\kappa_{0}}u_{r}^{p_{s}^{*}-1}$ tends to zero uniformly on compact annuli. Therefore, any
limit solves the pure Hardy equation. Hence, after passing to a subsequence $u_{r_j} \to u_{\infty}$ locally uniformly in $\mathbb{R}^N \setminus \{0\}$ where $u_{\infty}$ is a weak solution of \eqref{eq:pureHardy}. Moreover, $u_{\infty}\le C'|y|^{-\gamma_{-}}$ and $u_{\infty}\equiv0$ on $\{\tfrac12<|y|<2\}$. We can write \eqref{eq:pureHardy} as $\pLap u_{\infty}+c u_{\infty}^{p-1}=0$ with
$c=-\mu|y|^{-sp}\in L^{\infty}_{\mathrm{loc}}(\R^{N}\setminus\{0\})$. Applying the strong minimum principle of
\cite{DPQ17} gives $u_{\infty}\equiv0$ in $\R^{N}\setminus\{0\}$, which is (b).
\end{proof}

\section{Pointwise coercivity estimates}\label{app:coerc}
Here we prove the pointwise inequalities used in Sections~\ref{sec:verify} and \ref{sec:scheme}. We keep the notation of Section~\ref{sec:route}. The functions $G_M$, $H_M$ are defined in \eqref{eq:GH}.

\subsubsection{\textbf{The pure-power case}} The inequality \eqref{eq:routeC} holds for pure powers. This is the
case $G(X)=X^{m}$, $H(X)=X^{\frac{r}{p}}$ of \eqref{eq:routebracket}, for which $X^{p-1}G(X)=X^{r}$.

\begin{lemma}\label{lem:purepower}
Let $p>2$, $r\ge p$ and set $m:=r-p+1$. There is $c>0$ such that
\begin{equation}\label{eq:purepower}
  \Jp(aX-bY)\big(aX^{m}-bY^{m}\big)-\Jp(a-b)\big(aX^{r}-bY^{r}\big)
  \ge c(ab)^{\frac{p}{2}}\left|X^{\frac{r}{p}}-Y^{\frac{r}{p}}\right|^{p},
\end{equation}
for all $a,b>0$ and $X,Y\ge0$.
\end{lemma}

\begin{proof}
Both sides of \eqref{eq:purepower} are unchanged under the exchange
$(a,X)\leftrightarrow(b,Y)$, so we may assume $a\ge b$. Both sides are continuous in $(X,Y)$ on
$[0,\infty)^{2}$, so we only need to prove \eqref{eq:purepower} for $X,Y>0$ and then let $X\downarrow0$
or $Y\downarrow0$. We assume therefore $X,Y>0$ and put
\[
  t:=\frac ba\in(0,1],\quad A:=\frac XY\in[0,\infty),\quad \alpha:=\frac rp\ge1 .
\]
Since $\Jp(\lambda u)=\lambda^{p-1}\Jp(u)$ for $\lambda>0$ and $p-1+m=r$, the left-hand side of
\eqref{eq:purepower} equals $a^{p}Y^{r}\Theta(t,A)$ and the right-hand side equals
$ca^{p}Y^{r}t^{\frac{p}{2}}|A^{\alpha}-1|^{p}$, where
\begin{equation}\label{eq:Theta}
  \Theta(t,A):=\Jp(A-t)\big(A^{m}-t\big)-(1-t)^{p-1}\big(A^{r}-t\big).
\end{equation}
So \eqref{eq:purepower} is equivalent to
\begin{equation}\label{eq:Qinf}
  \inf\Big\{Q(t,A): 0<t\le1, A\ge0, A\neq1 \Big\}>0,
  \quad Q(t,A):=\frac{\Theta(t,A)}{t^{\frac{p}{2}}|A^{\alpha}-1|^{p}}.
\end{equation}
We use two identities repeatedly. For every $t$ we have
\begin{equation}\label{eq:doublevanish}
  \Theta(t,1)=0,\quad \partial_{A}\Theta(t,1)=0,\quad
  \partial_{A}^{2}\Theta(t,1)=t(1-t)^{p-2}(p-1)(2r-p)>0\ \ (0<t<1),
\end{equation}
so $\Theta(t,\cdot)$ vanishes to second order at $A=1$, uniformly in $t$.
 
\medskip
\noi \textbf{Step 1:} We show $\Theta(t,A)>0$ for $A\neq1$. Differentiating \eqref{eq:Theta} and using
$r-1-(p-2)=m$,
\[
  \partial_{A}\Theta(t,A)=A^{r-1}\,\Omega(t,1/A),\quad
  \Omega(t,B):=|1-tB|^{p-2}\,g(t,B)-r(1-t)^{p-1},
\]
with $g(t,B):=r-t(p-1)B^{m}-tmB$. We fix $t$, and note that $\partial_B g(t,B) = -t[m(p-1)B^{m-1}+m]<0$ for $B>0$, and hence $g(t,\cdot)$ is strictly decreasing. Further,
$g(t,1)=r(1-t)\ge0$, while the map $B\mapsto|1-tB|^{p-2}$ is nonincreasing (since $p > 2$) and nonnegative on $(0,1/t]$. Hence on
$(0,1]$ both factors are nonnegative and nonincreasing, so $\Omega(t,\cdot)$ is nonincreasing there. As
$\Omega(t,1)=0$ we get $\Omega(t,B)\ge0$ for $B\in(0,1]$. 
Now let $B\ge1$. As long as $g(t,B)>0$ the same
monotonicity gives $\Omega(t,B)\le\Omega(t,1)=0$. Since
$g(t,1/t)=(p-1)\big(1-t^{1-m}\big)\le0$ because $m\ge1$ and $t\le1$, we have $g(t,\cdot)\le0$ on
$[1/t,\infty)$, where $|1-tB|^{p-2}g(t,B)\le0\le r(1-t)^{p-1}$. So $\Omega(t,B)\le0$ for $B\ge1$.

Consequently, $\partial_{A}\Theta(t,A)\ge0$ for $A\ge1$ and $\le0$ for $0<A\le1$, so $A=1$ minimizes
$\Theta(t,\cdot)$, with value $0$ by \eqref{eq:doublevanish}. The minimum is strict for $A\neq1$. For
$t<1$ because $\Omega(t,\cdot)$ is then strictly decreasing, and for $t=1$ because
$\Theta(1,A)=\Jp(A-1)(A^{m}-1)>0$. Also $\Theta(t,0)=t^{p}+t(1-t)^{p-1}>0$.

\medskip
\noi \textbf{Step 2:} We show that the functions $F^{\pm}$ defined below are positive on $[0,\infty)$.
We put, for $\lambda\ge0$,
\[
  F^{+}(\lambda):=\Jp(\lambda-1)(\lambda-m)-\lambda^{p-1}(\lambda-r),\quad
  F^{-}(\lambda):=(1+\lambda)^{p-1}(m+\lambda)-\lambda^{p-1}(r+\lambda).
\]
Indeed $F^{\pm}(0)=m>0$, and for $\lambda>0$ we have
$F^{\pm}(\lambda)=\lambda^{p}f^{\pm}(1/\lambda)$ with
\[
  f^{+}(s)=\Jp(1-s)(1-ms)-(1-rs),\quad  f^{-}(s)=(1+s)^{p-1}(ms+1)-rs-1,
\]
using $r=m+p-1$. Both satisfy $f^{\pm}(0)=(f^{\pm})'(0)=0$. For $f^{-}$,
\[
  (f^{-})''(s)=(p-1)(p-2)(1+s)^{p-3}(ms+1)+2m(p-1)(1+s)^{p-2}>0 ,
\]
so $f^{-}$ is strictly convex and therefore positive on $(0,\infty)$. Indeed, since $(f^{-})'(s)$ is increasing and $(f^{-})'(0) =0$, we have $(f^{-})'(s) >0$, which again implies $f^{-}$ is strictly increasing. Since $f^{-}(0)=0$, we must have $f^{-}>0$ on $(0,\infty)$. For $f^{+}$ we split $[0, \infty)$ into three sub-intervals. On $(0,\frac{1}{m})$, we have $1-s>0$ and $1-ms>0$, and
$(f^{+})''(s)=(p-1)(p-2)(1-s)^{p-3}(1-ms)+2m(p-1)(1-s)^{p-2}>0$, so $f^{+}>0$ on $(0,\frac{1}{m})$. On
$[\frac{1}{m},1]$ write $f^{+}(s)=(rs-1)-(1-s)^{p-1}(ms-1)$. Here $0\le ms-1\le rs-1$ and
$0\le(1-s)^{p-1}<1$, so $f^{+}>0$ on $[\frac{1}{m},1]$. On $(1,\infty)$,
$f^{+}(s)=(s-1)^{p-1}(ms-1)+(rs-1)>0$. This proves Step 2.

\medskip
\noi \textbf{Step 3:} This step finally proves \eqref{eq:Qinf}. Suppose $Q(t_{n},A_{n})\to0$ and pass to a subsequence with $t_{n}\to t_{0}\in[0,1]$ and $A_{n}\to A_{0}\in[0,\infty]$. Now we consider three different cases based on the range of $t_0$.

\vspace{1.0mm}
\noi \emph{\textbf{Case 1}.} In this case, we assume $0<t_{0}<1$. If $A_{0}\in(0,\infty)\setminus\{1\}$ then $\Theta(t_n,A_n) \to\Theta(t_{0},A_{0})>0$
by Step 1 while the denominator tends to a finite positive limit. If $A_{0}=0$ then
$Q(t_n,0) \to t_{0}^{\frac{p}{2}}+t_{0}^{1-\frac{p}{2}}(1-t_{0})^{p-1}>0$. If $A_{0}=\infty$ then 
\begin{align*}
    \frac{\Theta(t_n,A_n)}{A_n^{r}} = \frac{\Jp(A_n-t_n)\big(A_n^{m}-t_n\big)-(1-t_n)^{p-1}\big(A_n^{r}-t_n\big)}{A_n^{r}} \longrightarrow 1-(1-t_0)^{p-1},
\end{align*}
and hence $Q(t_n,A_n) \to \big[1-(1-t_{0})^{p-1}\big]t_{0}^{-\frac{p}{2}}>0$. If $A_{0}=1$ then second-order Taylor expansion of $\Theta(t_n,\cdot)$ gives
\begin{align}\label{taylor-case1-1}
    \Theta(t_n, A) \approx \Theta(t_n, 1) + \partial_A\Theta(t_n, A)\Bigr|_{A=1} (A - 1) + \frac{1}{2}\partial_A^2\Theta(t_n, A)\Bigr|_{A=1} (A - 1)^2,
\end{align}
where, using \eqref{eq:doublevanish}, $\Theta(t_n,1)=0$ and $\partial_{A}\Theta(t_n,1)=0$. Therefore,  
$$\Theta(t_n,A) \approx \frac12\partial^{2}_{A}\Theta(t_n,A)\Bigr|_{A=1}(A-1)^{2}$$ with $\partial^{2}_{A}\Theta(t_n,1)\to
\partial^{2}_{A}\Theta(t_{0},1)>0$. The denominator is $\asymp|A_n-1|^{p}$, since as $A_n \to 1$, the first-order Taylor approximation of the inner function gives $A_n^\alpha - 1 \approx \alpha(A_n-1)$. So
$Q(t_n,A_n)\asymp|A_n-1|^{2-p}\to\infty$ as $A_n \ra 1$, because $p>2$. Therefore, each alternative contradicts $Q(t_n,A_n)\to0$.

\vspace{1.0mm}
\noi\emph{\textbf{Case 2}.} In this case, we assume $t_{0}=1$. For $A_{0}\neq1$ argue as above. For $A_{0}=1$ set $\delta_n:=1-t_n$, $h_n:=A_n-1$ and
$\lambda_n:=\frac{\delta_n}{|h_n|}$, and fix $\Lambda>1$ large enough that $1- \frac1\Lambda\ge\tfrac12$. If $\lambda_n \le \Lambda$, write $h_n=\pm\frac{1}{n}$ and
$\delta_n= \frac{\lambda_n}{n}$. 

We claim the following relation:
\begin{align}\label{relation-1}
   n^p\Theta(1-\lambda_n\varepsilon,1\pm\frac1n) = F^{\mp}(\lambda_n)+o_{n}(1). 
\end{align}
We only consider $A_n = 1+\frac1n$ and $t_n=1-\frac{\lambda_n}{n}$. The remaining case will follow similarly. We note $J_p(A_n-t_n) = (\frac1n(1 + \lambda_n))^{p-1}$, $A_n^m - t_n = (1 + \frac{m}{n}) - (1 - \frac{\lambda_n}{n}) + o_n(1) = \frac1n(m + \lambda_n) + o_n(1)$, $(1 - t_n)^{p-1} = (\frac{\lambda_n}{n})^{p-1}$, and $A_n^r - t_n = (1 + \frac{r}{n}) - (1 - \frac{\lambda_n}{n}) + o_n(1) = \frac{1}{n}(r + \lambda_n) + o_n(1)$. Hence from the definition \eqref{eq:Theta},
\begin{align*}
    \Theta(1-\frac{\lambda_n}{n},1+\frac{1}{n}) & = \left(\frac{1}{n}\right)^{p-1}(1 + \lambda_n)^{p-1} \left[\frac{1}{n}(m + \lambda_n) + o_n(1)\right] - \left(\frac{\lambda_n}{n}\right)^{p-1} \left[\frac{1}{n}(r + \lambda_n) + o_n(1) \right] \\
    & = \left(\frac{1}{n}\right)^{p} \left[ (1 + \lambda_n)^{p-1}(m + \lambda_n) - \lambda_n^{p-1}(r + \lambda_n) \right] + o_n(1) \\
    &= n^{-p} F^{-}(\lambda_n)+o_n(1).
\end{align*}
Thus, \eqref{relation-1} holds.

Further, since $A_n^\alpha - 1 = (1 \pm \frac{1}{n})^\alpha - 1 \approx \pm \frac{\alpha}{n}$, we have 
$n^{p}t_n^{\frac{p}{2}}|A_n^{\alpha}-1|^{p}\to\alpha^{p}$, in the compact
interval $[0,\Lambda]$, so $$Q(t_n,A_n) \longrightarrow \alpha^{-p}F^{\mp}(\lambda_0) \ge \alpha^{-p} \min_{\la \in [0,\Lambda]} \left(\min\{F^{+}(\la),F^{-}(\la)\}\right)>0,$$ 
where $\la_n \ra \la_0 \in [0, \Lambda]$, and the positivity holds using Step 2. If $\lambda_n\ge\Lambda$, we use no expansion. In this case, $|h_n|\le \frac{\delta_n}{\Lambda}$, so
$\frac{\delta_n}{2} \le \delta_n(1-\frac1\Lambda)\le\xi-t_n\le\delta_n(1+\frac1\Lambda)\le2\delta_n$ for every $\xi$ between $1$ and $A_n$. Indeed, since $|\xi-1| \le |A_n-1|=|h_n| \le \frac{\delta_n}{\Lambda}$, we get $\delta_n(1-\frac1\Lambda)\le \xi-1 + \delta_n \le \delta_n(1+\frac1\Lambda)$.
In view of \eqref{eq:doublevanish}, $\partial^{2}_{A}\Theta(t_n,\xi) \asymp \delta_n^{p-2}$ for $\xi$
between $1$ and $A_n$. Hence, using \eqref{taylor-case1-1}, $\Theta(t_n,A_n) \ge c\delta_n^{p-2}h_n^{2}$ and since $A_n^\al - 1 \approx \al(A_n - 1) = \al h_n$, we get  $Q(t_n,A_n) \ge c\lambda_n^{p-2}\ge c\Lambda^{p-2}$.
Both alternatives contradict $Q(t_n,A_n)\to0$.

\vspace{1.0mm}
\noi\emph{\textbf{Case 3.}} Finally, we assume $t_{0}=0$. We show 
\begin{align}\label{lb-1}
  \Theta(t_n,A_n)\ge c t_n |A_n^{\alpha}-1|^{p}, \text{ for } t_n \text{ small, }  
\end{align} 
which gives
$Q(t_n,A_n) \ge ct_n^{1-\frac{p}{2}}\to\infty$. We fix $\delta_{0}\in(0,\tfrac12)$.

Let $A_n\ge2$. Since $m\ge1$ and $A_n\ge1$, we have 
\begin{align*}
  (A_n-t_n)^{p-1}(A_n^{m}-t_n)\ge A_n^{r}\left(1-\frac{t_n}{A_n}\right)^{p-1}\left(1-\frac{t_n}{A_n}\right)=A_n^{r}\left(1-\frac{t_n}{A_n}\right)^{p},  
\end{align*}
which implies 
\begin{align*}
 \Theta(t_n,A_n) \ge A_n^{r}\left[\left(1- \frac{t_n}{2}\right)^{p}-(1-t_n)^{p-1}\right]\ge c_{p}t_nA_n^{r},   
\end{align*}
for $t_n$ small, the bracket being
$(\tfrac p2-1)t_n+O(t_n^{2})$, while $|A_n^{\alpha}-1|^{p}\le A_n^{r}$. Hence \eqref{lb-1} holds.

Let $|A-1|\le\delta_{0}$. By \eqref{eq:doublevanish}, $\partial_{t}\Theta(s,1)=\partial_{A}
\partial_{t}\Theta(s,1)=0$, so $\partial_{t}\Theta(s,A)=\tfrac12\partial_{t}\partial^{2}_{A}
\Theta(s,\xi)(A-1)^{2}$. Since $\partial_{t}\partial^{2}_{A}\Theta(0,1)=(p-1)(2r-p)>0$, continuity gives
$\partial_{t}\Theta(s,A)\ge\tfrac14(p-1)(2r-p)(A-1)^{2}$ for $s$ and $\delta_{0}$ small, and integrating
in $t$ yields $\Theta\ge\tfrac14(p-1)(2r-p)\,t(A-1)^{2}\ge ct|A^{\alpha}-1|^{p}$, where the last bound follows from $A^\alpha - 1 \approx \alpha(A-1)$ and since $(A-1)^2$ dominates $\vert{}A-1\vert{}^p$ near $1$, since $p>2$. Thus, when $|A_n-1|\le\delta_{0}$, \eqref{lb-1} holds. 

Now, we consider $A_n\in[0,2]$ with $|A_n-1|\ge\delta_{0}$. Here $\Theta(0,A_n)=0$ and
$\partial_{t}\Theta(0,A_n)=D(A_n):=1-A_n^{p-1}+(p-1)A_n^{r-1}(A_n-1)$. By the mean value theorem, for $A_n \in (0,1)$,
$1-A_n^{p-1}>(p-1)A_n^{p-2}(1-A_n)>(p-1)A_n^{r-1}(1-A_n)$ since $r-1>p-2>0$, and symmetrically for $A_n>1$. Hence
$D(A_n)>0$ for $A_n\neq1$, and $D\ge c_{\delta_{0}}>0$ on the compact set considered, since $A_n$ is bounded away from $1$. The partial derivative $\partial_t\Theta$ contains a factor involving $\abs{A_n - t_n}^{p-2}$. Because $p>2$ the
exponent in $|A_n-t_n|^{p-2}$ is positive, so $\partial_{t}\Theta$ is continuous on $[0,t_{1}]\times[0,2]$,
hence uniformly continuous. Therefore $\partial_{t}\Theta(s,A_n)\ge\tfrac12D(A_n)$ for $s$ small, and
integrating gives $$\Theta(t_n,A_n) \ge\tfrac12t_nD(A_n)\ge ct_n \ge \tilde{c}t_n|A_n^{\alpha}-1|^{p}.$$

All cases contradict $Q(t_{n},A_{n})\to0$. So \eqref{eq:Qinf} holds, and this proves the lemma.
\end{proof}

The inequality \eqref{eq:purepower} loses a second form of coercivity, which is present when $X$ and $Y$
are close but $a$ and $b$ are not. The following proposition recovers it. We write $H(X):=X^{\frac{r}{p}}$.

\begin{proposition}\label{prop:strong}
Let $p>2$, $r\ge p$ and $m=r-p+1$. There is $c>0$ such that, for all $a,b>0$ and
$X,Y\ge0$,
\begin{align}\label{eq:strong}
   &\Jp(aX-bY)\big(aX^{m}-bY^{m}\big)-\Jp(a-b)\big(aX^{r}-bY^{r}\big)\no\\
   &\ge cab|a-b|^{p-2}\big(H(X)+H(Y)\big)^{p-2}\big|H(X)-H(Y)\big|^{2}.
\end{align}
\end{proposition}

\begin{proof}
Following the approach used at the beginning of the proof of Lemma~\ref{lem:purepower}, we may assume
$a\ge b$ and, by continuity, $X,Y>0$, and with $t=\frac{b}{a}$, $A=X/Y$ and $\alpha=\frac{r}{p}$ the two sides of
\eqref{eq:strong} equal $a^{p}Y^{r}\Theta(t,A)$ and $c\,a^{p}Y^{r}D_{2}(t,A)$, where
\begin{equation}\label{eq:D2}
  D_{2}(t,A):=t(1-t)^{p-2}\big(A^{\alpha}+1\big)^{p-2}\big(A^{\alpha}-1\big)^{2} .
\end{equation}
For $t=1$, the right-hand side vanishes and the inequality is Step 1 in that proof, so we only need to
show
\begin{equation}\label{eq:Sinf}
  \inf\left\{\frac{\Theta(t,A)}{D_{2}(t,A)}: 0<t<1, A\ge0, A\neq1\right\}>0 .
\end{equation}
We use Step 1 and Step 2 in the proof of Lemma~\ref{lem:purepower}.

Suppose now that \eqref{eq:Sinf} fails, so that $\frac{\Theta(t_{n},A_{n})}{D_{2}(t_{n},A_{n})}\to0$
for some sequence. So pass to a subsequence with $t_{n}\to t_{0}\in[0,1]$ and $A_{n}\to A_{0}\in[0,\infty]$.

\vspace{1.0 mm}
\noi\emph{\textbf{Case $0<t_{0}<1$.}} If $A_{0}\in(0,\infty)\setminus\{1\}$ then $\Theta(t_n,A_n) \to\Theta(t_{0},A_{0})>0$
by Claim 1 while $D_{2}$ tends to a finite positive limit. If $A_{0}=0$ then
$\Theta(t_n,A_n) \to t_{0}^{p}+t_{0}(1-t_{0})^{p-1}>0$ and $D_{2}\to t_{0}(1-t_{0})^{p-2}>0$. If
$A_{0}=\infty$ then $\Theta(t_n,A_n)/A_n^{r}\to1-(1-t_{0})^{p-1}>0$, while
$(A_n^{\alpha}+1)^{p-2}(A_n^{\alpha}-1)^{2}=A_n^{r}(1+o_n(1))$ gives
$D_{2}(t_n,A_n)/A_n^{r}\to t_{0}(1-t_{0})^{p-2}>0$. If $A_{0}=1$ then, by \eqref{eq:doublevanish} and
$(A_n^{\alpha}-1)^{2}=\alpha^{2}(A_n-1)^{2}(1+o_n(1))$,
\[
  \frac{\Theta(t_n,A_n)}{D_{2}(t_n,A_n)}\ \longrightarrow\ \frac{(p-1)(2r-p)}{2^{p-1}\alpha^{2}}>0,
\]
a limit which does not depend on $t_{0}$.

\vspace{1.0 mm}
\noi\emph{\textbf{Case $t_{0}=1$.}} If $A_{0}\neq1$ then $D_{2}(t_n,A_n)\to0$ while $\Theta(t_n,A_n)\to\Theta(1,A_{0})>0$,
so the quotient tends to $+\infty$. If $A_{0}=1$, set $\delta_n=1-t_n$, $h_n=A_n-1$ and $\lambda_n=\frac{\delta_n}{|h_n|}$ as in
the proof of Lemma~\ref{lem:purepower}, and fix $\Lambda\ge2$. Since the corner expansion is uniform
only on compact $\lambda$-intervals, the two regimes $\lambda_n\le\Lambda$ and $\lambda_n\ge\Lambda$ must be treated separately.

If $\lambda_n \le \Lambda$, write $h_n=\pm \frac{1}{n}$ and $\delta_n=\frac{\lambda_n}{n}$. Then, from \eqref{relation-1}, we have $n^p\Theta=F^{\mp}(\lambda_n)+o_n(1)$ and
$n^{p}D_{2}=2^{p-2}\alpha^{2}\lambda_n^{p-2}+o_n(1)$. Since
$\lambda_n^{p-2}\le\Lambda^{p-2}$ on this range, Step 2 in the proof of Lemma~\ref{lem:purepower} gives
\[
  \liminf_{n \ra \infty} \frac{\Theta (t_n,A_n)}{ D_{2}(t_n,A_n) } \ge \frac{\min_{\la \in [0,\Lambda]} \left(\min\{F^{+}(\la),F^{-}(\la)\}\right)}{2^{p-2}\alpha^{2}\Lambda^{p-2}}> 0 .
\]

If $\lambda_n \ge \Lambda$, then exactly as in the proof of
Lemma~\ref{lem:purepower}, $\partial_{A}^{2}\Theta(t_n,\xi) \asymp \delta_n^{p-2}$ for every $\xi$ between $1$ and $A_n$. Therefore, $\Theta(t_n,A_n)\ge c\delta_n^{p-2}h_n^{2}$ by \eqref{eq:doublevanish}.
On the other hand $t\le1$, $(A_n^{\alpha}+1)^{p-2}\le C$ and $(A_n^{\alpha}-1)^{2}\le Ch_n^{2}$, so
$D_{2}(t_n,A_n)\le C\delta_n^{p-2}h_n^{2}$. Hence $\frac{\Theta(t_n,A_n)}{D_{2}(t_n,A_n)} \ge \frac{c}{C}>0$. 

In either regime, the quotient is bounded below by a positive constant.

\vspace{1.0 mm}
\noi\emph{\textbf{Case $t_{0}=0$.}} The three regions treated in the corresponding case of the proof of
Lemma~\ref{lem:purepower} give, for $t_n$ small, $\Theta(t_n,A_n)\ge ct_nA_n^{r}$ when $A_n\ge2$,
$\Theta(t_n,A_n)\ge ct_n(A_n-1)^{2}$ when $|A_n-1|\le\delta_{0}$, and $\Theta\ge ct_n$ when $A_n \in [0,2]$ with
$|A_n-1|\ge\delta_{0}$. Further, since $(1-t_n)^{p-2}\le1$ and 
$(A_n^{\alpha}+1)^{p-2}(A_n^{\alpha}-1)^{2}\le CA_n^{r}$ for $A_n\ge2$, $\le C(A_n-1)^{2}$ for $|A_n-1|\le\delta_{0}$,
and $\le C$ for $A_n\in[0,2]$, in the same three regions $D_{2}\le C t_n A_n^{r}$,
$D_{2}\le Ct_n(A_n-1)^{2}$ and $D_{2}\le Ct_n$. So the quotient is bounded below.

Every case contradicts the assumption, so \eqref{eq:Sinf} holds.
\end{proof}

\noi We use the case $r=p$ below.

\begin{corollary}\label{cor:Q}
Let $p>2$. Then there is $c>0$ such that
\begin{equation}\label{eq:Q}
  \mathfrak{B}(aX,bY,a,b) \ge cab|a-b|^{p-2}(X+Y)^{p-2}|X-Y|^{2}
\end{equation}
for all $a,b>0$ and $X,Y\ge0$, with $\mathfrak{B}$ as in \eqref{eq:bracket}.
\end{corollary}

\begin{proof}
We apply Proposition~\ref{prop:strong} with $r=p$, so that $m=1$, $\alpha=1$ and $H(X)=X$. Its left-hand
side becomes $\Jp(aX-bY)(aX-bY)-\Jp(a-b)(aX^{p}-bY^{p})=\mathfrak{B}(aX,bY,a,b)$.
\end{proof}

Next we extend Lemma~\ref{lem:purepower} to the truncated powers $G_M$.
\begin{proposition}\label{prop:trunc}
Let $p\ge2$, $r>p$ and $m=r-p+1$, and let $G_{M},H_{M}$ be as in \eqref{eq:GH}. We write
$\mathfrak{K}_{M}$ for \eqref{eq:routebracket} with $G=G_{M}$. Then there is $c>0$, independent of
$M$, such that
\begin{equation}\label{eq:truncC}
  \mathfrak{K}_{M}(a,b,X,Y) \ge c(ab)^{\frac{p}{2}}\big|H_{M}(X)-H_{M}(Y)\big|^{p}
\end{equation}
for all $a,b>0$ and $X,Y\ge0$.
\end{proposition}

\begin{proof}
We recall $T_{M}(X)=\min\{X,M\}$, and that both sides of \eqref{eq:truncC} are unchanged under
$(a,X)\leftrightarrow(b,Y)$.

If $p=2$ the bound follows at once from case (b) after \eqref{eq:routeC}, applied to $G=G_{M}$, which is
nondecreasing and locally Lipschitz. With $g:=\int_{0}^{\cdot}(G_{M}')^{\frac12}$ we have
$g'=\sqrt{r-1}\,\tau^{(r-2)/2}$ and $H_{M}'=\tfrac r2\tau^{(r-2)/2}$ on $(0,M)$, and
$g'=H_{M}'=M^{(r-2)/2}$ on $(M,\infty)$, so $g'\ge\min\{2\sqrt{r-1}/r,1\}\,H_{M}'$ and therefore
$|g(X)-g(Y)|\ge c|H_{M}(X)-H_{M}(Y)|$. We assume from now on $p>2$, and split into three regimes.
\begin{enumerate}
\item[{\rm(i)}] If $X,Y\le M$ then $G_{M}(X)=X^{m}$ and $H_{M}(X)=X^{\frac{r}{p}}$, and likewise for $Y$, so the
required bound is exactly \eqref{eq:purepower}.
\item[{\rm(ii)}] If $X,Y\ge M$ then $G_{M}(X)=M^{r-p}X$ and $H_{M}(X)=M^{\frac{r-p}{p}}X$, so
$\mathfrak{K}_{M}=M^{r-p}\,\mathfrak{B}(aX,bY,a,b)$ with $\mathfrak{B}$ as in \eqref{eq:bracket}. By
Lemma~\ref{lem:gspoint}, $\mathfrak{K}_{M}\ge c_{p}M^{r-p}(ab)^{\frac{p}{2}}|X-Y|^{p}
=c_{p}(ab)^{\frac{p}{2}}|H_{M}(X)-H_{M}(Y)|^{p}$, with $c_{p}$ independent of $M$.
\item[{\rm(iii)}] The mixed case $X<M<Y$. Writing $X=M\xi$ and $Y=M\zeta$ with $0\le\xi\le1\le\zeta$
gives $G_{M}(X)=M^{m}\xi^{m}$, $G_{M}(Y)=M^{m}\zeta$, $H_{M}(X)=M^{\frac{r}{p}}\xi^{\frac{r}{p}}$ and
$H_{M}(Y)=M^{\frac{r}{p}}\zeta$, so that every power of $M$ cancels and \eqref{eq:truncC} becomes
\begin{equation}\label{eq:mixC}
  \mathcal{K}\ \ge\ c\,(ab)^{\frac{p}{2}}\big|\xi^{\frac{r}{p}}-\zeta\big|^{p},
  \quad
  \mathcal{K}:=\Jp(a\xi-b\zeta)\big(a\xi^{m}-b\zeta\big)-\Jp(a-b)\big(a\xi^{r}-b\zeta^{p}\big).
\end{equation}
Suppose first $a\ge b$ and set $t:=\frac{b}{a}\in(0,1]$. Dividing by $a^{p}$,
$\mathcal{K}/a^{p}=\Jp(\xi-t\zeta)(\xi^{m}-t\zeta)-(1-t)^{p-1}(\xi^{r}-t\zeta^{p})=:\mathcal{K}_{L}$.
We put $\mathcal{S}_{L}:=|\xi-t\zeta|^{p}-(1-t)^{p-1}(\xi^{p}-t\zeta^{p})$. Since
$|\xi-t\zeta|^{p}=\Jp(\xi-t\zeta)(\xi-t\zeta)$ and $\xi^{r}-\xi^{p}=\xi^{p-1}(\xi^{m}-\xi)$,
\[
  \mathcal{K}_{L}-\mathcal{S}_{L}
  =\big(\xi^{m}-\xi\big)\Big[\Jp(\xi-t\zeta)-(1-t)^{p-1}\xi^{p-1}\Big].
\]
Here $\xi^{m}\le\xi$ because $m\ge1$ and $\xi\le1$, while $\xi-t\zeta\le\xi-t\xi=\xi(1-t)$ because
$\zeta\ge\xi$, so $\Jp(\xi-t\zeta)\le(1-t)^{p-1}\xi^{p-1}$ by monotonicity of $\Jp$. Both factors are
nonpositive, so $\mathcal{K}_{L}\ge\mathcal{S}_{L}$. If instead $b\ge a$, set $t:=a/b$ and divide by
$b^{p}$ to get
$\mathcal{K}/b^{p}=\Jp(\zeta-t\xi)(\zeta-t\xi^{m})-(1-t)^{p-1}(\zeta^{p}-t\xi^{r})=:\mathcal{K}_{H}$. With
$\mathcal{S}_{H}:=|\zeta-t\xi|^{p}-(1-t)^{p-1}(\zeta^{p}-t\xi^{p})$ the same computation gives
\[
  \mathcal{K}_{H}-\mathcal{S}_{H}
  =t\big(\xi-\xi^{m}\big)\Big[\Jp(\zeta-t\xi)-(1-t)^{p-1}\xi^{p-1}\Big]\ \ge\ 0,
\]
because now $\zeta-t\xi\ge\xi(1-t)$. In either orientation
$\mathcal{S}_{L}=\mathfrak{B}(\xi,t\zeta,1,t)$ and $\mathcal{S}_{H}=\mathfrak{B}(\zeta,t\xi,1,t)$ with
$\mathfrak{B}$ as in \eqref{eq:bracket}, so Lemma~\ref{lem:gspoint} gives
$\mathcal{S}\ge c_{p}\,t^{\frac{p}{2}}|\zeta-\xi|^{p}$. Finally, with $\alpha:=\frac{r}{p}>1$, convexity of
$\xi\mapsto\xi^{\alpha}$ gives $\xi^{\alpha}\ge1+\alpha(\xi-1)$, so
$\zeta-\xi^{\alpha}\le(\zeta-1)+\alpha(1-\xi)\le\alpha(\zeta-\xi)$ and therefore
$|\zeta-\xi|^{p}\ge\alpha^{-p}|\zeta-\xi^{\alpha}|^{p}$. Combining, and restoring $a$ or $b$,
gives \eqref{eq:mixC} with $c=c_{p}(p/r)^{p}$.
\end{enumerate}
Taking the smallest of the three constants proves \eqref{eq:truncC}.
\end{proof}

We next prove a quadratic version of Proposition~\ref{prop:trunc}. Here $G_M$, $H_M$ are as in \eqref{eq:GH}.

\begin{proposition}\label{prop:truncquad}
Let $p>2$, $r\ge p$ and $m=r-p+1$, and let $G_{M},H_{M}$ be as in \eqref{eq:GH}. Then there is
$c>0$, independent of $M$, such that
\begin{equation}\label{eq:truncquad}
  \mathfrak{K}_{M}(a,b,X,Y) \ge cab|a-b|^{p-2}
  \big(H_{M}(X)+H_{M}(Y)\big)^{p-2}\big|H_{M}(X)-H_{M}(Y)\big|^{2}
\end{equation}
for all $a,b>0$ and $X,Y\ge0$.
\end{proposition}

\begin{proof}
As before we may assume $a\ge b$, and we treat the same three regimes, writing $\alpha=\frac{r}{p}\ge1$.

\noi{\rm(i)} If $X,Y\le M$ then $G_{M}(X)=X^{m}$ and $H_{M}(X)=X^{\alpha}$, and likewise for $Y$, so
\eqref{eq:truncquad} is \eqref{eq:strong}.

\noi{\rm(ii)} If $X,Y\ge M$ then $G_{M}(X)=M^{r-p}X$ and $H_{M}(X)=M^{\frac{r-p}{p}}X$, so
$\mathfrak{K}_{M}=M^{r-p}\mathfrak{B}(aX,bY,a,b)$, while
\begin{align*}
    \big(H_{M}(X)+H_{M}(Y)\big)^{p-2}\big|H_{M}(X)-H_{M}(Y)\big|^{2}
  &=M^{\frac{(r-p)(p-2)}{p}+\frac{2(r-p)}{p}}(X+Y)^{p-2}|X-Y|^{2} \\
  &=M^{r-p}(X+Y)^{p-2}|X-Y|^{2}, 
\end{align*}
so the two powers of $M$ cancel and the claim is exactly \eqref{eq:Q}.

\noi{\rm(iii)} In the mixed case $X=M\xi$, $Y=M\zeta$ with $0\le\xi\le1\le\zeta$, every power of $M$
cancels as in Proposition~\ref{prop:trunc}-(iii), and the claim becomes
$\mathcal{K}\ge c(ab)\,|a-b|^{p-2}(\xi^{\alpha}+\zeta)^{p-2}|\xi^{\alpha}-\zeta|^{2}$.
The comparison $\mathcal{K}\ge\mathcal{S}$ proved there is unchanged, and
$\mathcal{S}_{L}=\mathfrak{B}(\xi,t\zeta,1,t)$, $\mathcal{S}_{H}=\mathfrak{B}(\zeta,t\xi,1,t)$, so
\eqref{eq:Q} gives in either orientation
$\mathcal{S}\ge c\,t(1-t)^{p-2}(\xi+\zeta)^{p-2}(\zeta-\xi)^{2}$. Finally $\alpha\ge1$ and
$\xi\le1$ give $\xi^{\alpha}\le\xi$, hence $\xi^{\alpha}+\zeta\le\xi+\zeta$, while the convexity bound
$\zeta-\xi^{\alpha}\le\alpha(\zeta-\xi)$ of Proposition~\ref{prop:trunc}-(iii) gives
$(\zeta-\xi^{\alpha})^{2}\le\alpha^{2}(\zeta-\xi)^{2}$. Therefore
\[
  (\xi+\zeta)^{p-2}(\zeta-\xi)^{2} \ge \alpha^{-2}
  \left(\xi^{\alpha}+\zeta\right)^{p-2}\left(\zeta-\xi^{\alpha}\right)^{2},
\]
and \eqref{eq:truncquad} follows in this regime, with $c$ replaced by $c\,(p/r)^{2}$.

Taking the smallest of the three constants proves \eqref{eq:truncquad}.
\end{proof}

We combine the last two propositions.
\begin{corollary}\label{cor:truncboth}
Let $p>2$, $r>p$ and $m=r-p+1$. There is $c_*>0$, independent of $M$, such that
\begin{equation}\label{eq:truncboth}
\begin{aligned}
  \mathfrak{K}_{M}(a,b,X,Y)\ge c_*\Big[&(ab)^{\frac{p}{2}}\big|H_{M}(X)-H_{M}(Y)\big|^{p}\\
  &+ab\,|a-b|^{p-2}\big(H_{M}(X)+H_{M}(Y)\big)^{p-2}\big|H_{M}(X)-H_{M}(Y)\big|^{2}\Big]
\end{aligned}
\end{equation}
for all $a,b>0$ and $X,Y\ge0$.
\end{corollary}

\begin{proof}
Add Propositions~\ref{prop:trunc} and \ref{prop:truncquad} and take for $c_*$ one half of the smaller of their two constants.
\end{proof}

\begin{remark}\label{rem:pis2}
Lemma~\ref{lem:purepower}, Proposition~\ref{prop:strong}, Corollary~\ref{cor:Q}, Proposition~\ref{prop:truncquad}, Corollary~\ref{cor:truncboth} and Proposition~\ref{prop:LC} are stated for $p>2$. For $p=2$ the two terms on the right of \eqref{eq:truncboth} coincide, so Corollary~\ref{cor:truncboth} follows from Proposition~\ref{prop:trunc}, and the proof of Proposition~\ref{prop:LC} works without change, with $|\Jp(q+z)-\Jp(q)|=|z|$.
\end{remark}

\noi \textbf{Acknowledgment.}
Nirjan Biswas acknowledges the support of the National Board for Higher Mathematics Postdoctoral Fellowship (0204/16(9)/2024/RD-II/6761).

\noi \textbf{Data availability.} This research has no associated data.

\noi \textbf{Conflict of interest.} The authors declare that they have no conflicts of interest.

\noi \textbf{AI assistance statement.} The authors used GPT to assist with improving the clarity, grammar, and overall presentation of the manuscript.

\bibliographystyle{abbrvnat}

\end{document}